\documentclass[11pt]{article}
\usepackage{graphicx}
\usepackage{subcaption}
\usepackage{amsmath}
\usepackage{amssymb}
\usepackage{amsthm}
\usepackage{MnSymbol}
\usepackage[dvipsnames]{xcolor}
\usepackage{comment}
\usepackage{bbm}

\newtheorem{theorem}{Theorem}
\newtheorem{remark}{Remark}
\newtheorem{lemma}{Lemma}
\newtheorem{prop}{Proposition}
\usepackage{scalerel,stackengine}
\stackMath
\newcommand\reallywidehat[1]{%
\savestack{\tmpbox}{\stretchto{%
  \scaleto{%
    \scalerel*[\widthof{\ensuremath{#1}}]{\kern.1pt\mathchar"0362\kern.1pt}%
    {\rule{0ex}{\textheight}}
  }{\textheight}%
}{2.4ex}}%
\stackon[-6.9pt]{#1}{\tmpbox}%
}
\numberwithin{equation}{section}
\usepackage{multirow}
\usepackage{pdflscape}
\usepackage{diagbox}

\usepackage{color}
\usepackage{pdfpages}
\usepackage{url}
\usepackage{chngcntr}
\usepackage{ulem}
\usepackage{cancel}
\DeclareMathOperator*{\argmax}{arg\,max}

\newcommand{\angflux}{f}

\newcommand{\angfluxvec}{\mathbf{f}}

\newcommand{\dx}{h}

\newcommand{\x}{\mathbf{x}}

\newcommand{\vel}{\mathbf{v}}

\newcommand{\vweight}{\omega}

\newcommand{\xset}{\Omega_{\x}}

\newcommand{\xboundary}{\partial \xset}

\newcommand{\normal}{\mathbf{n}}

\newcommand{\velset}{\mathbb{S}^{d_{\vel}-1}} 

\newcommand{\vareps}{\varepsilon}

\newcommand{\scat}{\sigma_s}

\newcommand{\absorp}{\sigma_a}

\newcommand{\total}{\sigma_t}

\newcommand{\source}{G}

\newcommand{\bc}{\angflux_{in}} 

\newcommand{\macro}{\rho}

\newcommand{\macrovec}{\boldsymbol{\macro}}

\newcommand{\pd}{K}

\newcommand{\edge}{\mathbf{e}}

\newcommand{\test}{\phi}

\newcommand{\Upwind}{\mathbf{D}}

\newcommand{\scatmat}{\mathbf{\Sigma}_s}

\newcommand{\absorpmat}{\mathbf{\Sigma}_a}

\newcommand{\rbs}{\boldsymbol{\mathcal{F}}_{RB}}

\newcommand{\orthorbsmat}{\boldsymbol{\mathcal{U}}_{RB}}

\newcommand{\reducedcoeff}{\mathbf{c}_{RB}}

\newcommand{\rbdim}{m}

\newcommand{\parameter}{\mu}

\usepackage[ruled,vlined,noresetcount]{algorithm2e}
\usepackage{algpseudocode}
\usepackage{enumitem}

\usepackage[title]{appendix}
\usepackage{mathtools}
\usepackage{chngcntr}
\usepackage{apptools}
\AtAppendix{\counterwithin{lemma}{section}}

\newcommand{\mP}{\mathcal{P}}

\newcommand{\mT}{\mathcal{T}}
\newcommand{\mN}{\mathcal{N}}
\newcommand{\T}{T}
\newcommand{\mE}{\mathcal{E}}

\newcommand{\mM}{\mathcal{M}}
\newcommand{\mG}{\mathcal{G}}
\newcommand{\mathR}{\mathbb{R}}
\newcommand{\bA}{\mathbf{A}}
\newcommand{\bAr}{\hat{\bA}}  
\newcommand{\bArr}{\breve{\bA}}  
\newcommand{\bB}{\mathbf{B}}
\newcommand{\bG}{\mathbf{G}}
\newcommand{\bW}{\mathbf{W}}
\newcommand{\bQ}{\mathbf{Q}}
\newcommand{\bP}{\mathbf{P}}
\newcommand{\bR}{\mathbf{R}}

\newcommand{\bY}{\mathbf{Y}}
\newcommand{\bI}{\mathbf{I}}
\newcommand{\bd}{\mathbf{d}}
\newcommand{\bb}{\mathbf{b}}
\newcommand{\bbr}{\hat{\bb}}  
\newcommand{\bbrr}{\breve{\bb}}  

\newcommand{\rbTrial}{U_{RB}}
\newcommand{\rbsMat}{\boldsymbol{\mathcal{F}}_{RB}}
\newcommand{\dataVec}{\mathbf{b}}

\newcommand{\Uhh}{\hat{U}_h^{\pd}}

\newcommand{\massMat}{\mathbf{M}}

\newcommand{\scatmatPar}{\mathbf{\Sigma}_{s,\parameter}}

\newcommand{\absorpmatPar}{\mathbf{\Sigma}_{a,\parameter}}
 


\usepackage{listings}
\begin{document}
\title{Reduced Basis Methods for Parametric Steady-State Radiative Transfer Equation}
\author{
Kimberly Matsuda\thanks{Department of Mathematical Sciences, Rensselaer Polytechnic Institute, Troy, NY 12180, U.S.A. Email: 
{\tt matsuk2@rpi.edu}. }
\and
Yanlai Chen\thanks{Department of Mathematics, University of Massachusetts Dartmouth, North Dartmouth, MA 02747, USA. Email: {\tt{yanlai.chen@umassd.edu}}. }
\and
Yingda Cheng \thanks{Department of Mathematics, Virginia Tech,
Blacksburg, VA 24061 U.S.A.  Email: 
 {\tt yingda@vt.edu}.  
  }
\and 
Fengyan Li\thanks{Department of Mathematical Sciences, Rensselaer Polytechnic Institute, Troy, NY 12180, U.S.A. Email: 
{\tt lif@rpi.edu}.}
}

\maketitle
\abstract{The radiative transfer equation (RTE) is a fundamental mathematical model to describe physical phenomena involving the propagation of   radiation and its interactions with the host medium, and it arises in many applications. Deterministic methods can produce accurate solutions without any statistical noise, yet often at a price of expensive computational costs originating from the intrinsic high dimensionality of the model. This is more prominent in multi-query tasks, e.g., inverse problems and optimal design, when the RTE needs to be solved repeatedly. This motivates the developments of dimensionality and model order reduction techniques for such transport models.

With this work, we present the first systematic investigation  of  projection-based reduced order models (ROMs) following  the reduced basis method (RBM) framework to simulate the parametric steady-state RTE with isotropic scattering and one energy group. The use of RBM compared to standard proper orthogonal decomposition (POD)  is well motivated, especially considering that a large number of degrees of freedom is needed by full order models to solve high dimensional transport models like RTE.   Four ROMs are designed, with each defining a nested family of reduced surrogate solvers of  different resolution/fidelity. They are based on either a Galerkin or least-squares Petrov-Galerkin projection and utilize either an $L_1$ or residual-based importance/error indicator.  Two of the proposed ROMs are certified in the setting when the absorption cross section is positively bounded below uniformly. One technical focus and contribution lie in the proposed implementation strategies under the affine assumption of the parameter dependence of the model. These well-crafted broadly applicable strategies not only ensure the efficiency and accuracy of the offline training stage and the online prediction of reduced surrogate solvers, they also take into account the conditioning of the reduced systems as well as the stagnation-free residual evaluation for numerical robustness.   Computational complexities are derived for both the offline training and online prediction stages of the proposed model order reduction strategies, and they are demonstrated numerically along with the accuracy and robustness of the reduced surrogate solvers.  Numerically we observe four to six orders of magnitude speedup of our ROMs compared to full order models for some 2D2v examples.
}


\section{Introduction}
\label{sec:intro}

The radiative transfer equation (RTE) is a fundamental mathematical model to describe physical phenomena involving the propagation of   radiation and its interactions with the host medium, and it arises in many areas of applications, such as astrophysics, medical imaging, nuclear engineering, and atmospheric science \cite{duderstadt1979transport,chandrasekhar2013radiative,rybicki2024radiative,stephens1984parameterization,klose1999iterative}. Deterministic methods directly simulate the unknown angular flux (i.e. the probability density function, after normalization) 
and can produce accurate solutions without any statistical noise as in Monte Carlo based simulations. However, they are computationally costly due to the intrinsic high dimensionality of the model, {originating} from unknown solutions being defined on a space-angle-energy phase space, possibly also depending on time. In multi-query tasks such as inverse problems, optimal design, and uncertainty quantification,
the RTE needs to be solved repeatedly, e.g., with different configurations of source, material property, or geometry.   It remains an important subject to design efficient numerical algorithms to provide accurate and robust direct solvers for the RTE and its parametric version.

Recent decades have seen active developments of dimensionality and model order reduction techniques for the RTE to mitigate and address the numerical challenge of high dimensionality.
Physics-based reduced models have a long history, including diffusion approximations found in classical textbooks, and moment methods with various closure strategies \cite{chandrasekhar1946radiative,mcclarren2010robust,huang2022machine,fan2020nonlinear}. Computation-based and data-driven strategies are more  versatile, with examples including  sparse-grid methods \cite{widmer2008sparse, guo2016sparse}, 
fast algorithms based on the integral formulation of the model defined in the physical space \cite{ren2019fast,fan2019fast}, 
tensor-based methods (e.g., proper generalized decomposition  \cite{dominesey2018reduced,dominesey2022reduced}, dynamical low rank approximations \cite{ding2021dynamical,einkemmer2021asymptotic}), 
and projection-based reduced order models.

Projection-based reduced order models (ROMs) \cite{benner2017model} are a general framework for constructing efficient surrogate solvers for (non-)parametric models. 
The framework is based on a reduced basis space (also referred to as reduced trial or reduced approximation space)  that is generated first, preferably of low dimension, to accurately approximate a (parameter-induced) solution manifold. A reduced solution is then sought from this reduced basis space through {a} Galerkin or Petrov-Galerkin projection. The efficiency of this framework is closely related to   the Kolmogorov $n$-width \cite{pinkus2012n}
of the respective solution manifold,  which measures how well the manifold can be approximated by $n$-dimensional linear spaces especially with relatively small $n$.
Proper orthogonal decomposition (POD) \cite{volkwein2011model} 
is one of the most widely used techniques to build reduced basis spaces in projection-based ROMs. Here, a problem-dependent reduced basis is generated by the method of snapshots from a sufficiently large collection of computed or physically measured samples in the solution manifold. 
 In the setting of steady-state RTE,  POD-based ROMs are developed in \cite{tencer2016reduced, tencer2017accelerated}   by treating the angular variable as a parameter in a purely absorbing medium (namely without the integral scattering operator), and in \cite{behne2022minimally, behne2023parametric} in the presence of physical parameters that describe, e.g., material properties and source.
In particular, a minimally-invasive  ROM is proposed in \cite{behne2022minimally}, and it is especially appealing when the operators in full order models (FOMs) are unavailable,  e.g. in matrix-free sweep-based FOMs. This moment-based ROM,
however,  suffers from the lifting bottleneck, in the sense that functions in the reduced basis space need to be lifted back to the full order approximation space of dimension $\mN$. To truly achieve dimensionality reduction, an intrusive angular-form based  ROM is developed in \cite{behne2023parametric}  by utilizing the affine assumption of the parameter dependence. 
A POD-based ROM is also developed in \cite{choi2021space} with an incremental space-time reduced basis construction algorithm for the time-dependent non-parametric RTE.

Reduced basis methods (RBMs) \cite{hesthaven2016certified,haasdonk2017reduced}
are another well-established framework to design projection-based ROMs for parametric models, with the main difference from the POD-based ones in their reduced basis generation. In RBM,   the reduced trial space  is  built greedily, with the reduced basis enriched in each greedy iteration by one snapshot that  corresponds to the high-fidelity full order solution at a parameter value  whose solution is approximated the least accurately,  were the current reduced basis space to be adopted to define the ROM. One signature advantage of the greedy procedure to generate reduced basis  is that only  the minimum number of full order solves is required.  Certified RBMs (e.g., based on some {\it a posteriori} error estimators)  are available for certain classes of parametric models \cite{hesthaven2016certified,haasdonk2017reduced}.   The development of RBM-based ROMs for transport models is rather limited. One line of recent work by three of the authors is to design RBM-based surrogate solvers for the steady-state \cite{peng2022reduced} and time-dependent \cite{peng2024micro} non-parametric RTE by treating the angular variable, possibly also time, as the parameter.

The goal of the present work is to further advance RBMs in simulating parametric transport models efficiently. More specifically, we design and numerically test several projection-based ROMs following the RBM framework to construct reduced surrogate solvers to efficiently simulate the parametric steady-state RTE with isotropic scattering and one energy group. In addition to efficiency and accuracy, great attention is paid to the conditioning of the reduced system matrices and the numerical robustness through the proposed implementation strategies under the affine assumption of the parameter dependence of the model. 
For our development, the high-fidelity FOM is chosen as the upwind discontinuous Galerkin discretization with piecewise $Q^K$ (with $K=1$) polynomial approximations  in space,  combined with the discrete ordinates method \cite{lewis1984computational} (i.e. the $S_N$ method) in angle, with the resulting algebraic system 
solved iteratively via the source iteration scheme accelerated by diffusion synthetic acceleration (SI-DSA, \cite{adams2002fast,peng2022reduced}, also see Appendix \ref{app:impt:cost:FOM}).
Our FOM is known to work uniformly well in a wide range of regimes as the scattering and absorption cross sections vary \cite{adams2001discontinuous,guermond2010asymptotic,sheng2021uniform}.

\subsection{Contributions and organization} Though RBM-based ROMs have been developed in \cite{peng2022reduced, peng2024micro} for the non-parametric RTE with the angular variable (and possibly also time) being treated as the parameter, or have been applied to design ROM-based preconditioners to accelerate full order solvers for the parametric RTE in \cite{peng2025flexible}, our work presents the following unique contributions.

\smallskip
\noindent {\bf The first systematic and multifaceted investigation of RBM-based ROMs for the parametric steady-state RTE.} In total, four ROMs are proposed that are based on either Galerkin projection or a least-squares Petrov-Galerkin projection, and uses either the $L_1$ error indicator \cite{chen2019robust} or the more traditional residual-based error indicator.  Our study covers  comprehensively their formulations, detailed implementation, estimation of computational complexity, the numerical  demonstration and evaluation as well as their comparison.

Several of the proposed ROMs  involve the residual of a given 
function in the full  order approximation space $U_h$, 
either through residual minimization in  \eqref{eq:rom:alg:p:PG} which is equivalent to a least-squares Petrov-Galerkin projection, or through a residual-based error indicator in \eqref{eq:res error indicator} to guide the greedy parameter selection. Unlike some other ROMs involving residuals such as POD-based Petrov-Galerkin ROMs in \cite{carlberg2011efficient,behne2022minimally, tano2021affine}, our ROMs are independent of the specific choice of the basis of the full order approximation space $U_h$ in \eqref{eq:Uh}, {because} the residual is measured as a function, instead of being measured through  their expansion coefficient vectors  with respect to a given  basis (see Remark \ref{rem:res}).

\smallskip
\noindent {\bf Robust, efficient, and stagnation-free algorithmic developments for implementation strategies.} 
Another main contribution of this work lies in the analysis-based implementation strategies in Section \ref{sec:impt:cost}. They are designed not only to fully harvest the {\it efficiency} of the proposed reduced order models, but also to ensure the reduced system matrices are reasonably conditioned while having good numerical {\it robustness} as the reduced dimension grows and the designed resolution/fidelity of the ROM increases.
By utilizing the affine assumption of the parameter dependence of the RTE model, strategies involving {\it offline-online} procedures are designed to implement several algorithmic ingredients. We 
{particularly want to} highlight the QR-based algorithms in Alg. \ref{alg:pLS:offline} - Alg. \ref{alg:ResC:1} (also see Theorems \ref{thm:why:QRp} - \ref{thm:ResC:1}) for efficient  residual minimization and residual evaluation at many instances of the parameter without the error stagnation that is often encountered as the reduced dimension grows (see Figure 2.6 in \cite{haasdonk2017reduced} and the associated discussion). Some variants to the implementation strategies in Appendix \ref{app:variants}  will provide more  insights into what is being proposed.  In contrast, the Petrov-Galerkin-based ROMs in \cite{behne2023parametric, tano2021affine} only focus on {ROM efficiency} by utilizing the affine assumption, without considering the conditioning or robustness of the implementation strategies,  which become important especially for ROMs with higher resolution/fidelity as  the reduced dimension grows.

With the greedy procedure to construct the reduced basis spaces recursively as in Alg. \ref{alg:RBM},   the RBM framework indeed generates a hierarchical family of ROMs of growing dimension and with 
{improved} resolution/fidelity, with the ROM of one dimension lower being used in the offline phase to construct the current ROM. Hence, our algorithmic advancements for implementation greatly improve the efficiency, conditioning, and robustness of the entire ROMs, including the model construction during the offline stage and the prediction and tests of the reduced order solvers during the online stage. It is important to note that these algorithmic developments for implementation are broadly applicable to ROMs for general parametric PDEs.

\smallskip
\noindent {\bf Practical speedup of four to six orders of magnitude.}  Numerically it is observed that the cost of the offline training phase to build one ROM of dimension $\rbdim$ is dominated by that of $\rbdim$ FOM solves,  evidencing the highly efficient offline training.    In  numerical experiments, three of the proposed ROMs (i.e. G-$L_1$, G-Res, PG-Res) consistently demonstrate themselves as efficient and reliable surrogate solvers of low or high resolution/fidelity to  simulate the parametric steady-state RTE  at  any parameter value, with a significant saving in computational cost over the FOM, e.g., by a speedup factor of $10^4-10^6$ in some 2D2v examples when FOM discretizations are of moderate size and when the reduced dimension is $\rbdim=15$. Such saving will be even more significant when the full dimension $\mN$ of the FOM is larger.  The fourth ROM, PG-$L_1$, with its efficiency to achieve certain resolution more sensitive to the first parameter selection during the offline greedy iterations,  can benefit from using an enhanced $L_1$ error indicator with extra but affordable offline training cost for an overall more  robust performance at the low or high resolution/fidelity.

The rest of the paper is organized as follows.
In Section \ref{sec:RTE and FOM}, we introduce the RTE as well as the FOM in both its variational form and its algebraic form. 
In Section \ref{sec:RBM}, four RBM-based ROMs are proposed, determined by their main building blocks including two projection-based reduced order solvers  and two error indicators, along with the spectral ratio stopping criterion.  Section \ref{sec:impt:cost} is devoted to the algorithmic developments to ensure efficient and robust implementation of the proposed ROMs. Computational complexities are also derived for the offline training stage to build each ROM and for the online prediction stage to compute reduced solutions. 
In Section \ref{sec:num}, we demonstrate the performance of our methods for their accuracy, efficiency, and robustness through a collection of physically relevant benchmark examples, {including 1D examples in slab geometry and 2D2v examples.}  
Concluding remarks are made in Section \ref{sec:conclusion}. {For easy reading, proofs of any lemma, proposition, and theorem are presented in Appendix, along with some variants of the algorithm and implementation strategies.}

\section{The RTE  and full order model (FOM)}
\label{sec:RTE and FOM}

Consider the steady-state RTE with isotropic scattering and one energy group,
\begin{subequations} \label{eq:kinetic transport equation}
\begin{align}
& \vel \cdot \nabla_{\x} \angflux(\x,\vel) + \total(\x) \angflux(\x,\vel) = \scat(\x) \macro(\x) + \source(\x),\quad (\x,\vel)\in \xset\times\velset,
\\
& \angflux(\x,\vel) = \bc(\x, \vel),
\
(\x, \vel)\in \Gamma_{-}(\xset). 
\end{align}
\end{subequations}
Here
$\angflux = \angflux(\x,\vel)$ is the angular flux at the spatial position $\x \in \xset\in\mathbb{R}^{d_\x}$ with the angular direction $\vel \in \velset$, $\macro=\macro(\x)=\frac{1}{|\velset|} \int_{ \velset} \angflux(\x,\vel) d\vel$ is the scalar flux (also the $0$-th order moment, or the macroscopic density), $\total=\total(\x) = \scat(\x) + \absorp(\x) >0$ is the total cross section, $\scat = \scat(\x) \geq 0$ is the scattering cross section, $\absorp = \absorp(\x) \geq 0$ is the absorption cross section, and $\source = \source(\x)$ is the isotropic source. The boundary condition is given at the inflow boundary $\Gamma_{-}(\xset)=\{(\x, \vel)\in \xboundary\times\velset, \vel \cdot \normal(\x) < 0\}$, with $\normal(\x)$ being the outward unit normal 
at $\x \in \xboundary$. In this paper, we are particularly interested in the parametric version of the RTE,
\begin{subequations} \label{eq:P-RTE}
\begin{align}
& \vel \cdot \nabla_{\x} \angflux(\x,\vel; \parameter) + \total(\x; \parameter) \angflux(\x,\vel; \parameter) = \scat(\x; \parameter) \macro(\x; \parameter) + \source(\x; \parameter),\quad (\x,\vel)\in \xset\times\velset,
\\
& \angflux(\x,\vel; \parameter) = \bc(\x, \vel; \parameter),
\
(\x, \vel)\in \Gamma_{-}(\xset).
\end{align}
\end{subequations}
The parameter $\parameter\in\mP\subset\mathR^{d}$ can be scalar or vector-valued, and it  represents one or more of the material properties (e.g. the scattering or absorption cross sections), source term,  boundary conditions and even the uncertainty of the model. The parameter set $\mP$ is compact in $\mathR^{d}$.

In the  following subsections, we will formulate a high-fidelity  FOM 
for the RTE \eqref{eq:kinetic transport equation} and derive its algebraic form for implementation.  
The FOM solutions will be regarded as the ground truth when we come to the design and the evaluation of reduced order models for the parametric RTE \eqref{eq:P-RTE}. The FOM will also be used to provide accurate snapshot solutions to build (low-dimensional)  reduced basis spaces in our ROMs.

\subsection{Angular discretization}

In the angular direction, we apply the discrete ordinates method \cite{lewis1984computational} (i.e. the $S_N$ method). To this end,
let  $\{\vel_j\}_{j=1}^{N_{\vel}}\subset \velset$ be a set of quadrature points and $\{\vweight_j\}_{j=1}^{N_{\vel}}$ be the 
scaled quadrature weights. We approximate $\angflux(\x, \vel_j)$ by $\angflux_j(\x)$, $j=1,\dots,N_{\vel}$, that satisfy
\begin{subequations} \label{eq:angular discretization}
\begin{align}
& \vel_j \cdot \nabla_{\x} \angflux_j(\x) + \total(\x) \angflux_j(\x) = \scat(\x) \sum_{k=1}^{N_\vel} \vweight_k\angflux_k(\x) + \source(\x), \quad \x\in\xset,
\\
& \angflux_j(\x) = \bc(\x,\vel_j),
\
\x \in \Gamma_{-}^j(\xset)=\{\x\in \xboundary, \vel_j \cdot \normal(\x) < 0\}.
\end{align}
\end{subequations}
When $d_{\vel} = 3$,  the following $N_\vel$-point Chebyshev-Legendre (CL) quadrature 
 on the unit sphere $\velset$ will be used,
\begin{equation}
\vel_j = (\cos(\theta_{k}) \sqrt{1-\xi_{l}^2}, \sin(\theta_{k}) \sqrt{1-\xi_l^2}, \xi_l), 
\quad
\vweight_j = \vweight^{(\theta)}_k \vweight^{(\xi)}_l, 
\quad 
1\leq k \leq N_{\theta},
\quad 
1 \leq l \leq N_{\xi},
\end{equation}
with $j=l+(k-1)N_\xi$, $N_\vel=N_\theta N_\xi$, and some positive even integer $N_\theta$. This quadrature formula, referred to as the $(N_\theta, N_\xi)$-CL quadrature,
 can be  regarded as the tensor product of the $N_\theta$-point Chebyshev quadrature 
  for the unit circle, with $\theta_k=(2k-1)\pi/N_\theta$ and {$\vweight^{(\theta)}_k=1/(2N_\theta)$}, and 
 the $N_\xi$-point Gauss-Legendre quadrature 
  on $[-1,1]$, with $\{\xi_l\}_{l=1}^{N_{\xi}}$ and $\{\vweight^{(\xi)}_l\}_{l=1}^{N_{\xi}}$ as the standard quadrature points and weights, respectively.

\subsection{Spatial discretization and FOM}

In space, we further apply the upwind  discontinuous Galerkin (DG) method \cite{adams2001discontinuous}. Without loss of generality, we present the 
discretization when $d_{\x}=2$. 
Let the spatial domain be $\xset = [x_L,x_R] \times [y_L,y_R]$, partitioned into a Cartesian mesh $\mT_h$, with a typical element $\T$. 
Related, let $\mE_h$ be the collection of the edges of all elements in $\mT_h$. In addition, we define 
\begin{equation}
\mE_{h,j}^{-} = \{\edge\in\mE_h: \edge \subset \xboundary, \vel_j \cdot \normal_\edge < 0\}, \quad \mE_{h,j}^{+} = \{\edge\in\mE_h: \edge \subset \xboundary, \vel_j \cdot \normal_\edge \geq 0\}.
\end{equation}
For an edge $\edge\subset \xboundary$,   $\normal_\edge$ represents the outward unit normal with respect to the domain $\xset$.

Associated with the mesh, we introduce a finite dimensional trial (i.e. approximation)  space
\begin{equation}
\Uhh= \{u(\x): u(\x) |_{\T} \in Q^{\pd}(\T), \; \T\in\mT_h \}, 
\end{equation}
where $Q^{\pd}(\T)$  is the set of polynomials on $\T$ with the degree up to $\pd$ in each spatial direction. 
Note that the functions in $\Uhh$ can be double-valued at interior edges. It is also helpful to define 
\begin{equation}U_h=U_h^{\pd}=\{g_h=(g_{h,1},g_{h,2},\dots,g_{h,N_{\vel}}): g_{h,j}=g_h(\cdot,\vel_j)\in\Uhh\}.
\label{eq:Uh}
\end{equation}
For each interior edge $\edge=\T\cap \T'$, with some $\T, \T'\in \mT_h$, the outward normal of $\T$ along $\edge$ is denoted as $\normal_{\edge, \T}$, and the trace of a function $\phi\in\Uhh$ along $\edge$ from the side of $T$ is denoted as $\phi|_\T$.  
 We also define the inflow boundary for an element $\T\in \mT_h$ with respect to the $j$-th angular direction $\vel_j$ as
\begin{equation}
\partial \T_{j}^{-} = \{\edge\subset \partial \T: \vel_j \cdot \normal_{\edge,\T} < 0\}.
\end{equation}

    We are now ready to write down the FOM after the discrete ordinates method is applied in angle and the upwind DG method is applied in space: we seek
 $\angflux_h=(\angflux_{\dx,1},\angflux_{\dx,2},\dots,\angflux_{\dx,N_{\vel}})\in U_h$,  with each $\angflux_{\dx,j}$ satisfying  
\begin{align} \label{eq:fom}
\sum_{T\in \mT_h}\int_{T} & \big(- (\vel_j \cdot \nabla_\x \test ) \angflux_{h,j} + \total \test \angflux_{h,j} \big) d \x
+
\sum_{\edge \in \mE_{h,j}^{+}}\int_{\edge}
\vel_j\cdot\normal_{\edge} \test\angflux_{h,j}ds\notag\\
&-
\sum_{T\in \mT_h}
\int_{\edge=T\cap T'\subset\big(\partial T_j^-\setminus  \mE_{h,j}^{-}\big)}\vel_j\cdot\normal_{\edge,T}\big(\test|_{T'}-\test|_{T}\big)\angflux_{h,j}|_{T'}ds\notag
\\
&=
\sum_{T\in \mT_h}\int_{T}
\big(\scat  \macro_h +\source\big) \test d \x
-
\sum_{\edge \in \mE_{h,j}^{-}} \int_{\edge} \bc(\cdot,\vel_j) \vel_j \cdot \normal_{\edge} \test ds,\quad \forall \test\in \Uhh,
\end{align}
where  $\macro_h(\x) = \sum_{k=1}^{N_{\vel}} \vweight_k \angflux_{h,k}(\x)$.   Equivalently, the FOM can be rewritten more compactly as follows: look for $\angflux_h\in U_h$ satisfying
\begin{equation}
a_h(\angflux_h, \test_h)=l_h(\test_h),\quad \test_h\in U_h,
\label{eq:fom:c}
\end{equation}
where the bilinear form $a_h(\cdot,\cdot)$ on $U_h\times  U_h$ and the linear form $l_h(\cdot)$ on $U_h$ are 
\begin{align}
a_h(\angflux_h, \test_h)=&
\sum_{j=1}^{N_{\vel}}\vweight_j\left\{\sum_{T\in \mT_h}\int_{T} \big(- (\vel_j \cdot \nabla_\x \test_{h,j} ) \angflux_{h,j} + \total \test_{h,j} \angflux_{h,j} \big) d \x+
\sum_{\edge \in \mE_{h,j}^{+}}\int_{\edge}
\vel_j\cdot\normal_{\edge} \test_{h,j}\angflux_{h,j}ds
\right.\notag\\
-&\left.
\sum_{T\in \mT_h}
\int_{\edge=T\cap T'\subset\big(\partial T_j^-\setminus  \mE_{h,j}^{-}\big)}\vel_j\cdot\normal_{\edge,T}\big(\test_{h,j}|_{T'}-\test_{h,j}|_{T}\big)\angflux_{h,j}|_{T'}ds
-\sum_{T\in \mT_h}\int_{T}
\scat  \macro_h  \test_{h,j} d \x \right\},\notag\\
l_h(\test_h)=&\sum_{j=1}^{N_{\vel}}\vweight_j\left\{\sum_{T\in \mT_h}\int_{T}
\source \test_{h,j} d \x
-
\sum_{\edge \in \mE_{h,j}^{-}} \int_{\edge} \bc(\cdot,\vel_j) \vel_j \cdot \normal_{\edge} \test_{h,j} ds\right\}.
\end{align}
When the FOM is applied to the parametric RTE \eqref{eq:P-RTE}, the numerical solution will be denoted as $\angflux_h(\cdot; \parameter)=\textrm{FOM}(\parameter)$.

\subsection{FOM: the algebraic form}
\label{sec:FOM:alg}

In this subsection, we will convert the FOM to its algebraic form, which will be solved iteratively by the source iteration scheme accelerated by diffusion synthetic acceleration (SI-DSA, \cite{adams2002fast,peng2022reduced}, also see Appendix \ref{app:impt:cost:FOM}).  Let $\{\phi_k(\x)\}_{k=1}^{N_\x}$ be a basis of $\Uhh$. For any function $g_h\in U_h$, we have $g_{h,j}=[\phi_1,\phi_2,\dots,\phi_{N_\x}]\mathbf{g}_j$,
with $\mathbf{g}_j\in\mathR^{N_\x}$
as the vector of the expansion coefficients. We further write $\mathbf{g}=[\mathbf{g}_1^T,\mathbf{g}_2^T,\dots,\mathbf{g}_{N_\vel}^T]^T$, referred to as the {\it coordinate vector} of $g_h$.  The matrix-vector form of the FOM \eqref{eq:fom} can be derived as
\begin{equation} \label{eq:fom:alg:j}
(\Upwind_j  + \hat{\mathbf{\Sigma}}_a) \angfluxvec_j  = \hat{\mathbf{\Sigma}}_s \big(\macrovec-\angfluxvec_j) + \dataVec_j,
\quad
\macrovec = \sum_{k=1}^{N_{\vel}} \vweight_k \angfluxvec_k,\quad
j=1,\cdots, N_{\vel},
\end{equation}
where $\Upwind_j, \hat{\mathbf{\Sigma}}_a, \hat{\mathbf{\Sigma}}_s 
\in \mathbb{R}^{N_{\x} \times N_{\x}}$, $\dataVec_j \in \mathbb{R}^{N_{\x}}$ are defined as
\begin{subequations}
\begin{align}
(\Upwind_j)_{kl}
 =&
-\sum_{T\in \mT_h}\int_{T} \left(\vel_j \cdot \nabla_\x \phi_k\right)  \phi_l   d \x
-
\sum_{T\in \mT_h}
\int_{\edge=T\cap T'\subset\big(\partial T_j^-\setminus  \mE_{h,j}^{-}\big)}\vel_j\cdot\normal_{\edge,T}\big(\phi_k|_{T'}-\phi_k|_{T}\big)\phi_l|_{T'}ds,\notag\\
&+\sum_{\edge \in \mE_{h,j}^{+}}\int_{\edge}
\vel_j\cdot\normal_{\edge} \phi_k\phi_l ds\\
(\hat{\mathbf{\Sigma}}_a)_{kl} 
 = &
\sum_{T\in\mT_h}\int_T \absorp(\x) \phi_k(\x) \phi_l(\x) d \x, \quad 
(\hat{\mathbf{\Sigma}}_s)_{kl} 
 = 
\sum_{T\in\mT_h}\int_T \scat(\x) \phi_k(\x) \phi_l(\x) d \x,  
\\
(\dataVec_j)_{k}
 = &
\sum_{T\in \mT_h}\int_{T}\source(\x) \phi_k(\x) d \x
-
\sum_{\edge \in \mE_{h,j}^{-}} \int_{\edge} \bc(\cdot,\vel_j) \vel_j \cdot \normal_{\edge} \phi_k ds.
\end{align}
\end{subequations}
For the convenience of later reference in the parametric setting, the contributions from the scattering and absorption processes are separated.   Given the local nature of $\Uhh$, we assume 
{that} each basis function $\phi_k(\x)$ is local, with its support being one mesh element $T\in\mT_h$. This will result in block-diagonal $\hat{\mathbf{\Sigma}}_a, \hat{\mathbf{\Sigma}}_s$ and a sparse 
$\Upwind_j$. 
If we further introduce $\dataVec=[\dataVec_1^T,\dataVec_2^T,\dots,\dataVec_{N_{\vel}}^T]^T$, \eqref{eq:fom:alg:j} will be written more compactly as
\begin{equation} \label{eq:fom:alg}
\mathbf{A} \angfluxvec = \dataVec,
\end{equation}
where $\bA\in\mathR^{\mN\times\mN}$ and $\dataVec\in \mathR^{\mN}$ with $\mN=N_{\x} N_{\vel}$. The system matrix $\bA$ is defined as
\begin{equation}
\label{eq:sys:mat}
\bA=\Upwind+\absorpmat+\scatmat,
    \end{equation}
 with $\Upwind=\text{diag}(\Upwind_1,\dots,\Upwind_{N_{\vel}})$, $\absorpmat=\text{diag}(\hat{\mathbf{\Sigma}}_a,\dots,\hat{\mathbf{\Sigma}}_a)$, and  
 \begin{equation}
     \scatmat=\big(\bI_{N_\vel}-[1,1,\dots,1]^T[\vweight_1,\dots,\vweight_{N_{\vel}}]\big)\otimes \hat{\mathbf{\Sigma}}_s.
     \label{eq:sys:scat}
 \end{equation} 
 Specifically, we have
\begin{equation}
\scatmat=\begin{bmatrix}
(1-\vweight_1)\hat{\mathbf{\Sigma}}_s   & -\vweight_2 \hat{\mathbf{\Sigma}}_s & \dots & -\vweight_{N_{\vel}} \hat{\mathbf{\Sigma}}_s \\
-\vweight_1 \hat{\mathbf{\Sigma}}_s &  (1-\vweight_2)\hat{\mathbf{\Sigma}}_s & \dots & -\vweight_{N_{\vel}} \hat{\mathbf{\Sigma}}_s \\
\vdots & \vdots & \vdots & \vdots \\
-\vweight_1 \hat{\mathbf{\Sigma}}_s & -\vweight_2 \hat{\mathbf{\Sigma}}_s & \dots & (1-\vweight_{N_{\vel}})\hat{\mathbf{\Sigma}}_s
\end{bmatrix}.
\end{equation}
 Here, $\bI_{N_\vel}$ is the $N_{\vel}\times N_{\vel}$ identity matrix and $\otimes$ is the Kronecker product.  The linear  system \eqref{eq:fom:alg} is essentially the algebraic representation of \eqref{eq:fom:c}.

For the discrete space $U_h$, we define  $\|g_h\|_h=\left(\sum_{j=1}^{N_\vel}{\vweight_j} \|g_{h,j}\|^2_{L_2(\xset)}\right)^{1/2}$, a ``semi-discrete'' $L_2$ norm, and the associated inner product is written  as $(\cdot, \cdot)_h$.  
This is a discrete analog of the scaled $L_2$ norm of $g\in L_2(\xset\times\velset)$, namely,  $|\velset|^{-1/2} \left(\int_{ \xset\times\velset} |g(\x,\vel)|^2d\x d\vel\right)^{1/2}$. 
For a given function $g_h\in U_h$, we further define its residual $r_h(g_h)\in U_h$ as follows,
\begin{equation}
    (r_h(g_h), \hat{g}_h)_h:=a_h(g_h, \hat{g}_h)-l_h(\hat{g}_h),\quad \forall \hat{g}_h\in U_h.
    \label{eq:res:gh}
\end{equation}
Note that $a_h(g_h, \cdot)-l_h(\cdot)$ is a linear {bounded} functional on $U_h$, measuring how much  $g_h$  does not satisfy the FOM, and  $r_h(g_h)$ is the  Riesz representation of this functional with respect to $(\cdot, \cdot)_h$.

\begin{lemma} Given $g_h\in U_h$,  the following algebraic representations hold for its $L_2$ norm $\|g_h\|_h$ and the $L_2$ norm of its residual $r_h(g_h)\in U_h$,
\begin{subequations}
\begin{align}
\|g_h\|^2_h&=
\mathbf{g}^T\textrm{diag}(\vweight_1\massMat,\vweight_2\massMat,\dots,\vweight_{N_{\vel}}\massMat)\mathbf{g},\label{lem:norm:alg.a}\\
\|r_h(g_h)\|_h^2&=(\bA\mathbf{g}-\dataVec)^T\textrm{diag}(\vweight_1\massMat^{-1},\vweight_2\massMat^{-1},\dots,\vweight_{N_{\vel}}\massMat^{-1})(\bA\mathbf{g}-\dataVec).
\label{lem:norm:alg.b}
\end{align}
\end{subequations}
Here $\massMat\in\mathR^{N_{\x}\times N_{\x}}$ is the mass matrix, and it is symmetric positive definite, with  $\int_{\xset}\phi_l\phi_k d\x$ as its  $(k,l)$ entry.
\label{lem:norm:alg}
\end{lemma}

{The proof of Lemma \ref{lem:norm:alg} is given in Appendix \ref{ap:lem1}. This lemma shows} that $\|g_h\|_h$ is a weighted norm of the coordinate vector $\mathbf{g}$, not simply $\|\mathbf{g}\|_{\ell_2}$. Similarly, $\|r_h(g_h)\|_h$ is a weighted norm of the algebraic residual $\bA\mathbf{g}-\dataVec$, not simply $\|\bA\mathbf{g}-\dataVec\|_{\ell_2}$.  The weighting matrices are related to the spatial and temporal discretizations of FOM, and particularly they depend on the specific choice of the basis being used for $U_h$. It is important to note that as a measurement for the residual of $g_h$, $\|r_h(g_h)\|_h$ is basis-independent, while   $\|\bA\mathbf{g}-\dataVec\|_{\ell_2}$ depends on the specific choice of basis of $U_h$.

\section{Proposed RBM-based reduced order models}
\label{sec:RBM}

In this section, we propose four \underline{reduced order models}\footnote{When needed, we distinguish the usages of the terms  {\it reduced order models} and {\it reduced order solvers}, with the former also including the reduced basis generation 
{during the} offline stage, while the latter specifically referring to 
{the} actual reduced  surrogate solvers, i.e., ROM($\rbTrial^{\rbdim};\parameter$).}, denoted as {G-$L_1$, G-Res, PG-$L_1$, PG-Res},  for efficiently simulating the parametric RTE  \eqref{eq:P-RTE} by following the  RBM framework. The standard RBM consists of an offline and online stage. In the offline training stage, a (low-dimensional) reduced basis space, also referred to as reduced trial or approximation space,   $$ \rbTrial^M=\textrm{span}\{\angflux_h(\cdot;\parameter_1), \angflux_h(\cdot;\parameter_2), \dots, \angflux_h(\cdot;\parameter_M)\}$$   is built/trained with $M\ll \mN$, where $\angflux_h(\cdot;\parameter_j)=\textrm{FOM}(\parameter_j)$ and the parameter values $\{\parameter_j\}_{j=1}^M$ are judiciously chosen via  a greedy procedure and guided by some rigorous {\it a posteriori} error estimator or efficient error indicator. 
In the online prediction stage, an approximate solution  at a given parameter value $\parameter$
 will be sought from $\rbTrial^M$ based on a \underline{reduced order solver}, ROM($\rbTrial^M;\parameter$). This reduced order solver  ROM($\rbTrial^{\rbdim};\parameter$), with $\rbTrial^{\rbdim}=\textrm{span}\{\angflux_h(\cdot;\parameter_1), \angflux_h(\cdot;\parameter_2), \dots, \angflux_h(\cdot;\parameter_{\rbdim})\}$, ${\rbdim}\leq M$, and the respective reduced solution  $\angflux_{RB}^{\rbdim}(\cdot;\parameter)$, is also used during the offline greedy iterations to build $\rbTrial^M$. Next we will start with the formulations of two reduced order solvers for the RTE model. $\rbdim$ will be referred to as the reduced dimension whose terminal value is $M$.

\subsection{Reduced order solver: ROM($\rbTrial^{\rbdim};\parameter$)}
\label{sec:ROM}

The reduced order solver ROM($\rbTrial^{\rbdim};\parameter$) will be presented in both the variational form and its algebraic form, with the latter directly related to numerical implementation. By explicitly including the dependence on the parameter, the FOM, with its variational formulation \eqref{eq:fom:c} and its algebraic formulation \eqref{eq:fom:alg}, is given as follows. 

\medskip
\noindent
\underline{FOM in its variational form:} seek  $\angflux_h(\cdot;\parameter)\in U_h$, satisfying 
\begin{equation}
a_{h,\parameter}(\angflux_h(\cdot;\parameter), \test_h)=l_{h,\parameter}(\test_h),\quad \test_h\in U_h.
\label{eq:fom:p}
\end{equation}
Related, the residual of $g_h\in U_h$, denoted as $r_{h,\parameter}(g_h)\in U_h$,  is defined as
\begin{equation}
    (r_{h,\parameter}(g_h), \hat{g}_h)_h:=a_{h,\parameter}(g_h, \hat{g}_h)-l_{h,\parameter}(\hat{g}_h),\quad \forall \hat{g}_h\in U_h.
    \label{eq:res:gh:par}
\end{equation}

\medskip
\noindent
\underline{FOM in its algebraic form:}
\begin{equation}\mathbf{A}_{\parameter} \angfluxvec_{\parameter}= \dataVec_{\parameter}, \quad \text{with}\;\; \mathbf{A}_{\parameter}\in\mathR^{\mN\times\mN},\;\; \dataVec_{\parameter}\in\mathR^{\mN}.
\label{eq:fom:alg:p}
\end{equation}
In this work, we consider two reduced order solvers: one is based on Galerkin projection, and the other is based on a least-squares Petrov-Galerkin projection. For their algebraic formulations, we first introduce the reduced matrix associated with the reduced basis space $\rbTrial^{\rbdim}$: 
\begin{equation}
    \rbsMat^{\rbdim} = [\angfluxvec_{\parameter_1}, \angfluxvec_{\parameter_2}, \dots, \angfluxvec_{\parameter_{\rbdim}}]\in \mathR^{\mN\times {\rbdim}}
\end{equation} 
whose column $\angfluxvec_{\parameter_j}$ is the coordinate vector of  
$\angflux_h(\cdot; \parameter_j)=\textrm{FOM}(\parameter_j)$. 
With the stopping criterion adopted during the greedy procedure, as discussed  in  Section \ref{sec:RBM2},   we assume $\rbsMat^{\rbdim}$ has full (column) rank, and $\rbdim<\mN$. 

To improve the ROM's conditioning, guided by \cite{haasdonk2017reduced} we will work with an orthonormal basis $\orthorbsmat^{\rbdim}$ (with respect to the standard $\ell_2$ inner product in $\mathR^{\mN}$) for the column space of $\rbsMat^{\rbdim}$. This will be obtained through the QR factorization $\rbsMat^{\rbdim}=\orthorbsmat^{\rbdim} \mathbf{R}_{RB}^{\rbdim}$, where   $(\orthorbsmat^{\rbdim})^T\orthorbsmat^{\rbdim}={\bI}_{{\rbdim}\times {\rbdim}}$ and $\mathbf{R}_{RB}^{\rbdim}\in\mathR^{\rbdim\times\rbdim}$ is upper-triangular. In our implementation, the QR factorization is computed {\it incrementally} via the Classical Gram-Schmidt with Reorthogonalization (CGSR) algorithm \cite{daniel1976reorthogonalization}, {resulting in the following nested structure when $\rbdim>1$:
\begin{equation}\orthorbsmat^{\rbdim}=[\orthorbsmat^{\rbdim-1}, \star], \quad \mathbf{R}_{RB}^{\rbdim}=\begin{bmatrix}
    \mathbf{R}_{RB}^{\rbdim-1} &\star\\
    0 &\star
    \end{bmatrix},
    \label{eq:CGSR}
    \end{equation}
    originated from $\rbsMat^{\rbdim}=[\rbsMat^{\rbdim-1}, \angfluxvec_{\parameter_{\rbdim}}]$.
    In particular, computing the last columns of $\orthorbsmat^{\rbdim}$ and $\mathbf{R}_{RB}^{\rbdim}$ is at a cost of $O(\mN)$.
It is shown numerically that the CGSR algorithm produces good quality orthonormal columns for $\orthorbsmat^{\rbdim}$. 
}

\subsubsection{ROM based on  Galerkin projection}
\label{sec:ROM-G}

Based on Galerkin projection, our first reduced order solver ROM($\rbTrial^{\rbdim};\parameter$) is defined as follows: seek a reduced solution $\angflux_{RB}^{\rbdim}(\cdot;\parameter)\in \rbTrial^{\rbdim}$, satisfying 
\begin{equation} 
a_{h,\parameter}(\angflux_{RB}^{\rbdim}(\cdot;\parameter), \test_h)=l_{h,\parameter}(\test_h),\quad \test_h\in \rbTrial^{\rbdim}.
\label{eq:rom:p:G}
\end{equation}

The reduced solution $\angflux_{RB}^m(\cdot;\parameter)$ approximates $\angflux_h(\cdot;\parameter)$, and its coordinate vector has the form $\orthorbsmat^{\rbdim} \reducedcoeff^{\rbdim}(\parameter)\approx \angfluxvec_\parameter$,  where the reduced expansion coefficients $\reducedcoeff^{\rbdim}(\parameter)\in\mathR^{\rbdim}$ solves
\begin{equation} \label{eq:rom:alg:p:G}
(\orthorbsmat^{\rbdim})^T\bA_{\parameter}\orthorbsmat^{\rbdim} \reducedcoeff^{\rbdim}(\parameter) = (\orthorbsmat^{\rbdim})^T \dataVec_{\parameter}. 
\end{equation}
For later reference, we denote the reduced system matrix and data vector in this ROM as 
\begin{equation}
\label{eq:SystemMatD-G}
\bArr_\parameter=(\orthorbsmat^{\rbdim})^T\bA_{\parameter}\orthorbsmat^{\rbdim},\quad  {\bbrr}_{\parameter}= (\orthorbsmat^{\rbdim})^T \dataVec_{\parameter},
\end{equation}
with the dependence on $\rbdim$ suppressed.

\subsubsection{ROM based on least-squares Petrov-Galerkin projection}
\label{sec:ROM-PG}

The second reduced order solver ROM($\rbTrial^{\rbdim};\parameter$) is defined based on residual minimization:  seek a reduced solution $\angflux_{RB}^{\rbdim}(\cdot;\parameter)\in \rbTrial^{\rbdim}$, satisfying 
\begin{equation}
\angflux_{RB}^{\rbdim}(\cdot;\parameter)=\textrm{argmin}_{g_h\in\rbTrial^{\rbdim}}\|r_{h,\parameter}(g_h)\|_h.
\label{eq:rom:p:PG}
\end{equation}

Using the algebraic representation of $\|r_{h,\parameter}(g_h)\|_h$ from Lemma \ref{lem:norm:alg}, and with $\reducedcoeff^{\rbdim}(\parameter)$ as the reduced expansion coefficients of $\angflux_{RB}^{\rbdim}(\cdot;\parameter)$,
the ROM in \eqref{eq:rom:p:PG}, written in its algebraic form, is indeed  a weighted least-squares problem,
\begin{equation}
\reducedcoeff^{\rbdim}(\parameter)=\textrm{argmin}_{\mathbf{c}\in\mathR^{\rbdim}}\| \bA_\parameter\orthorbsmat^{\rbdim} \mathbf{c}-\dataVec_{\parameter}\|_{\mM_h}.
\label{eq:rom:alg:p:PG}
\end{equation}
Here $\|\cdot\|_{\mM_h}$ is a weighted norm for $\mathR^{\mN}$, namely, 
$\|\mathbf{g}\|_{\mM_h}=\sqrt{\mathbf{g}^T \mM_h\mathbf{g}}$. The weighting matrix is 
\begin{equation}
\mM_h=\textrm{diag}(\vweight_1\massMat^{-1},\vweight_2\massMat^{-1},\dots,\vweight_{N_{\vel}}\massMat^{-1}),
\label{eq:Mh:weight}
\end{equation}
and it is symmetric positive definite, depending on the spatial and angular discretization of the FOM. Moreover, with the local nature of the spatial space $\Uhh$ (hence of $U_h$), $\mM_h$ is diagonal, or block-diagonal with each block being  $(\pd+1)^{d_{\x}}\times (\pd+1)^{d_{\x}}$.  In either case, it is easy to compute its Cholesky factorization $\mG$,  satisfying $\mM_h=\mG^T\mG$, preserving the same block-diagonal structure of $\mM_h$ with each block being upper-triangular or diagonal. This matrix $\mG$ will be used in Section \ref{sec:impt:cost:ROMPG} (also see Lemma \ref{lem:reform:ROMPG}) to efficiently solve \eqref{eq:rom:alg:p:PG} for a (large) collection of parameter values. For later reference (e.g., in Lemma \ref{lem:reform:ROMPG}), we denote the reduced system matrix and data vector in this ROM as \begin{equation}
\label{eq:SystemMatD-PG}
\bAr_\parameter=\mG\bA_{\parameter}\orthorbsmat^{\rbdim}, \quad  {\bbr}_{\parameter}= \mG\dataVec_{\parameter}.
\end{equation}

The ROM in \eqref{eq:rom:p:PG} (i.e. \eqref{eq:rom:alg:p:PG})  is also said to be 
 a least-squares Petrov-Galerkin (LS-Petrov-Galerkin) projection, due to its equivalent reformulation:  find $\reducedcoeff^{\rbdim}(\parameter)\in \mathR^{\rbdim}$ that satisfies
\begin{equation} \label{eq:rom:alg:abs}
\bW_{\parameter}^T\bA_{\parameter}\orthorbsmat^{\rbdim} \reducedcoeff^{\rbdim}(\parameter) = \bW_{\parameter}^T \dataVec_{\parameter}, 
\end{equation}
where 
\begin{equation} \label{eq:rom:Wh:PG}
\bW_{\parameter}=
\mM_h\bA_{\parameter}\orthorbsmat^{\rbdim}.
\end{equation}
In contrast, the ROM based on Galerkin projection in \eqref{eq:rom:alg:p:G} is defined as \eqref{eq:rom:alg:abs} yet with a different $\bW_{\parameter}$, namely, $\bW_{\parameter}=\orthorbsmat^{\rbdim}$.
The matrix $\bW_{\parameter}$ is the algebraic representation of the (reduced) test space. When $\bW_{\parameter}= \orthorbsmat^{\rbdim}$, \eqref{eq:rom:alg:abs} corresponds to Galerkin projection, otherwise, it corresponds to a Petrov-Galerkin projection.

\begin{remark} 
\label{rem:res}
An alternative and popular 
  Petrov-Galerkin-based ROM can be defined by taking $\bW_{\parameter}= \bA_{\parameter}\orthorbsmat^{\rbdim}$ in \eqref{eq:rom:alg:abs}, and   equivalently, by seeking
  \begin{equation}
\reducedcoeff^{\rbdim}(\parameter)=\textrm{argmin}_{\mathbf{c}\in\mathR^{\rbdim}}\| \bA_\parameter\orthorbsmat^{\rbdim} \mathbf{c}-\dataVec_{\parameter}\|_{\ell_2}.
\end{equation}
That is, the ROM is defined by minimizing the algebraic residual; hence, it will depend on the specific choice of the basis used for the  approximation space $\Uhh$ (hence for $U_h$). Strategies of this type have been widely used in model order reduction, including POD-based ROMs for the RTE in \cite{carlberg2011efficient,behne2022minimally, tano2021affine}. On the other hand, the ROMs involving the residual in their function form as in \eqref{eq:res:gh} and \eqref{eq:rom:p:PG}  are {\bf independent} of a specific choice of bases for $U_h$.
\end{remark}

\begin{remark}
\label{rem:PG:monotone}
    With the nested/embedded  structure of the reduced basis spaces, namely $\rbTrial^{{\rbdim}-1}\subset\rbTrial^{\rbdim}$,  the ROM solution in \eqref{eq:rom:p:PG} or \eqref{eq:rom:alg:p:PG} satisfies
\begin{align*}
&\|r_{h,\parameter}(\angflux_{RB}^{\rbdim}(\cdot;\parameter))\|_h\leq \|r_{h,\parameter}(\angflux_{RB}^{\rbdim-1}(\cdot;\parameter))\|_h\\
\Leftrightarrow&
   \| \bA_{\parameter}\orthorbsmat^{\rbdim} \reducedcoeff^{\rbdim}(\parameter) - \dataVec_{\parameter} \|_{\mM_h}\leq  \| \bA_{\parameter}\orthorbsmat^{\rbdim-1} \reducedcoeff^{\rbdim-1}(\parameter) - \dataVec_{\parameter} \|_{\mM_h}.
\end{align*}
That is,  for the reduced solution $\angflux_{RB}^{\rbdim}(\cdot;\parameter)$ related to a given $\parameter$, the $L_2$ norm  of its residual (or equivalently, the weighted algebraic residual)   is monotonically non-increasing as $\rbdim$ grows. 

\end{remark}

 \subsection{Four RBM-based reduced order models}
\label{sec:RBM2}

There are a total of four reduced order models proposed in this work. They all follow the RBM framework, with its offline training/learning stage given in Alg. \ref{alg:RBM}. (The implementation of the entire algorithm of the proposed reduced order models will be given later in Alg. \ref{alg:entireAlg} after implementation strategies are described in detail.)  Upon termination, a reduced order solver $\textrm{ROM}(\rbTrial^M; \cdot)$ is available, and will be used during the online stage for prediction or testing.  The offline  stage  starts with a sufficiently large parameter training set $\mP_{\text{train}}\subset\mP$ of a finite cardinality. Once the reduced space $\rbTrial^m$ is available, 
the next parameter value $\parameter_{m+1}$ will be chosen in a greedy fashion as the one whose FOM solution is least-represented if $\rbTrial^m$ were taken to define the reduced surrogate solver ROM($\rbTrial^m; \parameter$). To make this optimization task tractable, the greedy pick will be guided by an importance/error indicator $\Delta_{\rbdim}(\parameter)$.  
In the rest of this section, we will present two error indicators, 
 ${\bf L}_1$ and Res  for short, and the stopping criterion used in this work.

\begin{algorithm}[ht]
    \SetAlgoLined
{\bf Initialization:} With $\mu_1\in\mP_{\textrm{train}}$ arbitrarily chosen or as some average, compute $\angflux_h(\cdot; \parameter_1)=\textrm{FOM}(\parameter_1)$. Define $\rbTrial^1 = \mbox{span}\left\{\angflux_h(\cdot; \parameter_1)\right\}$.
Set
$\mP_1=\{\parameter_1\}$, and  $m=1$.\\
\While{stopping criteria not met,}{

{\it i.} For each $\parameter \in \mP_{\textrm{train}}\setminus\mP_m$,
 solve ROM($\rbTrial^m; \parameter$)
 and compute an importance/error indicator  $\Delta_{m}(\parameter)$.

 {\it ii.}  Choose $\parameter_{m+1} = 
\underset{\parameter\in{\mP_{\textrm{train}}\setminus\mP_m}}{\textrm{argmax}}
 {\Delta_{m}(\parameter)}$.
 
  {\it iii.}  Compute $\angflux_h(\cdot; \parameter_{m+1})=\textrm{FOM}(\parameter_{m+1})$. Update $\rbTrial^{m+1} =\rbTrial^m\bigoplus \{ \angflux_h(\cdot; \parameter_{m+1})\}$ and   $\mP_{m+1}:= \mP_{m}\cup\{\parameter_{m+1}\}$.
Set $m:= m+1$.
 }
{\bf{Output:}~}  Set $M=m$, and output $\rbTrial^M$ 
that defines the online $\textrm{ROM}(\rbTrial^M; \cdot).$
\medskip 
\caption{The offline stage of the RBM framework}
\label{alg:RBM}
\end{algorithm}

With the two reduced order solvers ROM($\rbTrial^{\rbdim};\parameter$) in Section \ref{sec:ROM} and the two error indicators  $\Delta_{\rbdim}(\parameter)$, {\it four}  RBM-based reduced order models are proposed in this work for the parametric RTE model, and they are denoted with the following shorthand notation:

\begin{itemize}
\item {\bf G-$L_1$}: $L_1$ importance indicator
+ ROM($\rbTrial^{\rbdim};\parameter$) via Galerkin projection;
\item {\bf PG-$L_1$}: $L_1$ importance indicator  
+ 
ROM($\rbTrial^{\rbdim};\parameter$) via LS-Petrov-Galerkin
projection;
\item {\bf G-Res}: residual-based error indicator  
+ ROM($\rbTrial^{\rbdim};\parameter$) via Galerkin projection;
\item {\bf PG-Res}: residual-based error indicator 
+  
ROM($\rbTrial^{\rbdim};\parameter$) via  LS-Petrov-Galerkin projection.
\end{itemize}

\medskip
\noindent {\bf 1.) $L_1$ error indicator.}  
The first choice of $\Delta_{\rbdim}(\parameter)$ is the $L_1$ importance/error indicator $\Delta^{(L)}_{\rbdim}(\parameter)$. This indicator is proposed in \cite{chen2019robust} and defined as 
\begin{equation} \label{eq:L1 error indicator}
\Delta_{\rbdim}^{(L)}(\parameter): = \| \tilde{\mathbf{c}}_{RB}^{\rbdim}(\parameter) \|_{\ell_1},
\end{equation}
where $\tilde{\mathbf{c}}_{RB}^{\rbdim}(\parameter)$ solves
$\rbs^{\rbdim} \tilde{\mathbf{c}}_{RB}^{\rbdim}(\parameter) = \orthorbsmat^{\rbdim} \reducedcoeff^{\rbdim}(\parameter).$    
In other words, it is the $\ell_1$-norm of the generalized coordinates vector of the reduced solution in $\rbTrial^m$ with respect to the greedily selected snapshot basis $\{\angflux_h(\cdot, \parameter_j)\}_{j=1}^\rbdim$. 
It was shown in \cite{chen2019robust} that $\tilde{\mathbf{c}}_{RB}^{\rbdim}(\parameter)$ 
represents a Lagrange interpolation basis in the parameter space, and $\Delta_{\rbdim}^{(L)}(\parameter)$  represents the corresponding Lebesgue function. The strategy of using the  $L_1$ importance indicator for the greedy selection will
control the growth of the respective Lebesgue constant.
In actual implementation, we utilize  $\rbs^{\rbdim} = \orthorbsmat^{\rbdim} \mathbf{R}_{RB}^{\rbdim}$, and compute  the indicator via
\begin{equation} \label{eq:L1 error indicator:b}
\Delta_{\rbdim}^{(L)}(\parameter)=\|(\mathbf{R}_{RB}^{\rbdim})^{-1} \reducedcoeff^{\rbdim}(\parameter)\|_{\ell_1}.\end{equation}

{This simple error indicator works well in most numerical experiments. It is observed (e.g. in Section \ref{sec:lattice}) that it may not always find the best parameter values especially during the early greedy iterations. In such a case, we propose and apply a variant, referred to as the {\it enhanced $k$-point  $L_1$ error indicator},  with $k$ as a hyper-parameter of a small natural number. This will be discussed in Appendix \ref{app:enhancedL1}.}

\medskip
\noindent {\bf 2.) Residual-based error indicator.}  
The second choice of $\Delta_{\rbdim}(\parameter)$ is the residual-based error indicator  $\Delta^{(R)}_{\rbdim}(\parameter)$, defined as the $L_2$ norm of the residual of the reduced solution,
\begin{equation} \label{eq:res error indicator}
\Delta_{\rbdim}^{(R)}(\parameter): =\|r_{h,\parameter}(\angflux_{RB}^{\rbdim}(\cdot;\parameter))\|_h= \| \mathbf{A}_\parameter \orthorbsmat^{\rbdim} \reducedcoeff^{\rbdim}(\parameter) - \dataVec_\parameter\|_{\mM_h}.
\end{equation}

When the absorption cross section is positively bounded below uniformly (in $\parameter$) by $\absorp^\star>0$ (see the example in Section \ref{sec:spatially homogeneous material}), the scaled residual $\Delta^{(R)}_{\rbdim}(\parameter)/\absorp^\star$ defines a rigorous {\it a posteriori} error estimator, hence the respective ROM is certified. This will be stated more accurately next and can be deduced from standard analysis for the upwind DG method applied to transport equations (e.g. see \cite{sheng2021uniform}) and \cite{haasdonk2017reduced}. {For completeness, the proof is outlined in Appendix \ref{ap:prop1}.}
\begin{prop}
  \label{prop:aPostErr}
  Let $\absorp^\star=\inf_{\x\in\xset,\parameter\in\mP} \absorp(\x; \parameter)$. If $\absorp^\star>0$, then 
  \begin{equation}
       \| \angflux_h(\cdot;\parameter) - \angflux_{RB}^\rbdim(\cdot;\parameter) \|_h\leq \frac{1}{\absorp^\star}\|r_{h,\parameter}(\angflux_{RB}^\rbdim(\cdot;\parameter))\|_h=\frac{\Delta^{(R)}_{\rbdim}(\parameter)}{\absorp^\star}.
       \label{eq:aPostErr}
  \end{equation} 
\end{prop}

{When $\inf_{\x\in\xset,\parameter\in\mP} \absorp(\x; \parameter) = 0$, our model is no longer 
coercive with respect to $\|\cdot\|_h$. The stability of the operator will then depend on other factors (e.g. the geometry of the domain), not just on the cross sections. The  {\it a posteriori} error analysis for the FOM, which our residual-based error indicator inherits, remains to be developed, and we will leave the related  analysis and algorithm development to our future effort.}

\medskip
\noindent {\bf 3.) Spectral ratio  as stopping criterion.}
 It is known from \cite{chen2019robust} that the $L_1$ error indicator does not inform the actual errors in reduced solutions. There are also physically relevant cases with $\absorp^\star=0$ (see numerical examples in Section \ref{sec:num}). This indicates that the two  importance/error indicators  (or their scaled versions) adopted in this work generally neither certify the errors in the reduced solutions nor inform quantitatively how rich the reduced space $\rbTrial^\rbdim$ is for approximation.  Motivated by the strategy in \cite{peng2022reduced}, in Alg. \ref{alg:RBM}  we adopt the stopping criterion based on the following ``energy-type" spectral ratio $r^{(\rbdim)}$ of $\rbs^{\rbdim}$, 
\begin{equation} \label{eq:spectral ratio}
r^{(\rbdim)} = \frac{\sigma_{\rbdim}^{{(\rbdim)}}}{\sqrt{\sum_{j=1}^\rbdim \left(\sigma_j^{(\rbdim)}\right)^2}}.
\end{equation}
Here, $\sigma_1^{(\rbdim)}\geq \sigma_2^{(\rbdim)}\geq\cdots\sigma_\rbdim^{(\rbdim)}\geq 0$ are the singular values of $\rbs^{\rbdim}$. Though not being proved rigorously, it is observed that $r^{(\rbdim)}$ by our proposed methods is monotonically non-increasing. 
The reduced space $\rbTrial^\rbdim$ is deemed rich enough once $r^{(\rbdim)}$ reaches below a given error tolerance, namely $r^{(\rbdim)}\leq \text{tol}_{SRatio}$, for the first time. The greedy iteration in Alg. \ref{alg:RBM} will be terminated subsequently.  This ensures $\rbs^{\rbdim}$ $(\rbdim\leq M)$ has full rank.  The squared spectral ratio in \eqref{eq:spectral ratio} reflects the fraction of the energy introduced by the newly added snapshot $\angfluxvec_{\parameter_{\rbdim}}$ in the greedy procedure. In \cite{peng2022reduced}, the ``nuclear-type'' spectral ratio $\sigma_{\rbdim}^{{(\rbdim)}}/\left(\sum_{j=1}^\rbdim \sigma_j^{(\rbdim)}\right)$ was used.  The two definitions lead to negligible difference in the history of the spectral ratios as the reduced dimension $\rbdim$ increases.  
In actual implementation, one only needs to compute the singular values of $\rbs^{\rbdim}$, not its entire singular value decomposition.

\section{On implementation and computational complexity}
\label{sec:impt:cost}

Upon termination of the offline training/learning stage of the RBM framework in Alg. \ref{alg:RBM},  reduced order solvers $\textrm{ROM}(\rbTrial^M; \cdot)$ are available and they are defined  with respect to an $M$-dimensional reduced basis space $\rbTrial^{M}$, a much smaller approximation space than the $\mN$-dimensional $U_h$. 
Special care is needed in actual implementation to fully harvest the potential computational efficiency of these surrogate solvers  
in simulating the parametric RTE \eqref{eq:P-RTE} in multi-query settings.
This is indeed closely related to the efficient implementation of both the ROM-building offline stage and the online prediction stage, given that $\textrm{ROM}(\rbTrial^{\rbdim}; \cdot)$ with $\rbdim<M$ is used in the training stage.  This section will be dedicated to the algorithmic details and strategies to implement the major ingredients to build the reduced order solvers as described in Alg. \ref{alg:RBM}, with attention not only to the offline and online efficiency, but also to the conditioning of the sub-tasks along with numerical robustness that especially ensures stagnation-free residual evaluation.

By following the standard practice to achieve the algorithmic efficiency for multi-query tasks, we assume the parameter dependence of the RTE model \eqref{eq:P-RTE} is affine, and this is also referred to parameter separability \cite{haasdonk2017reduced}. Otherwise,  techniques such as the empirical interpolation method (EIM) 
 \cite{barrault2004empirical} will be applied first. 
With this assumption, terms such as $\scat(\x;\parameter)$, $\absorp(\x;\parameter)$, $\source(\x;\parameter)$,  $\bc(\x,\cdot;\parameter)$ are assumed to have a separable form, namely $\psi(\x;\parameter)=\sum_{q=1}^Q\theta_q(\parameter)\psi_q(\x)$, with a small or moderate $Q$. Algebraically, one equivalently has 
\begin{equation}
\scatmatPar  = \sum_{q=1}^{Q_s} \theta_q^s(\parameter) \scatmat^q,\quad
\absorpmatPar  = \sum_{q=1}^{Q_a} \theta_q^a(\parameter) \absorpmat^q,
\quad
\dataVec_{\parameter}  = \sum_{q=1}^{Q_b} \theta_q^b(\parameter) \dataVec^q,
\label{eq:sys:mat:affine}
\end{equation}
where $\{\scatmat^q\}_{q=1}^{Q_s}$ and $\{\absorpmat^q\}_{q=1}^{Q_a}$ are in $\mathR^{\mN\times\mN}$, and they are either sparse or have a special structure as discussed in Section \ref{sec:FOM:alg}, and $\{\dataVec^q\}_{q=1}^{Q_b}$ is in $\mathR^{\mN}$. They are {\it independent} of the parameter $\parameter$.
 The coefficients $\theta_q^s(\parameter), \theta_q^a(\parameter), \theta_q^b(\parameter)$ are continuous functions on the compact parameter set $\mP$.

In the offline greedy selection procedure of 
Alg. \ref{alg:RBM} or during the online prediction stage in multi-query settings, one would want to solve ROM($\rbTrial^{\rbdim};\parameter)$ and/or to compute the error indicator $ \Delta_{\rbdim}(\parameter)$ for a (large) collection of $N_{\parameter}$ parameter values in $\mP_{\textrm{interest}}\subset\mP$, with $N_{\parameter}=|\mP_{\textrm{interest}}|$.  With the affine structure of the parameter dependence being utilized, good computational efficiency can be achieved for these building blocks by embedding an additional offline-online framework.  More specifically, in the offline phase of such framework, {\it parameter-independent} pre-computation is performed, while in the online phase,  {\it parameter-dependent} evaluation or finding solutions to problems with much reduced sizes will be carried out at each parameter instance.
Such a strategy is not new for model order reduction and can be done fairly straightforwardly for the ROM($\rbTrial^{\rbdim};\cdot$) based on Galerkin projection,  while more algorithmic development is needed for that based on LS-Petrov-Galerkin projection to avoid worsened conditioning (see discussions in Section \ref{sec:spatially varying scattering}).
In addition,  when the residual $r_{h,\parameter}(\cdot)$ is involved, e.g., in  ROM($\rbTrial^{\rbdim};\parameter$) based on LS-Petrov-Galerkin projection,  or in calculating the residual-based error indicator, or even in computing  training or test   residual errors to assess the proposed methods, additional algorithmic strategy is needed for robust residual evaluation that is free from stagnation when the reduced dimension $\rbdim$ grows and the designed resolution/fidelity of the ROMs improves (see numerical experiments in Sections \ref{sec:two-material}-\ref{sec:spatially varying scattering}; also see Figure 2.6 in \cite{haasdonk2017reduced} and the associated discussion). 
{All of these will be elaborated on in Sections \ref{sec:impt:cost:ROMG}-\ref{sec:indicator-cost}, along with the computational complexity of each building block algorithm.   In Section \ref{sec:TotalCost}, we provide the implementation as well as the computational complexity of the offline and online stages of the entire algorithm of all four proposed RBM-based ROMs.} 

This section presents one major technical contribution of our work. To further provide perspective and insights to some of our proposed implementation strategies, several variants will be discussed in Appendix \ref{app:variants} and they will be  referred to in this section whenever relevant. 
These variants have their own merits, and we will point out the situations when they work well and when they may not be as robust or as efficient as the strategies proposed in this section.

\subsection{ROM($\rbTrial^{\rbdim};\parameter$) via Galerkin projection} 
\label{sec:impt:cost:ROMG}

To implement the reduced order solver ROM($\rbTrial^{\rbdim};\parameter$) based on  Galerkin projection in Section \ref{sec:ROM-G}, we first 
utilize the affine assumption and write the reduced system matrix ${\bArr}_\parameter$ and data vector ${\bbrr}_{\parameter}$ in \eqref{eq:SystemMatD-G} as follows,
\begin{subequations}
\begin{align}
{\bArr}_\parameter&=\sum_{q=1}^{Q_A} \theta_q^A(\parameter) {\bArr}^q\\
&:=
\Big((\orthorbsmat^{\rbdim})^T\Upwind\orthorbsmat^{\rbdim}\Big)+ 
\sum_{q=1}^{Q_s} \theta_q^s(\parameter) \Big((\orthorbsmat^{\rbdim})^T\scatmat^q\orthorbsmat^{\rbdim}\Big)+ \sum_{q=1}^{Q_a} \theta_q^a(\parameter) \Big((\orthorbsmat^{\rbdim})^T\absorpmat^q\orthorbsmat^{\rbdim}\Big),\notag\\
{\bbrr}_{\parameter} & = \sum_{q=1}^{Q_b} \theta_q^b(\parameter){\bbrr}^q,
\end{align}
\label{eq:A:G:p}
\end{subequations}
with the parameter-independent  $\{{\bArr}^q\}_{q=1}^{Q_A}\subset\mathR^{\rbdim\times \rbdim}$ consisting of $(\orthorbsmat^{\rbdim})^T\Upwind\orthorbsmat^{\rbdim}$, $\{(\orthorbsmat^{\rbdim})^T\scatmat^q\orthorbsmat^{\rbdim}\}_{q=1}^{Q_s}$, $\{
(\orthorbsmat^{\rbdim})^T\absorpmat^q\orthorbsmat^{\rbdim}\}_{q=1}^{Q_a}$,  $Q_A=1+Q_s+Q_a$, and 
the parameter-independent ${\bbrr}^q=(\orthorbsmat^{\rbdim})^T\dataVec^q\in\mathR^{\rbdim}$, $q=1,\cdots, Q_b$. The ROM($\rbTrial^{\rbdim};\parameter$) in  \eqref{eq:rom:alg:p:G} will become
\begin{equation}
  {\bArr}_\parameter\reducedcoeff^{\rbdim}(\parameter)={\bbrr}_{\parameter}.
    \label{eq:rom:G:sol}
\end{equation}

$\bArr_{\parameter}$ has full rank, and this property is inherited from the reduced matrix $\rbsMat^{\rbdim}$ (or its orthonormalized 
$\orthorbsmat^{\rbdim}$) being full rank, as discussed in Section \ref{sec:RBM2},  and $\bA_{\parameter}$ being invertible\footnote{This is observed to hold in  all our numerical examples, and is proved in \cite{sheng2021uniform}  when the absorption cross section is bounded below uniformly by a positive constant, namely $\absorp^\star=\inf_{\x\in\xset,\parameter\in\mP} \absorp(\x; \parameter)>0$.}.
To solve this reduced $\rbdim\times\rbdim$ problem for a (large) collection of $N_{\parameter}$ parameter values in $\mP_{\textrm{interest}}$,  we will utilize the separable structure of ${\bArr}_\parameter$ and ${\bbrr}_{\parameter}$ and follow an offline-online framework to achieve  good computational efficiency.  This is outlined in Alg. \ref{alg:pG}. {As detailed  in Appendix \ref{app:CC:alg2}, the associated computational complexity is  
\begin{equation}
  \boxed{  O\big(\mN(\rbdim^2 Q_A+\rbdim Q_b)\big)+ O\big(N_{\parameter}(\rbdim^3+\rbdim^2 Q_A+\rbdim Q_b)\big)}.
    \label{eq:CC:ROM-G}
\end{equation}
In  comparison, without the affine structure, a cost of $O(\mN N_{\parameter} \rbdim^2)$ is required. 
With the dependence on $\mN$ and  $N_{\parameter}$ in a multiplicative form, namely, $\mN N_{\parameter}$,  the cost in \eqref{eq:CC:ROM-G} is much lower in  a multi-query setting when $N_\parameter\gg 1$.}

\begin{algorithm}[ht]
\caption{To solve the Galerkin-based ROM  \eqref{eq:rom:alg:p:G} (or \eqref{eq:rom:G:sol}) for all $\parameter\in\mP_{\textrm{interest}}$}

  \begin{algorithmic}[1]  
 \State {\bf Offline:} Pre-compute  ${\bArr}^q\in\mathR^{\rbdim\times\rbdim}$, $q=1,\dots, Q_A$ and ${\bbrr}^q\in\mathR^{\rbdim}$, $q=1,\dots, Q_b$ as defined in Section \ref{sec:impt:cost:ROMG}.
\State {\bf Online:} For  each $\parameter\in\mP_{\textrm{interest}}$, 
    \begin{algorithmic}
  \State i.) Compute ${\bArr}_\parameter=\sum_{q=1}^{Q_A} \theta_q^A(\parameter) {\bArr}^q\in\mathR^{\rbdim\times\rbdim}$, and ${\bbrr}_{\parameter}  =\sum_{q=1}^{Q_b} \theta_q^b(\parameter){\bbrr}^q\in\mathR^{\rbdim}$. 
  \State ii.)  Solve \eqref{eq:rom:G:sol} to get $\reducedcoeff^{\rbdim}(\parameter)\in\mathR^{\rbdim}$.
  \end{algorithmic}
\end{algorithmic}
\label{alg:pG}
\end{algorithm}

\subsection{ROM($\rbTrial^{\rbdim};\parameter$) via LS-Petrov-Galerkin projection}
\label{sec:impt:cost:ROMPG}

To implement the reduced order solver ROM($\rbTrial^{\rbdim};\parameter$) based on  LS-Petrov-Galerkin projection,  one needs to solve a weighted parametric least-squares problem \eqref{eq:rom:alg:p:PG}. Directly solving the associated (weighted) normal equation \eqref{eq:rom:alg:abs} with \eqref{eq:rom:Wh:PG} can suffer from worsened conditioning, due to $\textrm{cond}_2\big((\mM_h\bA_{\parameter}\orthorbsmat^{\rbdim})^T \bA_{\parameter}\orthorbsmat^{\rbdim}\big) = 
\big(\textrm{cond}_2(\mG\bA_{\parameter}\orthorbsmat^{\rbdim})\big)^2$ (also see numerical example in Section \ref{sec:spatially varying scattering}).  
Instead, this reduced weighed least-squares problem \eqref{eq:rom:alg:p:PG} will be solved via QR factorization. To this end, we first reformulate \eqref{eq:rom:alg:p:PG} into a standard (i.e. non-weighted) least-squares problem, and this can be verified directly.

\begin{lemma}
\label{lem:reform:ROMPG}
The weighted least-squares problem \eqref{eq:rom:alg:p:PG} is equivalent to 
\begin{equation}
\reducedcoeff^{\rbdim}(\parameter)=\textrm{argmin}_{\mathbf{c}\in\mathR^{\rbdim}}\| {\bAr}_{\parameter} \mathbf{c}-{\bbr}_{\parameter}\|_{\ell_2},
\label{eq:rom:PG:noW}
\end{equation}
where the reduced system matrix ${\bAr}_\parameter$ and data vector ${\bbr}_{\parameter}$ in \eqref{eq:SystemMatD-PG} with the affine assumption are given as follows, 
\begin{equation}
{\bAr}_{\parameter}=\sum_{q=1}^{Q_A} \theta_q^A(\parameter) {\bAr}^q:=
\Big(\mG\Upwind\orthorbsmat^{\rbdim}\Big)+ 
\sum_{q=1}^{Q_s} \theta_q^s(\parameter) \Big(\mG\scatmat^q\orthorbsmat^{\rbdim}\Big)+ \sum_{q=1}^{Q_a} \theta_q^a(\parameter) \Big(\mG\absorpmat^q\orthorbsmat^{\rbdim}\Big),
\;\;
{\bbr}_{\parameter}  = \sum_{q=1}^{Q_b} \theta_q^b(\parameter){\bbr}^q.
\label{eq:A:PG:p}
\end{equation}
Here $\mG$ is the Cholesky  factor of the weighting matrix $\mM_h$, as defined in Section \ref{sec:ROM-PG};  $\{{\bAr}^q\}_{q=1}^{Q_A}\subset\mathR^{\mN\times \rbdim}$, and they are parameter-independent and consist of $\mG\Upwind\orthorbsmat^{\rbdim}$, $\{  \mG\scatmat^q\orthorbsmat^{\rbdim}\}_{q=1}^{Q_s}$, $\{
\mG\absorpmat^q\orthorbsmat^{\rbdim}\}_{q=1}^{Q_a}$, with $Q_A=1+Q_s+Q_a$;
$\{{\bbr}^q\}_{q=1}^{Q_b}\subset \mathR^{\mN}$ are parameter-independent, with ${\bbr}^q=\mG\dataVec^q$.
\end{lemma}

$\bAr_{\parameter}$ has full rank,  just as $\bArr_{\parameter}$ in the previous subsection.  
Suppose we have access to its (reduced or economy) QR factorization, namely ${\bAr}_{\parameter}=\widehat{\bQ}_\parameter\bR_\parameter$, with $\widehat{\bQ}_\parameter\in \mathR^{\mN\times \rbdim}$ satisfying $\widehat{\bQ}_\parameter^T\widehat{\bQ}_\parameter=\bI_{\rbdim\times \rbdim}$, and $\bR_\parameter\in \mathR^{\rbdim\times \rbdim}$ being upper triangular and invertible; then mathematically the solution to \eqref{eq:rom:PG:noW} is
\begin{equation}
\reducedcoeff^{\rbdim}(\parameter)=\bR_\parameter^{-1} \widehat{\bQ}_\parameter^T {\bbr}_\parameter.
\label{eq:rom:PG:sol}
\end{equation}
It is computationally expensive to directly compute the QR factorization of the parametric matrix $\bAr_{\parameter}$ for many parameter values and then compute the respective solution from \eqref{eq:rom:PG:sol}. 
To reduce the computational complexity,  the following offline-online procedure, described in  Alg. \ref{alg:pLS:offline}-Alg. \ref{alg:pLS:online} and inspired by  recent development for low-rank approximations of parametric matrices in \cite{kressner2024randomized}, is proposed to solve \eqref{eq:rom:PG:noW}, with the mathematical justification provided by Theorem \ref{thm:why:QRp} {and proved in Appendix \ref{ap:thm1}.} Recall $\mN\gg \rbdim$ and $Q_A, Q_b$ are moderate. Thus, it is reasonable to assume $\mN\ge \rbdim Q_A+Q_b$.

\begin{algorithm}[ht]
\caption{Offline stage to solve the parametric least-squares problem \eqref{eq:rom:alg:p:PG} (or \eqref{eq:rom:PG:noW})}
%
  \begin{algorithmic}[1]  
 \State Compute $\bAr^q, q=1,\dots Q_A$ and $\bbr^q, q=1,\dots Q_b$ as defined in Lemma \ref{lem:reform:ROMPG}.
 \State For the concatenated matrix $\bB\in\mathbb{R}^{\mN\times Q_B}$, defined as
 \begin{equation}
\bB=[\bAr^1, \bAr^2, \dots, \bAr^{Q_A}, \bbr^1,\bbr^2,\dots,\bbr^{Q_b}],\quad \textrm{with}\;\; {Q_B=\rbdim Q_A+Q_b,}
\label{eq:QRp:0}
 \end{equation}
 perform economy QR factorization with column pivoting, namely
\begin{equation}
\label{eq:QRp:1}
    \bQ\bR=\bB\bP,
\end{equation}
where $\bQ\in \mathR^{\mN\times s}$ has orthonormal columns, $ \bR\in\mathR^{s\times Q_B}$ is upper-triangular, and $\bP\in\mathR^{Q_B\times Q_B}$ is a permutation matrix. Here $\rbdim \leq s\leq Q_B$.
  \State Compute $\bY^q=\bQ^T\bAr^q \in\mathR^{s\times \rbdim}$, $q=1,2,\dots, Q_A$; Compute  $\tilde{\bb}^q=\bQ^T\bbr^q\in\mathR^{s}$, $q=1,2,\dots, Q_b$.
\end{algorithmic}
\label{alg:pLS:offline}
\end{algorithm}

\begin{algorithm}[ht]
\caption{Online stage to solve the parametric least-squares problem \eqref{eq:rom:alg:p:PG} (or \eqref{eq:rom:PG:noW}) for all $\parameter\in\mP_{\textrm{interest}}$}

For each $\parameter\in\mP_{\textrm{interest}}$, we proceed as follows to solve \eqref{eq:rom:alg:p:PG} (or \eqref{eq:rom:PG:noW}).
  \begin{algorithmic}[1]  

\State Compute $\bY_\parameter=\sum_{q=1}^{Q_A}\theta_q^A(\parameter)\bY^q\in\mathR^{s\times \rbdim}$.
 \State Compute QR factorization for the full-rank $\bY_\parameter$ and   get
\begin{equation}
\label{eq:QRp:2}
\tilde{\bQ}_\parameter\bR_\parameter=\bY_\parameter. 
\end{equation}
Here $\tilde{\bQ}_\parameter \in\mathR^{s\times \rbdim}$ has orthonormal columns, $\bR_\parameter\in  \mathR^{\rbdim\times \rbdim}$ is upper-triangular. 
 \State   Compute $\bd_\parameter=\tilde{\bQ}_\mu^T\left(\sum_{q=1}^{Q_b} \theta_q^b(\parameter) \tilde{\bb}^q\right)\in\mathR^{\rbdim}.$
 \State  Solve an upper-triangular system and get $\reducedcoeff^{\rbdim}(\parameter)=\bR_\parameter^{-1}\bd_\parameter\in\mathR^{\rbdim}.$
\end{algorithmic}
\label{alg:pLS:online}
\end{algorithm}

\begin{theorem} 
\label{thm:why:QRp}
Assume $\mN\ge Q_B=\rbdim Q_A+Q_b$. The economy QR factorization of $\bAr_{\parameter}$ is $\bAr_{\parameter}=\widehat{\bQ}_\parameter\bR_\parameter$, where $\widehat{\bQ}_\parameter=\bQ\tilde{\bQ}_\parameter\in\mathR^{\mN\times\rbdim}$, satisfying $\widehat{\bQ}_\parameter^T\widehat{\bQ}_\parameter=\bI_{\rbdim\times \rbdim}$, and $\bR_\parameter\in \mathR^{\rbdim\times \rbdim}$ being upper triangular. Here, $\bQ\in\mathR^{\mN\times s}$ is computed from \eqref{eq:QRp:1} in Alg. \ref{alg:pLS:offline} and its rank is $s$, with some $s\in[\rbdim, Q_B]$, while $\tilde{\bQ}_\mu\in\mathR^{s\times \rbdim}$ and $ \bR_\mu$ are computed from \eqref{eq:QRp:2} in Alg. \ref{alg:pLS:online}. In addition,  $\bY_\parameter$ from 
Step 1 of Alg. \ref{alg:pLS:online} has full rank. As a consequence, based on \eqref{eq:rom:PG:sol}, the solution to \eqref{eq:rom:PG:noW} is 
\begin{equation}
    \reducedcoeff^{\rbdim}(\parameter)=\bR_\parameter^{-1} (\bQ\tilde{\bQ}_\parameter)^T (\sum_{q=1}^{Q_b} \theta_q^b(\parameter){\bbr}^q)=\bR_\parameter^{-1} \tilde{\bQ}_\parameter^T \left(\sum_{q=1}^{Q_b} \theta_q^b(\parameter)(\bQ^T{\bbr}^q)\right).
    \label{eq:QRp:3}
\end{equation}
\end{theorem}

{As detailed in Appendix \ref{app:CC:alg34}, the computational complexity of Alg. \ref{alg:pLS:offline}-Alg. \ref{alg:pLS:online}  is 
\begin{equation}
\boxed{O\big(\mN (m^2Q_A^2+Q_b^2)\big)+O\big(N_{\parameter} (\rbdim^3 Q_A+\rbdim^2 (Q_A^2+Q_b)+Q_b^2\big)}.
    \label{eq:CC:ROM-PG}
\end{equation}
In comparison, without the affine structure, a cost of $O(\mN N_{\parameter} \rbdim^2)$ is required. Again, our proposed strategy Alg. \ref{alg:pLS:offline}-Alg. \ref{alg:pLS:online} is more efficient in a multi-query setting when $N_\parameter\gg 1$.}

\begin{remark} At a first glance, it seems more natural to define $\textbf{B}$ in \eqref{eq:QRp:0} without the data vectors, namely, to define $\bB=[\bAr^1, \bAr^2, \dots, \bAr^{Q_A}]$ just as in \cite{kressner2024randomized}. This is largely true. The data vectors  are included in $\textbf{B}$ in our proposed strategy, due to some extra consideration (e.g., for robust residual evaluation) of the overall algorithm.  This will be further elaborated on in Appendix \ref{app:variants:1} where a variant Alg. \ref{alg:pLS:offline}$^\prime$-Alg. \ref{alg:pLS:online} of our implementation strategy is discussed.
\end{remark}

\begin{remark}
To ensure good computational efficiency in solving \eqref{eq:rom:PG:noW} for many parameter values,  it is important that one never directly forms or calculates $\widehat{\bQ}_{\parameter}$ in the proposed procedure above, as this matrix depends on the full order dimension $\mN$ and is of size $\mN\times\rbdim$.
\end{remark}

\subsection{Evaluations of error indicators and residual}
\label{sec:indicator-cost}

We now turn to the task of computing error indicators  for a (large) collection of  $N_\parameter$ parameter values. 
A closely related task is to evaluate the residual 
$\|\bA_{\parameter} \orthorbsmat^{\rbdim}\mathbf{c} - \dataVec_\parameter\|_{\mM_h}=\| {\bAr}_{\parameter} \mathbf{c}-{\bbr}_{\parameter}\|_{\ell_2}$ with $\mathbf{c}=\reducedcoeff^{\rbdim}(\parameter)$ repeatedly, arising either from the computation of residual-based error indicators, or from the evaluation of training and test residual errors to demonstrate the proposed methods.

\medskip
\noindent {\bf 1.) $L_1$ error indicator $\Delta^{(L)}_{\rbdim}(\parameter)$.}  
Recall $\Delta_{\rbdim}^{(L)}(\parameter)=\|(\mathbf{R}_{RB}^{\rbdim})^{-1} \reducedcoeff^{\rbdim}(\parameter)\|_{\ell_1}$ in \eqref{eq:L1 error indicator:b}, with the upper-triangular $\mathbf{R}_{RB}^{\rbdim}$ and the reduced expansion coefficient vector $\reducedcoeff^{\rbdim}(\parameter)$  available as outlined in Alg. \ref{alg:RBM}, and  $(\mathbf{R}_{RB}^{\rbdim})^{-1} \reducedcoeff^{\rbdim}(\parameter)$ solved by back substitution with exactly $m^2$ basic operations. 
The total cost of computing $\Delta_{\rbdim}^{(L)}(\parameter)$ at $N_\parameter$ parameter values is hence 
\begin{equation} \label{eq:L1 error indicator cost}
\boxed{O(N_{\parameter} (\rbdim^2 + \rbdim))=O(N_{\parameter} \rbdim^2)}.
\end{equation}
\noindent {\bf 2.)  Residual $\|\bA_{\parameter} \orthorbsmat^{\rbdim}\mathbf{c} - \dataVec_\parameter\|_{\mM_h}=\| {\bAr}_{\parameter} \mathbf{c}-{\bbr}_{\parameter}\|_{\ell_2}$ with $\mathbf{c}=\reducedcoeff^{\rbdim}(\parameter)$ at $\parameter$}.  This is also the  residual-based error indicator  $\Delta^{(R)}_{\rbdim}(\parameter)$.
    In order to compute $\| {\bAr}_{\parameter} \mathbf{c}-{\bbr}_{\parameter}\|_{\ell_2}$ with $\mathbf{c}=\reducedcoeff^{\rbdim}(\parameter)$ at a (large) collection of $\parameter$
    both efficiently and robustly, Alg. \ref{alg:ResC:1}  is proposed\footnote{Alg. \ref{alg:ResC:1} can be used to evaluate $\| {\bAr}_{\parameter} \mathbf{c}-{\bbr}_{\parameter}\|_{\ell_2}$ at any $(\parameter, \mathbf{c})$ pair, not just the case with $\mathbf{c}=\reducedcoeff^{\rbdim}(\parameter)$.},  with the mathematical justification  provided in Theorem \ref{thm:ResC:1}  {and proved in Appendix \ref{ap:thm2}.}  This algorithm is based on QR factorization, in a similar spirit as Alg. \ref{alg:pLS:offline}-Alg. \ref{alg:pLS:online} yet with additional consideration to improve the numerical robustness and efficiency, also see  Appendix \ref{app:variants:2}-\ref{app:variant:3} where some related variants are discussed.

\begin{theorem} 
\label{thm:ResC:1} 
Using the same notation as in  Lemma \ref{lem:reform:ROMPG} and Alg. \ref{alg:pLS:offline}, and denoting
\begin{equation}
    \boldsymbol{\theta}^A(\parameter)=[\theta_1^A(\parameter),\theta_2^A(\parameter),\dots,\theta_{Q_A}^A(\parameter)]^T, \quad \boldsymbol{\theta}^b(\parameter)=[\theta_1^b(\parameter),\theta_2^b(\parameter),\dots,\theta_{Q_b}^b(\parameter)]^T,
    \label{eq:ResC:1}
\end{equation} 
the following holds
\begin{equation} 
\|\bA_{\parameter} \orthorbsmat^{\rbdim}\mathbf{c} - \dataVec_\parameter\|_{\mM_h}=\| {\bAr}_{\parameter} \mathbf{c}-{\bbr}_{\parameter}\|_{\ell_2}=\left\|\bR \bP^T\begin{bmatrix}\boldsymbol{\theta}^A(\parameter)\otimes \mathbf{c}\\
-\boldsymbol{\theta}^b(\parameter)
\end{bmatrix}\right\|_{\ell_2}.
\label{eq:ResC:0}
    \end{equation}
\end{theorem}

\begin{algorithm}[ht]
\caption{To compute the residual $\| {\bAr}_{\parameter} \mathbf{c}-{\bbr}_{\parameter}\|_{\ell_2}$ with $\mathbf{c}=\reducedcoeff^{\rbdim}(\parameter)$ for all $\parameter\in\mP_{\textrm{interest}}$}
{\bf Offline stage:}
Following the same notation as in Alg. \ref{alg:pLS:offline}  
and Theorem  \ref{thm:ResC:1},
  \begin{algorithmic}[1]  
 \State If the residual computation is for either PG-$L_1$ or PG-Res, go to next step; if it is for  G-$L_1$ or G-Res, perform Steps 1-2 of Alg. \ref{alg:pLS:offline}. 
 \State Compute $\bG=\bR \bP^T\in\mathR^{s\times Q_B}$.
\end{algorithmic}
{\bf Online stage:} For each $\parameter\in\mP_{\textrm{interest}}$ with $\mathbf{c}=\reducedcoeff^{\rbdim}(\parameter)$,
 \begin{algorithmic}[1]  
 \State Compute $\left\|\bG\begin{bmatrix}\boldsymbol{\theta}^A(\parameter)\otimes \mathbf{c}\\
-\boldsymbol{\theta}^b(\parameter)
\end{bmatrix}\right\|_{\ell_2}$, and this gives
the residual $\| {\bAr}_{\parameter} \mathbf{c}-{\bbr}_{\parameter}\|_{\ell_2}$.
\end{algorithmic}
\label{alg:ResC:1}
\end{algorithm}

By similar procedures as in Appendix \ref{app:CC}  to count the cost and $s\leq Q_B=\rbdim Q_A+Q_b$, one can come up with the computational complexity\footnote{Caution is needed to interpret the cost. For example, if one builds the reduced order model \textrm{G}-$L_1$ without the need to compute the training residual errors, Alg. \ref{alg:ResC:1} will  not be needed.} of Alg. \ref{alg:ResC:1} to evaluate the residual $\| {\bAr}_{\parameter} \mathbf{c}-{\bbr}_{\parameter}\|_{\ell_2}$ with $\mathbf{c}=\reducedcoeff^{\rbdim}(\parameter)$ for $N_\parameter$ parameter values in $\mP_{\textrm{interest}}$,  equivalently, at $N_\parameter$ instances of $(\mu, \mathbf{c})$.
\begin{subequations}
 \label{eq:CC:ResC}
\begin{align}
        &\boxed{O\big((\mN+N_{\parameter}) (\rbdim^2 Q_A^2 + Q_b^2)\big)}
        &\textrm{if this is for G-$L_1$ or G-Res,}\label{eq:CC:ResC:G}\\
    &\boxed{O\big(N_{\parameter} (\rbdim^2 Q_A^2 + Q_b^2)\big)}&\textrm{if this is for PG-$L_1$ or PG-Res.}\label{eq:CC:ResC:PG}
\end{align}
\end{subequations}

\begin{remark}
\label{rem:cost:2indicators}
With a more detailed counting, one can see that the number of basic operations ($+, -, *, \div,\sqrt{\cdot}$) required to evaluate one $L_1$  error indicator $\Delta^{(L)}_{\rbdim}(\cdot)$ is always no bigger than that for one residual-based error indicator $\Delta^{(R)}_{\rbdim}(\cdot)$. For the former, it is $m^2+2m$, while for the latter, the online stage of Alg. \ref{alg:ResC:1} alone will require at least $s\times(2Q_B-1)$ operations, and 
$$s\times(2Q_B-1)=s(2(mQ_A+Q_b)-1)\geq m(2(m+1)-1)=2m^2+m\geq m^2+2m.$$ 
\end{remark}

\subsection{The entire algorithm  and  computational complexity}
\label{sec:TotalCost}

In Alg. \ref{alg:entireAlg}, the implementation of the entire algorithm of all four RBM-based ROMs is summarized. Based on this, the computational complexities will be assembled and discussed for the offline stage and the online stage of our proposed reduced order models.

\begin{algorithm}[!ht]  
\caption{Implementation of the entire algorithm: RBM-based ROM}
\label{alg:entireAlg}
\SetAlgoLined

{\bf Offline stage:} \Comment{to build/train ROM($\rbTrial^{M};\parameter$)}

\smallskip
\Indp
{\bf Input:}  Maximum number of greedy iterations $M_{\textrm{tol}}$; tolerance for spectral ratio $\text{tol}_{SRatio}$;  training set  $\mP_{\textrm{train}}$

\smallskip
 {\bf Initialization:} With $\mu_1\in\mP_{\textrm{train}}$ arbitrarily chosen or as some average,
        \\
        \Indp i.) Solve $\textrm{FOM}(\parameter_1)$ to obtain $\angfluxvec_{\parameter_1}$.
        \\
        ii.) Form $\rbsMat^1 = [\angfluxvec_{\parameter_1}]$ and its normalized version $\orthorbsmat^1$ satisfying $\rbsMat^1 = \orthorbsmat^1 \mathbf{R}_{RB}^1$ with $\mathbf{R}_{RB}^1 = \| \angfluxvec_{\parameter_1} \|_{\ell_2} \in \mathbb{R}$.
        \\
        iii.) Set $\mP_1 = \{\parameter_1\}$, $r^{(1)}=1$, $\rbdim = 1$.

\Indm 
\smallskip
 {\bf Reduced basis generation:}\\
 \While{$r^{(\rbdim)}>\text{tol}_{SRatio}$ and $\rbdim<M_{\textrm{tol}}$,}{
    i.) For each $\parameter \in \mP_{\textrm{interest}}=\mP_{\textrm{train}} \setminus \mP_{\rbdim}$,  run  ROM$(\rbTrial^\rbdim;\parameter)$ by 
$
\begin{cases}
      \text{Alg. \ref{alg:pG} \quad \tcc{for G-$L_1$ and G-Res}}\\
      \text{Alg. \ref{alg:pLS:offline}- Alg. \ref{alg:pLS:online} 
      \quad \tcc{for PG-$L_1$ and PG-Res}}
  \end{cases}
  $
  
       to get reduced expansion coefficients $\reducedcoeff^{\rbdim}(\parameter)$, and compute error indicator as 
   $\begin{cases}
         \Delta_{\rbdim}(\parameter)=\Delta_{\rbdim}^{(L)}(\parameter) \;\text{in \eqref{eq:L1 error indicator:b}
          \quad \tcc{for G-$L_1$ and PG-$L_1$}}\\
         \Delta_{\rbdim}(\parameter)=\Delta_{\rbdim}^{(R)}(\parameter) \; \text{by Alg.  \ref{alg:ResC:1}} 
         \quad \text{\tcc{for G-Res and PG-Res}}
     \end{cases} 
     $
     
    ii.) Choose $\parameter_{m+1} = \underset{\parameter\in{\mP_{\textrm{train}}\setminus\mP_m}}{\textrm{argmax}} {\Delta_{m}(\parameter)}$.
    \\
    iii.) Solve $\textrm{FOM}(\parameter_{\rbdim+1})$ to obtain $\angfluxvec_{\parameter_{\rbdim+1}}$.
    \\
    iv.) Form $\rbsMat^{\rbdim+1} = [\rbsMat^{\rbdim}, \angfluxvec_{\parameter_{\rbdim+1}}]$ and update its QR factorization $\rbsMat^{\rbdim+1} = \orthorbsmat^{\rbdim+1} \mathbf{R}_{RB}^{\rbdim+1}$ by computing the last columns of $\orthorbsmat^{\rbdim+1}$ and $\mathbf{R}_{RB}^{\rbdim+1}$ (see \eqref{eq:CGSR} for their nested form) through  the  incremental Classical Gram-Schmidt with Reorthogonalization (CGSR) algorithm.
    \\
    v.) Compute the spectral ratio $r^{(\rbdim+1)}$ based on the singular values of $\rbsMat^{\rbdim+1}$.\\
    vi.) Set $\mP_{\rbdim+1} := \mP_{\rbdim} \cup \{\parameter_{\rbdim+1}\}$ and $\rbdim: = \rbdim + 1$. 
}
{\bf Pre-compute for online stage:} pre-compute by 
$\begin{cases}
      \text{Step 1 of Alg. \ref{alg:pG} \quad \tcc{for G-$L_1$ and G-Res}}\\
      \text{Alg. \ref{alg:pLS:offline} 
      \quad \tcc{for PG-$L_1$ and PG-Res}}\\
  \end{cases}$

{\bf Output:} Set $M = \rbdim$; 
Output 
$\orthorbsmat^{M}$, and 
$\begin{cases}
{\bArr}^q\in\mathR^{M\times M}, q=1,\dots, Q_A; {\bbrr}^q\in\mathR^{M}, q=1,\dots, Q_b\quad \text{\tcc{for G-$L_1$ and G-Res}}\\
\bY^q\in\mathR^{s\times M}, q=1,2,\dots, Q_A; \tilde{\bb}^q\in\mathR^{s}, q=1,2,\dots, Q_b 
\quad \text{\tcc{for PG-$L_1$ and PG-Res}}
\end{cases}$

\medskip
\Indm 
{\bf Online stage:} \Comment{to predict/test by  ROM($\rbTrial^{M};\parameter$)}

 \Indp
For each $\parameter \in \mP_{\textrm{interest}}=\mP_{\textrm{test}}$,  compute the reduced expansion coefficient $\reducedcoeff^{M}(\parameter)$ that defines the reduced solution $\angfluxvec_{RB}^{M}(\parameter) = \orthorbsmat^{M} \reducedcoeff^{M}(\parameter)$ by 
 $ \begin{cases}
      \text{Step 2 of Alg. \ref{alg:pG} \quad \tcc{for G-$L_1$ and G-Res}}\\
      \text{Alg. \ref{alg:pLS:online}  \quad \tcc{for PG-$L_1$ and PG-Res}}\\
  \end{cases}$
 with   $\rbdim = M$.
\end{algorithm}

The main cost of the offline stage comes from reduced basis generation, and we will examine the cost of each major step of its $\rbdim$-th greedy iteration. For Step i, the cost $
\mathcal{C}^{(\rbdim)}$ is summarized in Table \ref{tab:cost:RBM:m:stepi} and this is to run  ROM($\rbTrial^\rbdim;\parameter)$ and to compute $\Delta_\rbdim(\parameter)$ at $N_\parameter$ instances of the parameter, with $N_\parameter\leq |\mP_{\textrm{train}}|$. Some estimates, e.g.,  $2\rbdim Q_b\leq 2\rbdim Q_AQ_b\leq Q_b^2+\rbdim^2 Q_A^2$ are used, to simplify the bound of the cost.
For Step iii, the cost to compute one full order solve is $O(\mN N_{iter})=O^{(FOM)}(\mN N_{iter})$, as  discussed in Appendix \ref{app:impt:cost:FOM}, while for Step iv, the cost to orthogonalize a new column of $\rbsMat^{\rbdim+1}$ is $O(\mN)$. We here use $O^{(FOM)}$ to specifically denote and highlight the big-O from the full order solver. The cost of $O(\mN\rbdim^2)$ to compute the spectral ratio $r^{(\rbdim)}$ in Step v  will be excluded, given that $r^{(\rbdim)}$ does not certify the errors in reduced solutions and the total complexity is estimated for a fixed reduced dimension $M$ upon offline termination. Once reduced basis generation is completed, one also pre-computes for the online stage, with the termination reduced dimension $M$, and the cost  
$\mathcal{C}^{(M)}_{\textrm{pre}}$  will be
\begin{equation}
   \mathcal{C}^{(M)}_{\textrm{pre}} =
   \begin{cases}
     O\left(\mN(M^2 Q_A+MQ_b)\right)& \textrm{ for G-}L_1 \;\textrm{and G-Res}\\
     O\left(\mN(M^2 Q_A^2+Q_b^2)\right)& \textrm{ for PG-}L_1 \;\textrm{and PG-Res}\\       
   \end{cases}.
\end{equation}

\begin{table}[ht]
    \centering
    \begin{tabular}{l|c|c|c}
    \hline
    &ROM($\rbTrial^\rbdim;\parameter)$& $\Delta_\rbdim(\parameter)$& ${\mathcal C}^{(\rbdim)}=$ total cost of Step i of Alg. \ref{alg:entireAlg} in $\rbdim$-th iteration\\
    \hline
{\bf G}-$L_1$&\eqref{eq:CC:ROM-G}&\eqref{eq:L1 error indicator cost}&$O\big(\mN(\rbdim^2 Q_A+\rbdim Q_b)\big)+ O\big(N_{\parameter}(\rbdim^3+\rbdim^2 Q_A+\rbdim Q_b)\big)$\\
{\bf G-Res}&\eqref{eq:CC:ROM-G}&\eqref{eq:CC:ResC:G}&$O\big(\mN(\rbdim^2 Q_A^2+Q_b^2)\big)+ O\big(N_{\parameter}(\rbdim^3+\rbdim^2 Q_A^2+Q_b^2)\big)$\\
 {\bf PG}-$L_1$&\eqref{eq:CC:ROM-PG}&\eqref{eq:L1 error indicator cost}&$O\big(\mN (m^2Q_A^2+Q_b^2)\big)+O\big(N_{\parameter} (\rbdim^3 Q_A+\rbdim^2 (Q_A^2+Q_b)+Q_b^2\big)$
\\
{\bf PG-Res}&\eqref{eq:CC:ROM-PG}&\eqref{eq:CC:ResC:PG}&$O\big(\mN (m^2Q_A^2+Q_b^2)\big)+O\big(N_{\parameter} (\rbdim^3 Q_A+\rbdim^2 (Q_A^2+Q_b)+Q_b^2\big)$\\
\hline
    \end{tabular}
    \caption{The computational cost ${\mathcal C}^{(\rbdim)}$ of Step i of Alg. \ref{alg:entireAlg} during the $\rbdim$-th greedy iteration. }
    \label{tab:cost:RBM:m:stepi}
\end{table}

With all these, the total offline computational complexity of Alg. \ref{alg:entireAlg} to build the four RBM-based reduced order models ROM($\rbTrial^M;\cdot)$ of $M$-dimension will be
\begin{align}\label{eq:final:cost:off}
  &\mathcal{C}^{(M)}_{\textrm{offline}} =O^{(FOM)}(\mN N_{iter}M)+ O(\mN M)+ \mathcal{C}^{(M)}_{\textrm{pre}}+ \sum_{\rbdim=1}^{M-1} {\mathcal C}^{(\rbdim)} \notag\\
  &\hspace{0.37in}=O^{(FOM)}(\mN N_{iter}M)+ O(\mN M)\\
  +& 
  \scalebox{0.8}{$\begin{cases}O\big(\mN(((M-1)^3+M^2) Q_A+((M-1)^2+M) Q_b)\big)+ O\big(N_{\parameter}((M-1)^4+(M-1)^3 Q_A+(M-1)^2 Q_b)\big) & \textrm{for G-}L_1\\
  O\big(\mN((M-1)^3 Q_A^2+M^2 Q_A+(M-1) Q_b^2+MQ_b)\big)+ O\big(N_{\parameter}((M-1)^4+(M-1)^3 Q_A^2+ (M-1) Q_b^2)\big)&\textrm{for G-Res}\\
  O\big(\mN (((M-1)^3+M^2)Q_A^2+M Q_b^2)\big)+O\big(N_{\parameter}((M-1)^4 Q_A+(M-1)^3 (Q_A^2+Q_b)+(M-1) Q_b^2)\big) & \textrm{for PG-}L_1, \textrm{PG-Res}
 \end{cases}$
 }
 \notag
\end{align}
with $N_{\parameter}=|\mP_{\textrm{interest}}|$.  Finally, the online prediction cost to compute  reduced solutions at $N_{\parameter}=|\mP_{\textrm{test}}|$ instances of parameter is
\begin{equation}
   \mathcal{C}^{(M)}_{\textrm{online}} =
   \begin{cases}
     O\left(|\mP_{\textrm{test}}| (M^3+M^2 Q_A+M Q_b) \right)& \textrm{ for G-}L_1 \;\textrm{and G-Res}\\
     O\left(|\mP_{\textrm{test}}| (M^3 Q_A+M^2 (Q_A^2+Q_b)+Q_b^2)\right)& \textrm{ for PG-}L_1 \;\textrm{and PG-Res}\\       
   \end{cases}.
   \label{eq:final:cost:on}
\end{equation}

\begin{remark}
Though the theoretical estimates of the offline computational complexity depend on $M$ at a faster than linear rate, our experiments in Section \ref{sec:num:2D2v} show that the FOM solves {(with $\mN\gg 1$)} can dominate the computational cost, and a much lower growth rate in $M$ can be observed in practice throughout the offline greedy process until its termination.
\end{remark}
\begin{remark}
\label{rem:costComp:offline}
    One cannot directly compare the computational complexity of different ROMs based on \eqref{eq:final:cost:off}. Numerical comparisons will be performed instead. From Remark \ref{rem:cost:2indicators}, one can conclude that to build reduced order models of the same dimension, G-Res is relatively more costly than G-$L_1$, and PG-Res is relatively more costly than PG-$L_1$, and these will be confirmed numerically.  
\end{remark}

\section{Numerical examples}
\label{sec:num}

In this section, we demonstrate the performance of the four proposed RBM-based ROMs for their accuracy, efficiency, and robustness through a collection of physically relevant benchmark {1D examples in slab geometry (with $d_\x=1$, also see Section \ref{sec:num:1D1v})} and 2D2v ($d_\x=2, d_\vel=3$) examples.   In all numerical examples, the trial space in \eqref{eq:Uh} for the FOM consists of piecewise linear polynomials with $\pd=1$, and the FOM solutions $\angflux_h(\cdot;\parameter)$ are regarded as the ground truth.
To quantify the resolution of the reduced solutions computed by $\textrm{ROM}(\rbTrial^{\rbdim};\parameter)$ at a specific parameter value $\parameter$, we consider the following errors.
\begin{equation}
\mathcal{E}_{L_2,\rbdim}(\parameter)  = \| \angflux_h(\cdot;\parameter) - \angflux_{RB}^\rbdim(\cdot; \parameter) \|_h, \quad
\mathcal{E}_{Res,\rbdim}(\parameter)  = \|r_{h,\parameter}(\angflux_{RB}^\rbdim(\cdot;\parameter))\|_h.
\end{equation}
The former measures the $L_2$ error between the FOM and ROM solutions, and the latter measures the residual of a ROM solution. 

ROM errors will be reported during the offline training/learning stage and also during the online prediction/testing stage, and will be referred to as training errors and test errors, respectively. Recall that during the offline stage,   a nested sequence of reduced basis spaces,  $\rbTrial^{\rbdim}\subset \rbTrial^{\rbdim+1}$ ($m=1,2,3,\dots$), are built along with a hierarchical family of ROMs, $\textrm{ROM}(\rbTrial^{\rbdim};\parameter)$, with different resolutions/fidelity. 
This follows a greedy procedure and is guided by an error indicator, with a new parameter value $\parameter_{\rbdim+1}\in\mP_{\textrm{train}}\subset \mP$ 
{chosen} in each greedy iteration.
The $L_2$ and residual training errors are defined at $\parameter_{\rbdim+1}$, namely 
\begin{equation} \label{eq:training errors}
\mathcal{E}_{L_2,\rbdim}^{train} = \mathcal{E}_{L_2,\rbdim}(\parameter_{\rbdim+1}), \qquad
 \mathcal{E}_{Res,\rbdim}^{train} = \mathcal{E}_{Res,\rbdim}(\parameter_{\rbdim+1}). 
\end{equation}
They measure the resolutions of the reduced surrogate solvers $\textrm{ROM}(\rbTrial^{\rbdim};\cdot)$. How these errors change in reduced dimension $\rbdim$  informs the effectiveness of the error indicator and reduced surrogate solvers. 

The online $L_2$ and residual test errors below measure the predictability and generalizability of the ROMs for unseen parameter values from $\mP_{\textrm{test}}\subset \mP$ 
\begin{equation}\label{eq:test errors}
    \mathcal{E}_{L_2,\rbdim}^{test} = \max_{\parameter \in \mP_{\textrm{test}}} \mathcal{E}_{L_2,\rbdim}(\parameter),\qquad 
    \mathcal{E}_{Res,\rbdim}^{test} = \max_{\parameter \in \mP_{\textrm{test}}} \mathcal{E}_{Res,\rbdim}(\parameter)
\end{equation}
If the offline error indicators are effective, one would expect that both training and test errors of the same type (i.e. $L_2$ type or residual type) are comparable in {magnitude} and they generally decrease as $\rbdim$ grows. Faster decay of these errors in $\rbdim$ implies better efficiency of the respective $\textrm{ROM}(\rbTrial^{\rbdim};\cdot)$ with relatively smaller reduced dimension $\rbdim$.  Related to the affine form of the reduced matrices and data vectors in \eqref{eq:A:G:p} and \eqref{eq:A:PG:p},  $Q_b=1$ in all examples, while $Q_A=3$ for most examples, except for the line source example in Section \ref{sec:line source} ($Q_A=2$) and the pin-cell example in Section \ref{sec:pin-cell} ($Q_A=4$).  {In this work, the training set for each example is chosen deterministically, i.e., as uniform meshes in the parameter space $\mP$, with its actual size  based on extensive numerical experiments.}
Unless otherwise specified, the starting $\parameter_1$ in Alg. \ref{alg:entireAlg} is chosen as the geometric center of the parameter set $\mP$ that is also in $\mP_{\textrm{train}}$.

 \subsection{{1D slab geometry}}
\label{sec:num:1D1v}

We start with a set of experiments to study the proposed ROMs when they are applied to  a simplified {1D model in slab geometry},  
\begin{subequations}
\label{eq:1d:RTE}
\begin{align}
& v \partial_x \angflux(x,v) + \total(x) \angflux(x,v) = \scat(x) \macro(x) + \source(x), \quad (x,v) \in [x_L,x_R] \times [-1,1], \\
& \angflux(x_L,v) = g_L, \;v > 0, \quad \angflux(x_R,v) = g_R, \;v < 0, 
\end{align}
\end{subequations}
where $\macro(x) = \frac{1}{2} \int_{-1}^{1} \angflux(x,v) d v$.
The model is derived by assuming the angular flux only depends on the $x$ variable  with $d_\x=1$ (in the slab direction). $\angflux$ in \eqref{eq:1d:RTE} is in fact the average of the angular flux with respect to the $(y,z)$ cosine directions of $\vel$, while $v$ stands for the $x$ cosine of $\vel$.  The FOM discretization in Section \ref{sec:ROM} can be directly adapted to \eqref{eq:1d:RTE}. 
For all numerical examples, a uniform mesh of $N_x$ elements is used for the spatial domain $[x_L,x_R]$, and $N_v$-point Gauss-Legendre quadrature is applied in the $v$ direction, with $N_v=16$.
Besides, with the relatively low dimensionality of this simplified model, direct methods are used to solve the FOM.

\subsubsection{Spatially homogeneous material}
\label{sec:spatially homogeneous material}

In this subsection, we consider an example with spatially homogeneous material on the domain $[x_L, x_R]=[0,4]$, with constant source  $\source(x) = 0.01$ and zero inflow boundary conditions.  The parameters are the scattering and absorption cross sections, namely $\scat = \parameter_s$, $\absorp = \parameter_a$, with $\parameter=(\parameter_s, \parameter_a)\in \mP= [1,2] \times [5,6]$.   The training parameter set $\mP_{\textrm{train}}$ consists of $21 \times 21$ equally spaced points in $\mP$, and the test set  $\mP_{\textrm{test}}$ consists of $10 \times 10$   random points uniformly sampled from $\mP$. A uniform spatial mesh with $N_x = 80$ is used, and this corresponds to $\mN=(K+1)*N_x*N_v=2560$. We take $\text{tol}_{SRatio}=10^{-8}$ as the tolerance for the spectral ratio to terminate the offline stage of the RBM.

We implement the four ROMs, namely, G-$L_1$, G-Res, PG-$L_1$, PG-Res, with their $L_2$ and residual type training and test errors presented in Figure \ref{fig:1d Weak Scattering Strong Absorption Training and Test Errors}, and their spectral ratios $r^{(\rbdim)}$ in Figure \ref{fig:1d Spatially Homogeneous Material Training and Test Errors Scaled Residual and SR}-(a), all as a function of the reduced dimension $\rbdim$. In Figure \ref{fig:1d Spatially Homogeneous Material Points Picked in Greedy Search}, we further present the first 10 parameter points that are picked in the offline stage of the four ROMs. The sizes of the symbols encode the order of the greedy
selection, and this convention will be followed for the remaining numerical examples except in Figure \ref{fig:2d Lattice 20 Points Greedy Search}.
 The following is a summary of our observations.

\begin{enumerate}[label=(\roman*)]

\item All training and test errors generally decrease as the reduced dimension $\rbdim$ grows, indeed at exponential rates, showing that   the reduced surrogate solvers $\{\textrm{ROM}(\rbTrial^{\rbdim};\cdot)\}_{\rbdim}$ improve in resolution as the reduced trial space $\rbTrial^{\rbdim}$ is enriched.

\item The training errors are generally comparable to the corresponding test errors, confirming the good predictability and generalization capacity of the ROMs.

\item The number of iterations required for the spectral ratio to reach a certain level of tolerance 
in general is comparable across the four ROMs, evidencing the robustness of the spectral ratio as a stopping criterion.  It is important to note that the spectral ratio does not certify the errors, either $L_2$ type or residual type, of the ROM solutions. 
\item  The greedy strategy appears to favor ``extreme" points in the parameter space $\mP$, in the sense that all points beyond the initial $\parameter_1$, that is hand-picked as the center of $\mP$, are located on the boundary of $\mP$. In particular, greedy search picks points at or near the corners of $\mP$ during the early iterations.

\item It can be shown based on Remark \ref{rem:PG:monotone} that the training and test errors of the residual type by  PG-Res  will decrease monotonically. This is confirmed by Figure \ref{fig:1d Weak Scattering Strong Absorption Training and Test Errors}-(d). Though not guaranteed in general,  such monotonicity is observed for G-Res in this example.
\item In this example, $\absorp^\star=\inf_{x\in[x_L, x_R],\parameter\in\mP} \absorp(x; \parameter)=5>0$. Based on Proposition \ref{prop:aPostErr}, G-Res and PG-Res are certified once the residual-based error indicator 
is scaled by $(\absorp^\star)^{-1}$ with $\absorp^\star=5$. Moreover,  the scaled residual training (resp. test) errors are upper bounds of the $L_2$ training (resp. test) errors, namely,
\begin{equation}
    \mathcal{E}_{L_2,\rbdim}^{train}\leq \frac{1}{5}\mathcal{E}_{Res,\rbdim}^{train},\qquad \mathcal{E}_{L_2,\rbdim}^{test}\leq \frac{1}{5}\mathcal{E}_{Res,\rbdim}^{test},
    \label{eq:certified}
\end{equation}
and this is numerically confirmed by Figure \ref{fig:1d Spatially Homogeneous Material Training and Test Errors Scaled Residual and SR}-(b)(c).

\item 
On the other hand,  for G-$L_1$ and PG-$L_1$ when the $L_1$ error indicator is used, 
there are some dips in the training errors compared with the test errors in the relatively early iterations, and these dips are relatively bigger for PG-$L_1$. This seems to indicate that the best parameter points may not always be selected during these iterations. Another contributing factor can be our handpicked $\parameter_1$. As evidenced above, parameter values on the boundary of $\mP$ are preferred during the greedy iterations, our strategy of using the geometry center of $\mP$ as $\parameter_1$ likely has a more lingering effect on G-$L_1$ and PG-$L_1$ for this example as the greedy iterations progress.

\end{enumerate}

 It is important to note that the choice of both the error indicator (i.e. $L_1$ or residual type) and the projection type (i.e.  Galerkin or LS-Petrov-Galerkin type) contribute to each ROM. For this example, G-Res and PG-Res perform better in the sense that they are more consistent when improving the resolution as the reduced dimension $\rbdim$ increases, and the quality of the ROMs seem to be more sensitive to the choice of error indicators.

\begin{figure}[h!]
\centering{
\begin{subfigure}{0.45\textwidth}
\includegraphics[width=\linewidth]{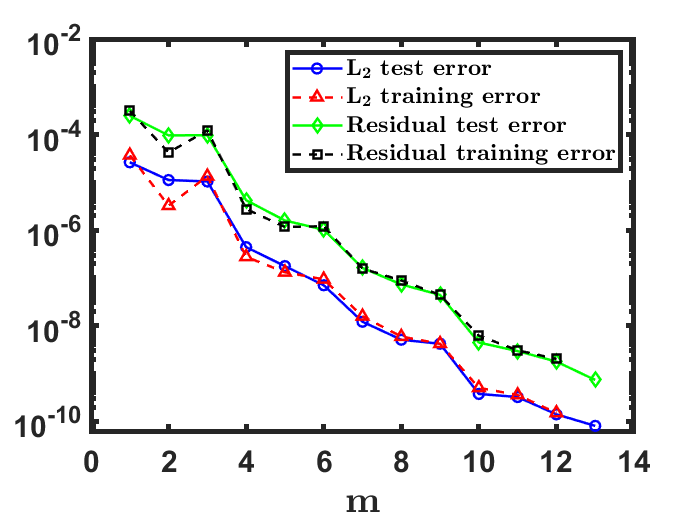}
\caption{}
\end{subfigure}
\begin{subfigure}{0.45\textwidth}
\includegraphics[width=\linewidth]
{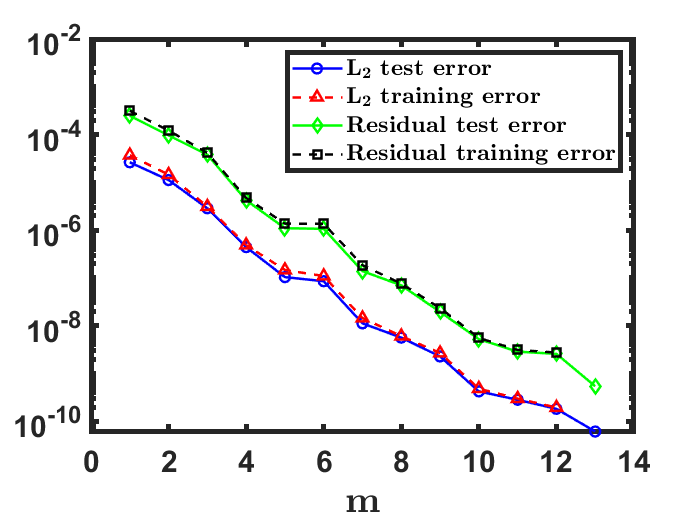}
\caption{}
\end{subfigure}

\begin{subfigure}{0.45\textwidth}
\includegraphics[width=\linewidth]
{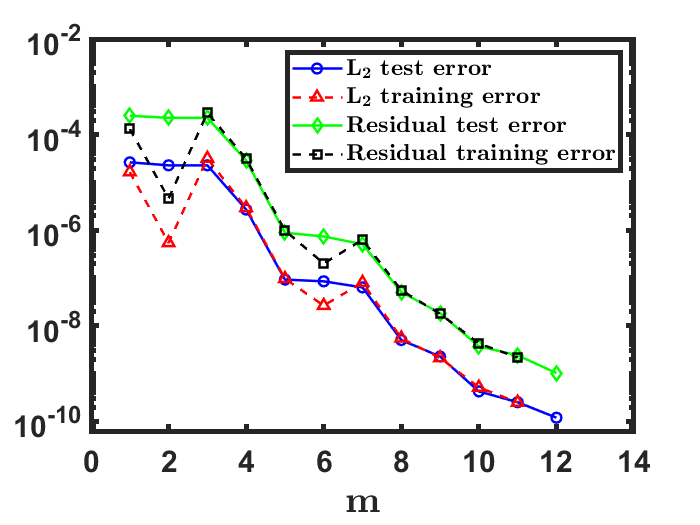}
\caption{}
\end{subfigure}
\begin{subfigure}{0.45\textwidth}
\includegraphics[width=\linewidth]
{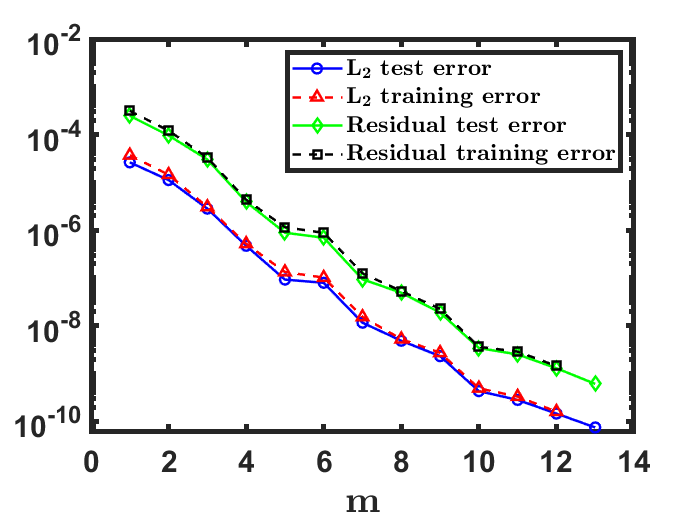}
\caption{}
\end{subfigure}
}
\caption{{1D slab geometry:} Spatially homogeneous material. Training and test errors using (a) G-$L_1$, (b) G-Res, (c) PG-$L_1$, (d) PG-Res.} 
\label{fig:1d Weak Scattering Strong Absorption Training and Test Errors}
\end{figure}

\begin{figure}[h!]
\centering{
\begin{subfigure}{0.3\textwidth}
\includegraphics[width=\linewidth]
{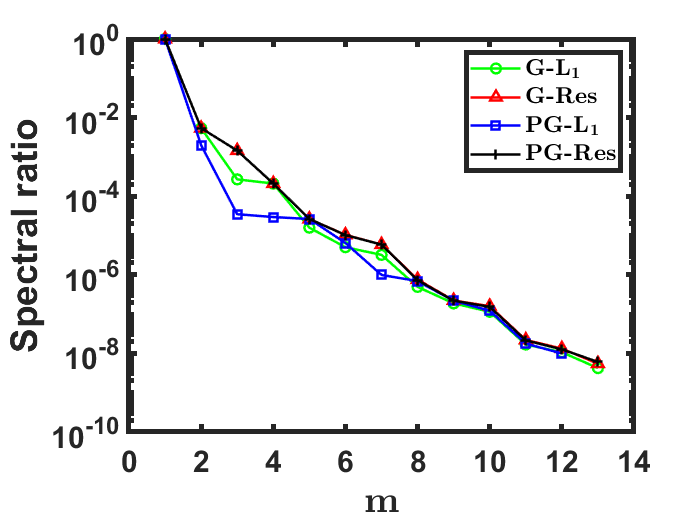}
\caption{}
\end{subfigure}
\begin{subfigure}{0.3\textwidth}
\includegraphics[width=\linewidth]
{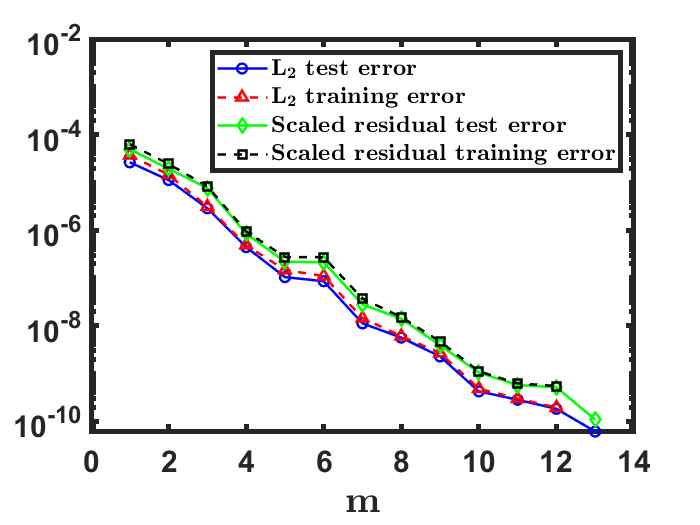}
\caption{}
\end{subfigure}
\begin{subfigure}{0.3\textwidth}
\includegraphics[width=\linewidth]
{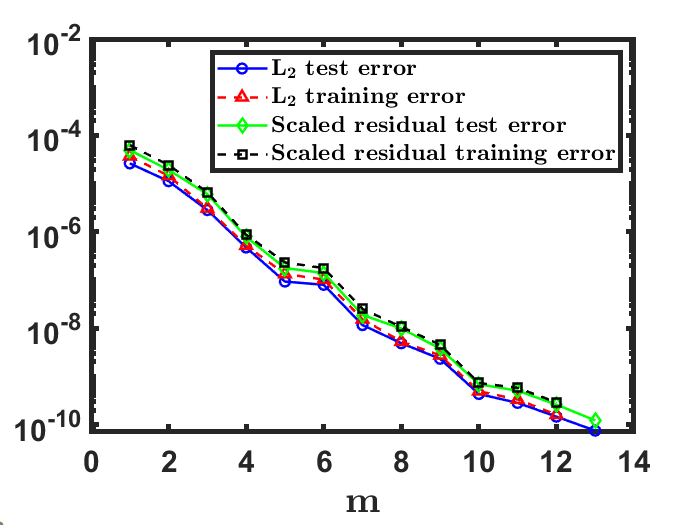}
\caption{}
\end{subfigure}
}
\caption{1D slab geometry: Spatially homogeneous material. (a) Histories of  the spectral ratio. $L_2$ training and test errors with ROM residual training and test errors scaled by $\frac{1}{\absorp^\star}$ with $\absorp^\star =5$ using (b) G-Res, (c) PG-Res. } 
\label{fig:1d Spatially Homogeneous Material Training and Test Errors Scaled Residual and SR}
\end{figure}

\begin{figure}[ht!]
\includegraphics[width=0.24\textwidth,height=1.3in]{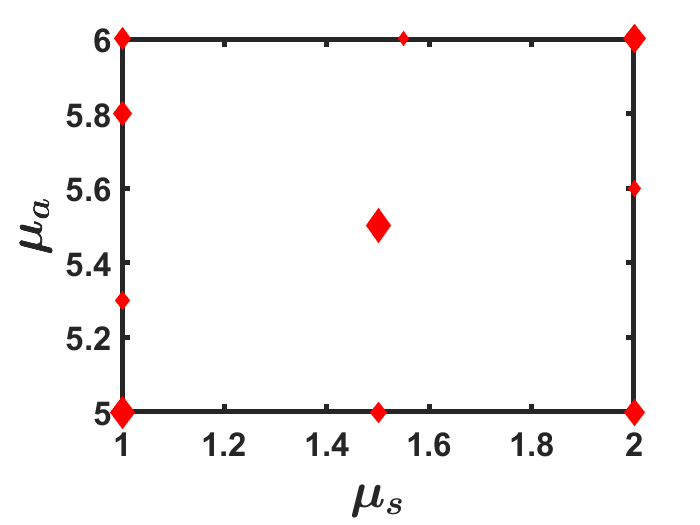}
\includegraphics[width=0.24\textwidth, height=1.3in]
{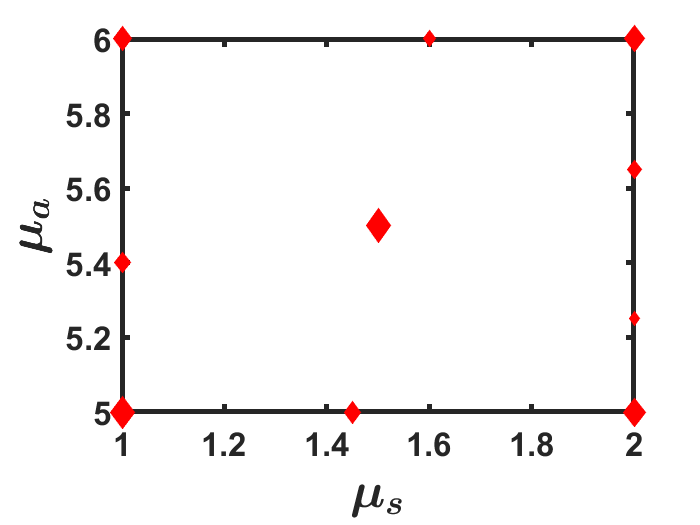}
\includegraphics[width=0.24\textwidth, height=1.3in]
{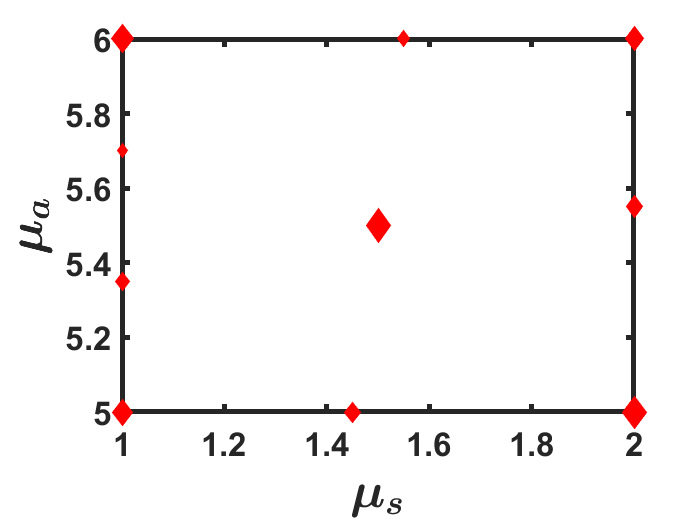}
\includegraphics[width=0.24\textwidth, height=1.3in]
{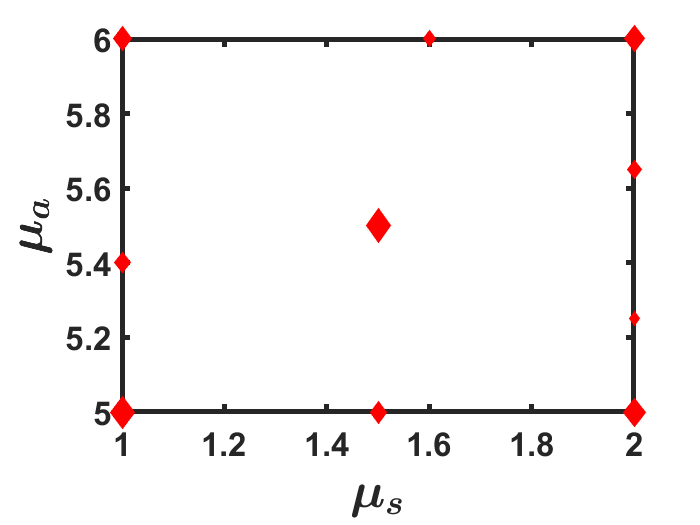}
\caption{1D slab geometry: Spatially homogeneous material. From left to right: first 10 parameter points picked in greedy search using G-$L_1$, G-Res,  PG-$L_1$, PG-Res, respectively.}
\label{fig:1d Spatially Homogeneous Material Points Picked in Greedy Search}
\end{figure}

\subsubsection{Two-material problem}
\label{sec:two-material}

In this subsection, we consider a two-material problem on the domain $[x_L, x_R]=[0,4]$ with zero source and the inflow boundary conditions $\angflux(x_L,v) = 5$ for $v > 0$, and $\angflux(x_R,v) = 0$ for $v < 0$.
The parameter $\parameter=(\parameter_s,\parameter_a)\in \mP = [90,100] \times [1,2]$, arising from the modeling of the scattering and absorption cross sections, namely,
\[
\scat(x; \parameter)
=
\begin{cases}
0, & 0 < x < 1 \\
\parameter_s, & 1 < x < 4
\end{cases}
,
\quad
\absorp(x; \parameter)
=
\begin{cases}
\parameter_a, & 0 < x < 1 \\
0, & 1 < x < 4
\end{cases}
.
\]
In this multi-scale problem, the background medium on $[0,1]$ is purely absorbing, and that on $[1,4]$ displays strong scattering. 
The training parameter set $\mP_{\textrm{train}}$ consists of $101 \times 21$ equally spaced points in $\mP$, and the test set $\mP_{\textrm{test}}$ consists of $10 \times 10$ randomly chosen points uniformly sampled from $\mP$.
 A uniform spatial mesh with $N_x = 120$ is used and this gives $\mN=3840$, and we take $\text{tol}_{SRatio}=10^{-10}$ as the tolerance for the spectral ratio to terminate the offline stage of the RBM.

Again, we implement all four ROMs and present their $L_2$ and residual type training and test errors in Figure \ref{fig:1d Two-Material Problem Training and Test Errors}, and their spectral ratios $r^{(\rbdim)}$ in Figure \ref{fig:1d TWo-Material Problem Spectral Ratio and Max Cond(ROM) and PG-Res alternative strategy}-(a), all as functions of the reduced dimension $\rbdim$. 
In Figure \ref{fig:1d Two-Material Problem Points Picked in Greedy Search}, we further present the first 10 parameter points that are picked in the offline stage of the four ROMs.   Some similar observations are made as the example in Section \ref{sec:spatially homogeneous material}, including the overall decreasing trends of the training and testing errors {at exponential  decay rates} as $\rbdim$ {increases} in Figure \ref{fig:1d Two-Material Problem Training and Test Errors},
the training and test errors of the same type {having comparable magnitudes} for each ROM,  the training and test errors of the residual
type by PG-Res decreasing monotonically, and the greedy selection favoring ``extreme" parameter values during the early iterations in Figure \ref{fig:1d Two-Material Problem Points Picked in Greedy Search}.  The following observations are particularly made for this example.

\begin{itemize}
    \item 
Among the four ROMs and with the same $\rbdim$, the training and test errors of the residual type are all comparable, while the training and testing errors of $L_2$ type by Galerkin based ROMs are relatively smaller, indicating {that} G-$L_1$ and G-Res are more accurate for this example.

\item The number of iterations required for the spectral ratio to reach a certain level of tolerance in general is comparable in the four ROMs, with PG-$L_1$ requiring 2-3 fewer iterations.

\item  For this example, G-$L_1$ ROMs and G-Res ROMs perform more {consistently}.  With PG-$L_1$, the training error displays some visible dips {in comparison to} the test errors of the same type (i.e. $L_2$ or residual type), indicating {that} the best parameter values are not  always selected. Such dips are also observed for PG-Res, yet only in the $L_2$ errors, while the residual errors in PG-Res  stay monotone in $\rbdim$ as discussed in  Remark \ref{rem:PG:monotone}.

\item For the two components of the parameter $\parameter=(\parameter_s,\parameter_a)$, though $\parameter_a$ has a narrower range, relatively more parameter values are selected along the $\parameter_a$ direction, evidencing that the solution manifold is more sensitive to this parameter component, at least during the early greedy selection.

\end{itemize}

\begin{figure}[ht!]
\centering{
\begin{subfigure}{0.45\textwidth}
\includegraphics[width=\linewidth]{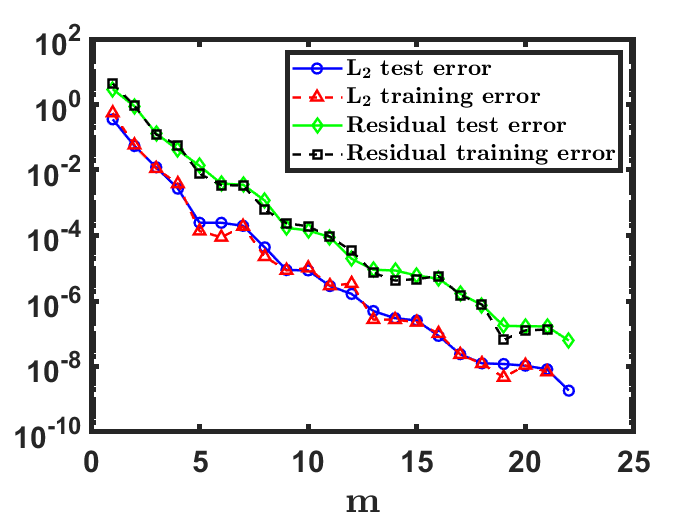}
\caption{}
\end{subfigure}
\begin{subfigure}{0.45\textwidth}
\includegraphics[width=\linewidth]{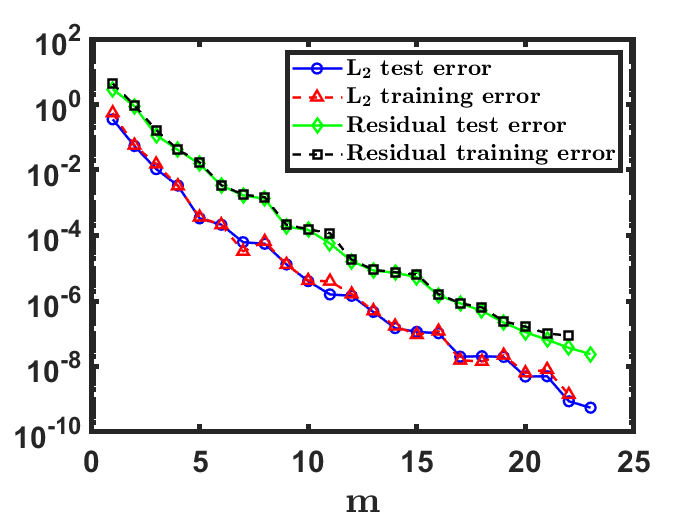}
\caption{}
\end{subfigure}

\begin{subfigure}{0.45\textwidth}
\includegraphics[width=\linewidth]{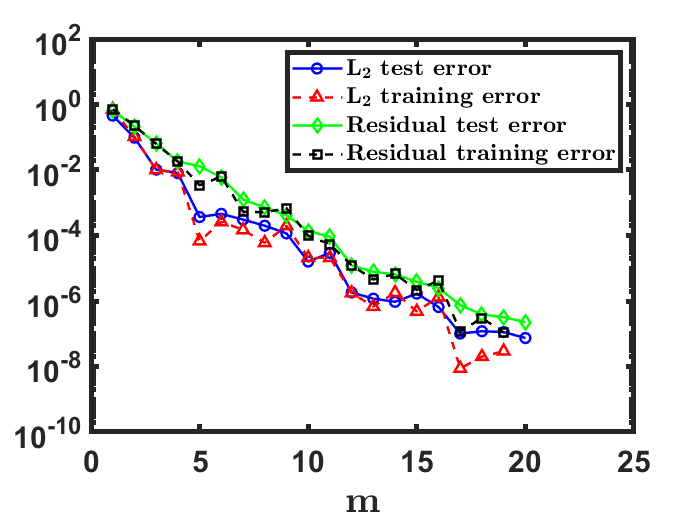}
\caption{}
\end{subfigure}
\begin{subfigure}{0.45\textwidth}
\includegraphics[width=\linewidth]{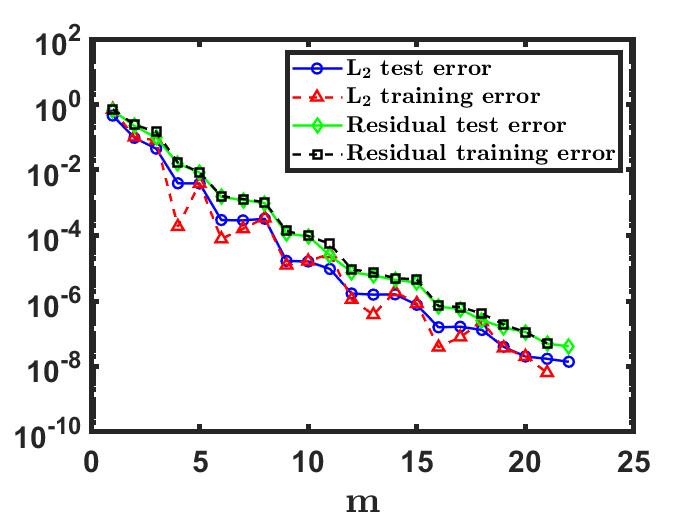}
\caption{}
\end{subfigure}
}
\caption{1D slab geometry: Two-material problem. Training and test errors using (a) G-$L_1$, (b) G-Res, (c) PG-$L_1$, (d) PG-Res.} 
\label{fig:1d Two-Material Problem Training and Test Errors}

\end{figure}

\begin{figure}[ht]
\centering{
\begin{subfigure}{0.3\textwidth}
\includegraphics[width=\linewidth]{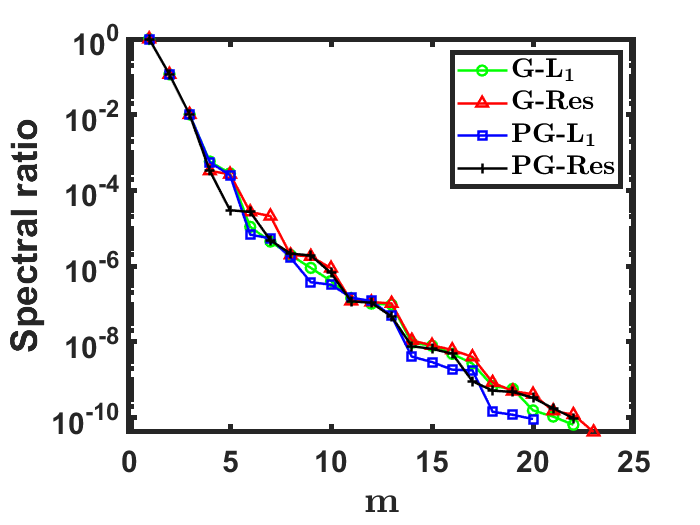}
\caption{}
\end{subfigure}
\begin{subfigure}{0.3\textwidth}
\includegraphics[width=\linewidth]{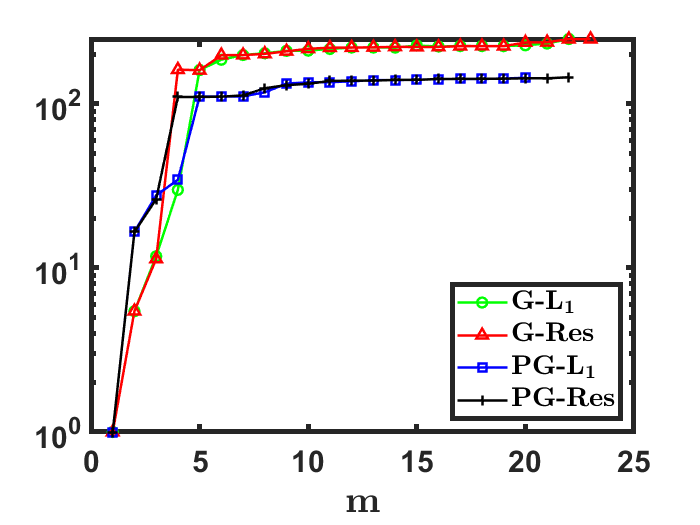}
\caption{}
\end{subfigure}
\begin{subfigure}{0.3\textwidth}
\includegraphics[width=\linewidth]{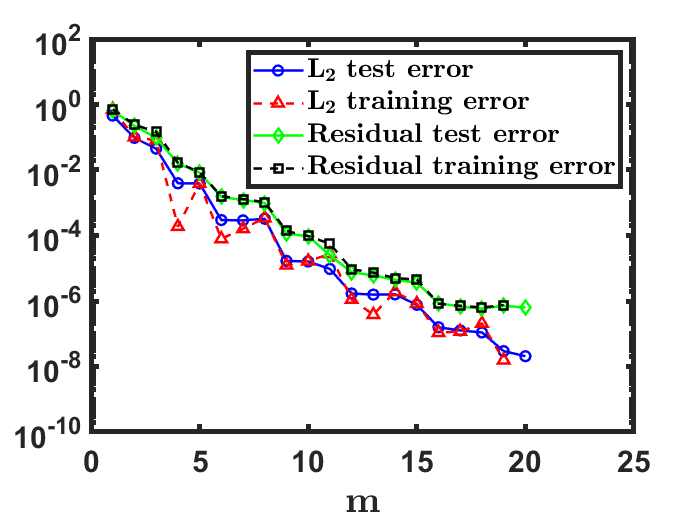}
\caption{}
\end{subfigure}
}
\caption{1D slab geometry: Two-material problem. (a) Histories of the spectral ratio. (b)  Maximum condition numbers of reduced system matrices over  test parameter set $\mP_{\textrm{test}}$. (c) Training and test errors using PG-Res and the variant implementation strategy Alg. \ref{alg:pLS:offline}$^\prime$ - Alg. \ref{alg:pLS:online} in Appendix \ref{app:variants:1} {and Eqn.}  \eqref{eq:Res-Alt1} to evaluate the residual.
}
\label{fig:1d TWo-Material Problem Spectral Ratio and Max Cond(ROM) and PG-Res alternative strategy}
\end{figure}

\begin{figure}[ht]
\includegraphics[width=0.24\textwidth,height=1.3in]
{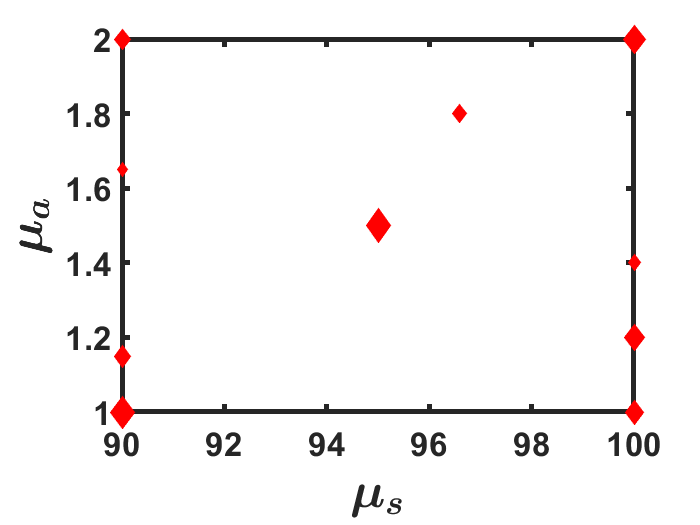}
\includegraphics[width=0.24\textwidth, height=1.3in]
{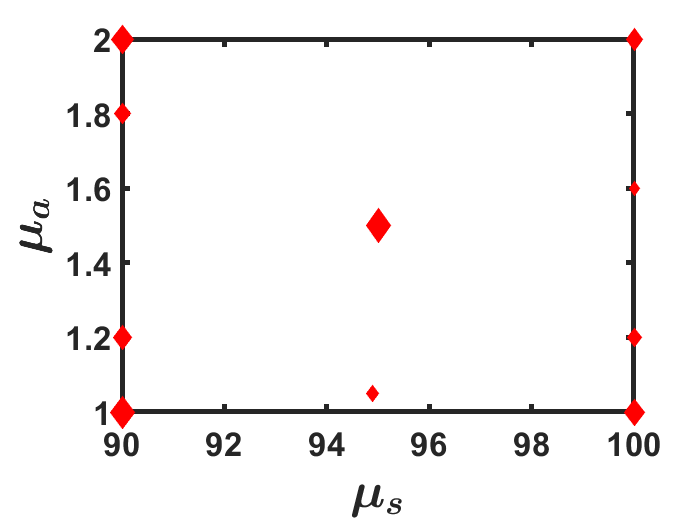}
\includegraphics[width=0.24\textwidth, height=1.3in]
{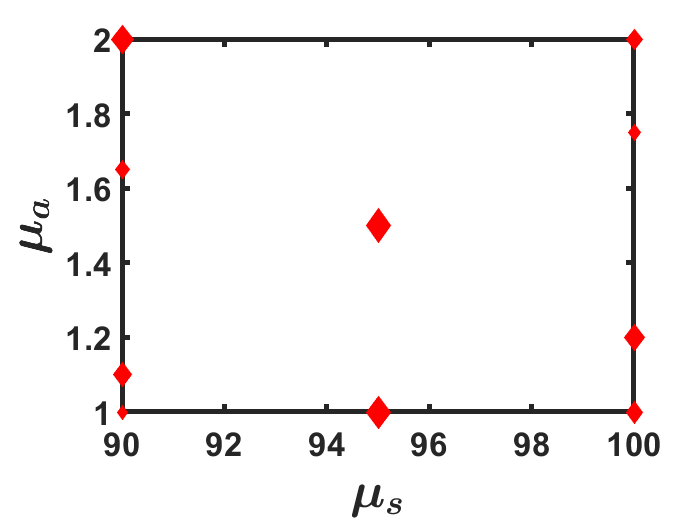}
\includegraphics[width=0.24\textwidth, height=1.3in]
{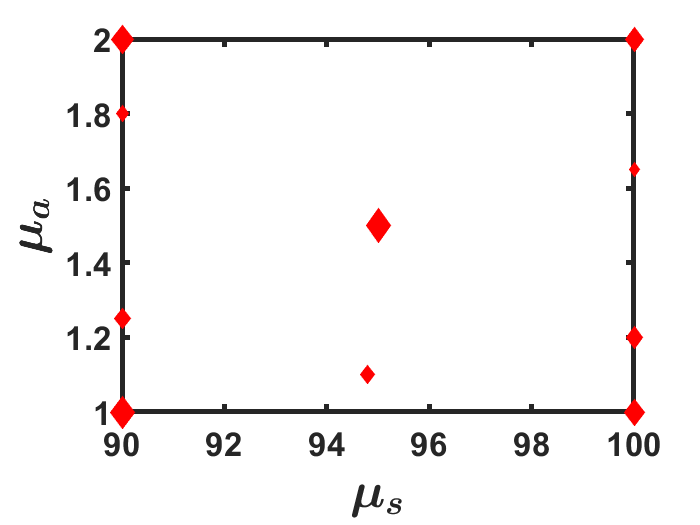}
\caption{1D slab geometry: Two-material problem. From left to right: first 10 parameter points picked in greedy search using G-$L_1$, G-Res,  PG-$L_1$, PG-Res, respectively.}
\label{fig:1d Two-Material Problem Points Picked in Greedy Search}
\end{figure}

For G-$L_1$ and G-Res, the gap between the $L_2$ and residual type errors is relatively larger. It is speculated that this is due to the difference in the conditioning of the reduced system matrices, and the supporting evidence can be found in Figure \ref{fig:1d TWo-Material Problem Spectral Ratio and Max Cond(ROM) and PG-Res alternative strategy}-(b), where we report 
the maximum condition numbers of the reduced system matrices over $\mP_{\textrm{test}}$ for each ROM, namely
$\max_{\parameter\in\mP_{\textrm{test}}}\textrm{Cond}_2({\bArr}_\parameter)$ for G-$L_1$ and G-Res, and $\max_{\parameter\in\mP_{\textrm{test}}}\textrm{Cond}_2({\bAr}_\parameter)$ for PG-$L_1$ and PG-Res, as $\rbdim$ grows. Recall ${\bArr}_\parameter$ is defined in \eqref{eq:SystemMatD-G} and ${\bAr}_\parameter$ is defined in \eqref{eq:SystemMatD-PG}.

In Figure \ref{fig:1d Two-Material Problem Rho Basis Functions}, we further plot the zero-th order moment, the reduced density $\macro$, of the first 5 (orthonormalized) reduced basis functions by the four ROMs.
With the reduced basis functions being problem-dependent, the profiles of their respective density display a sharp transition and capture the physical interior layer at the material interface $x=1$.  Note that 
the leading three $\macro$'s by all methods are comparable. This corresponds to the spectral ratio around $10^{-2}$.  For the fourth $\macro$, those by G-$L_1$ and PG-$L_1$ are similar, and those by G-Res and PG-Res are similar. Even though the reduced densities with smaller $\rbdim$ are close for all four ROMs, the training and test errors of these methods show greater difference, and this can be attributed to the difference of Galerkin and Petrov-Galerkin projections, and surely also the dependence of the methods more on the angular flux $\angflux$, not just on the macroscopic $\rho$.

\begin{figure}[ht]
\includegraphics[width=0.24\textwidth,height=1.5in]
{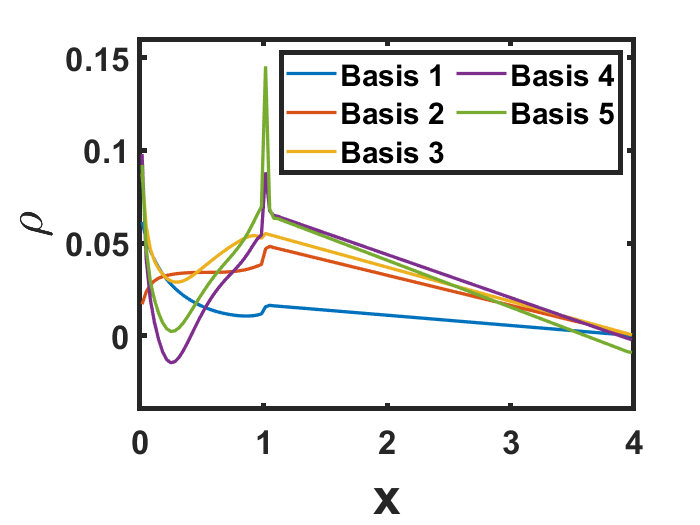}
\includegraphics[width=0.24\textwidth, height=1.5in]
{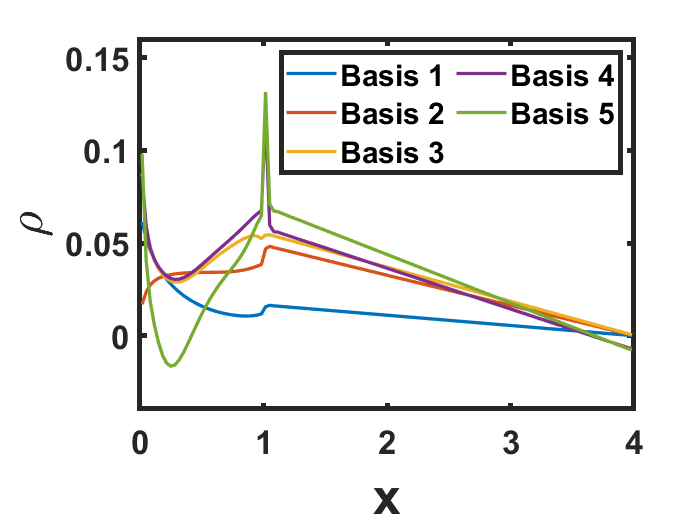}
\includegraphics[width=0.24\textwidth, height=1.5in]
{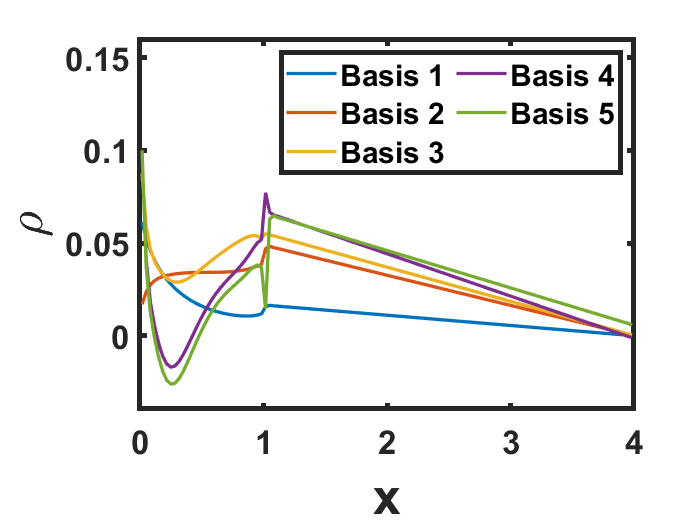}
\includegraphics[width=0.24\textwidth, height=1.5in]
{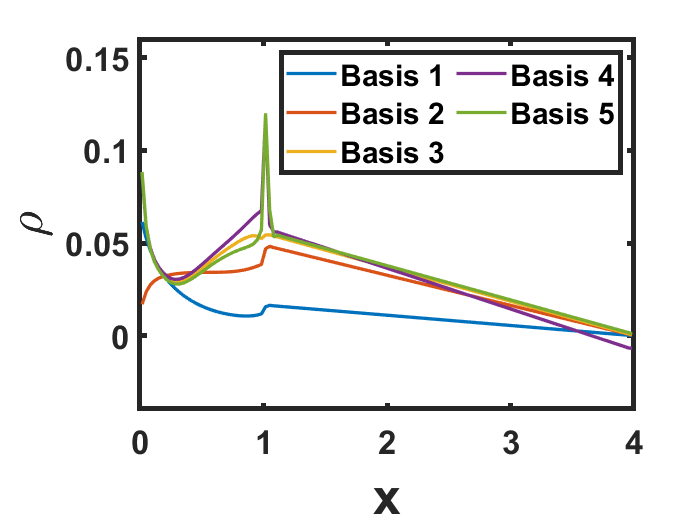}
\caption{1D slab geometry: Two-material problem. From left to right: the reduced density $\macro$ of the first 5 (orthonormalized) reduced basis functions  using  G-$L_1$,  G-Res,  PG-$L_1$,  PG-Res, respectively.}
\label{fig:1d Two-Material Problem Rho Basis Functions}
\end{figure}

Finally, we want to demonstrate the performance of PG-Res implemented by a variant strategy,  consisting of Alg. \ref{alg:pLS:offline}$^\prime$ - Alg. \ref{alg:pLS:online} in  Appendix \ref{app:variants:1} (for reduced order solvers) and the residual evaluation (in both the error indicator and the residual error computation) {based on Eqn. \eqref{eq:Res-Alt1}.} The resulting training and test errors are reported in Figure \ref{fig:1d TWo-Material Problem Spectral Ratio and Max Cond(ROM) and PG-Res alternative strategy}-(c). It is observed that the residual errors plateau when $\rbdim \geq 15$. In contrast, the residual errors in Figure \ref{fig:1d Two-Material Problem Training and Test Errors}-(d) continue to decrease when $\rbdim \geq 15$. This shows that our chosen implementation strategy suffers less from the loss of significance, especially for ROMs to achieve high resolution, and is hence more numerically robust. On the other hand, the residual stagnation is not of a concern if one wants a reduced order model with residual errors of magnitude not lower than $10^{-6}$. In such a situation, the variant considered here makes a good efficient strategy to implement PG-Res.

\subsubsection{Spatially varying scattering}
\label{sec:spatially varying scattering}

In this subsection, we consider an example with a variable scattering cross section on the domain $[x_L, x_R]= [0,4]$ with the constant source $\source(x) = 0.01$ and zero absorption cross section and zero inflow boundary conditions.
The parameter $\parameter=(\parameter_s^{(1)}, \parameter_s^{(2)})\in \mP=[90,100] \times [90,100]$,  modeling the spatially varying scattering cross section, namely,
$\scat(x; \parameter) = \parameter_s^{(1)} + \parameter_s^{(2)} x.$
%
As the parameter varies, the scattering cross section $\scat$ {varies widely} from $90$ to $500$. 
The training parameter set $\mP_{\textrm{train}}$ consists of $101 \times 101$ equally spaced points in $\mP$, and the test set $\mP_{\textrm{test}}$ consists of $10 \times 10$ randomly chosen points uniformly sampled from $\mP$.
A uniform spatial mesh with $N_x = 80$ is used and this gives $\mN = 2560$, and we take $\text{tol}_{SRatio}=10^{-14}$ as the tolerance for the spectral ratio.
Note that this error tolerance is much smaller than that in the previous two examples, and this allows us to better observe the difference of the proposed ROMs for this multi-scale example.

For all four ROMs, the training and test errors of $L_2$ and residual types are presented in Figure \ref{fig:1d Spatially Varying Scattering Training and Test Errors}, and the spectral ratios are presented in Figure \ref{fig:1d Spatially Varying Scattering Spectral Ratio and Max Cond(ROM) and PG-Res alternative strategy}-(a), all as a function of the reduced dimension $\rbdim$. In Figure \ref{fig:1d Spatially Varying Scattering Spectral Ratio and Max Cond(ROM) and PG-Res alternative strategy}-(b), we also report the maximum condition numbers of the reduced system matrices over $\mP_{\textrm{test}}$ for each ROM as $\rbdim$ grows. Some similar observations are made as for the previous examples, regarding the overall
decreasing trends of the training and testing errors as $\rbdim$ grows,  the training and test errors of the same
type being at comparable magnitude for each ROM, and the training and test errors of the residual type by
PG-Res decreasing monotonically. The following observations are particularly made for this example.

\begin{itemize}
\item The most striking distinction among the four ROMs is the fast decay of the training and test errors, especially the $L_2$ type, for the LS-Petrov-Galerkin based ROMs (see Figure \ref{fig:1d Spatially Varying Scattering Training and Test Errors}-(c)(d)). For example, for the $L_2$ type errors to reach $10^{-6}$, both PG-$L_1$ and PG-Res require about 6-7 iterations, while  G-$L_1$ and G-Res require about 12-13 iterations. 
 The observation here in the training and test errors by different methods is not entirely informed by the respective spectral ratio in Figure \ref{fig:1d Spatially Varying Scattering Spectral Ratio and Max Cond(ROM) and PG-Res alternative strategy}-(a). Nevertheless, the spectral ratio of a specific ROM with its monotonic trend still informs about the richness of the reduced trial space of this ROM.

\item Upon termination, the test $L_2$ errors are at magnitude $10^{-9}$ or $10^{-8}$. PG-Res reaches the terminal error sooner and plateaus after about 10 iterations. The attainable error level 
is closely related to  the conditioning of the reduced systems. Specifically, the maximum of the 2-norm condition numbers of the reduced system matrices $\bAr_\parameter$ and $\bArr_\parameter$ over the test set $\mP_{\textrm{test}}$ is at magnitude $10^{5}$ or $10^{6}$ (see Figure \ref{fig:1d Spatially Varying Scattering Spectral Ratio and Max Cond(ROM) and PG-Res alternative strategy}-(b)). Note {that} all simulations are carried out in double precision environment. While the $L_2$ errors reach a plateau for PG-$L_1$ and PG-Res, the respective residual errors continue to decay.

\item    For the same reduced dimension $\rbdim$, LS-Petrov-Galerkin based ROMs (bottom row of Figure \ref{fig:1d Spatially Varying Scattering Training and Test Errors}) overall lead to smaller residual training and test errors than Galerkin-based ROMs.

\end{itemize}

With the condition numbers of the reduced system matrices being of large magnitudes as in Figure \ref{fig:1d Spatially Varying Scattering Spectral Ratio and Max Cond(ROM) and PG-Res alternative strategy}-(b), it is crucial to solve the parametric least-squares problem \eqref{eq:rom:alg:p:PG} using the QR-based algorithms proposed in this work, instead of any procedure based on the normal equation in \eqref{eq:rom:alg:abs}-\eqref{eq:rom:Wh:PG}, to avoid working with reduced systems with condition number $10^{11}$.  
Similarly, it is important not to evaluate the residual for G-Res using the variant strategy in Appendix \ref{app:variant:3}.

\begin{figure}[ht]
\centering{
\begin{subfigure}{0.45\textwidth}
\includegraphics[width=\linewidth]{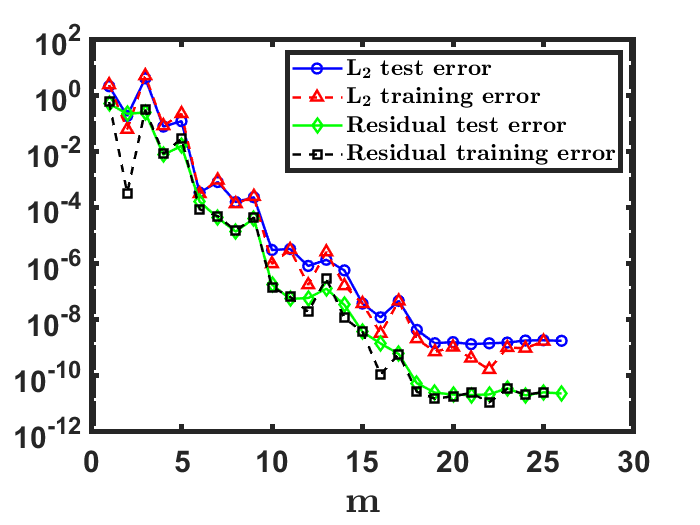}
\caption{}
\end{subfigure}
\begin{subfigure}{0.45\textwidth}
\includegraphics[width=\linewidth]{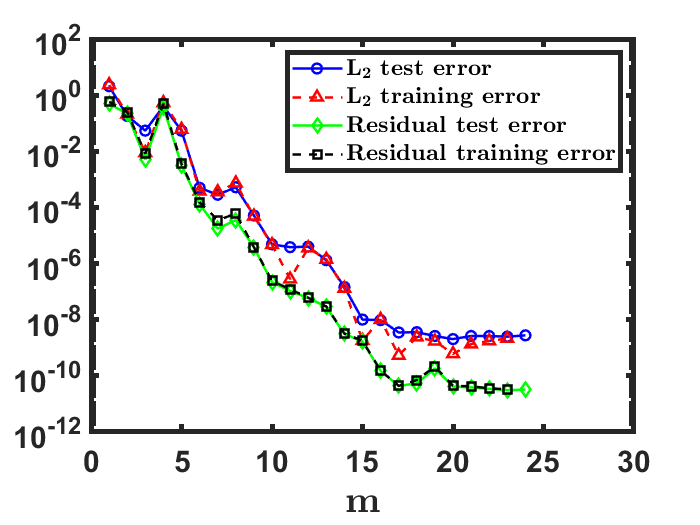}
\caption{}
\end{subfigure}

\begin{subfigure}{0.45\textwidth}
\includegraphics[width=\linewidth]{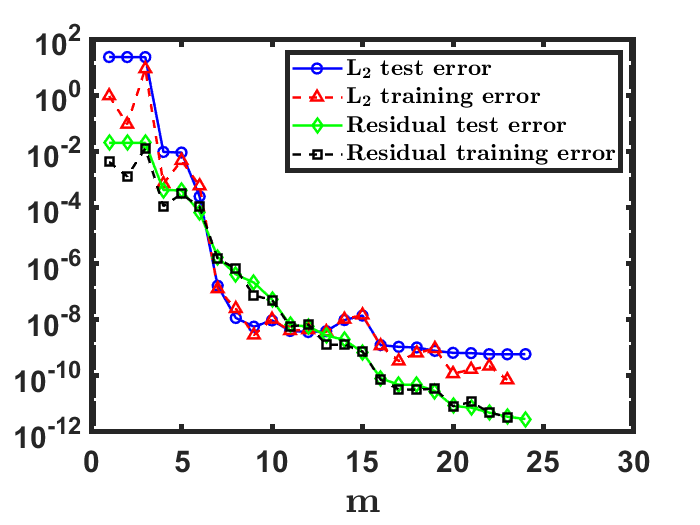}
\caption{}
\end{subfigure}
\begin{subfigure}{0.45\textwidth}
\includegraphics[width=\linewidth]{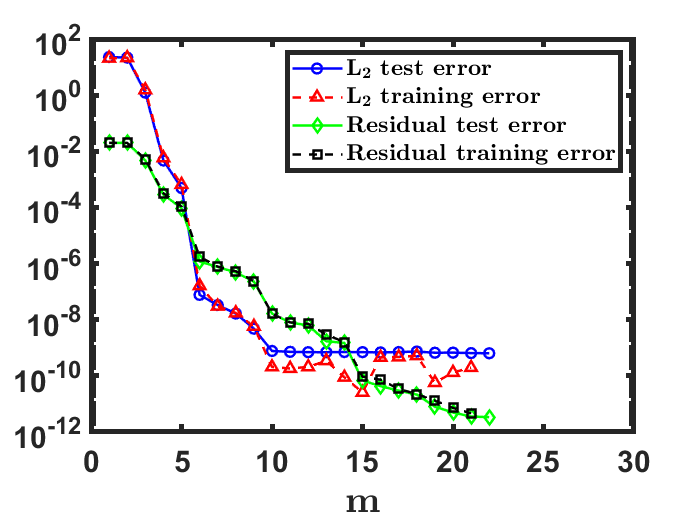}
\caption{}
\end{subfigure}
}
\caption{1D slab geometry: Spatially varying scattering. Training and test errors using (a) G-$L_1$, (b) G-Res, (c) PG-$L_1$, (d) PG-Res.}
\label{fig:1d Spatially Varying Scattering Training and Test Errors}

\end{figure}

\begin{figure}[h]
\centering{
\begin{subfigure}{0.3\textwidth}
\includegraphics[width=\linewidth]{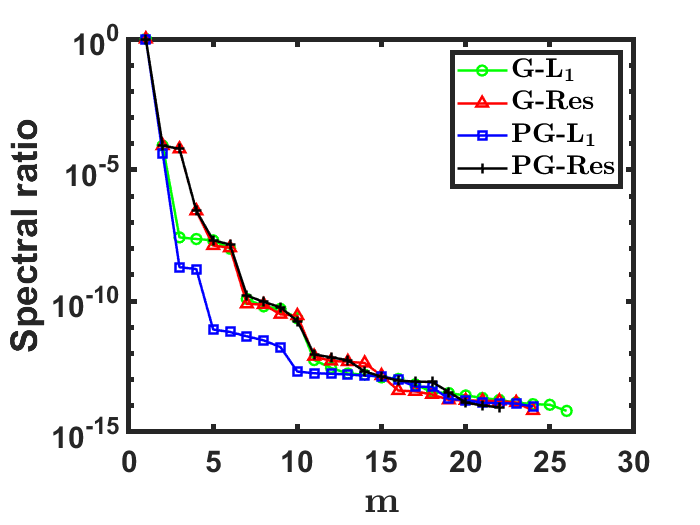}
\caption{}
\end{subfigure}
\begin{subfigure}{0.3\textwidth}
\includegraphics[width=\linewidth]{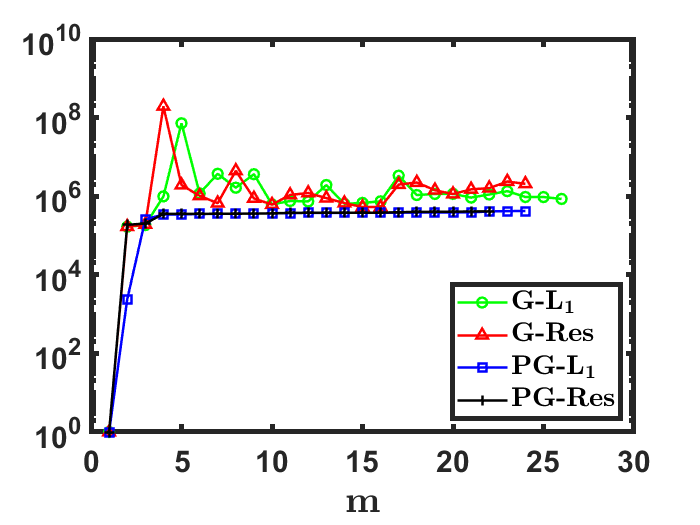}
\caption{}
\end{subfigure}
\begin{subfigure}{0.3\textwidth}
\includegraphics[width=\linewidth]{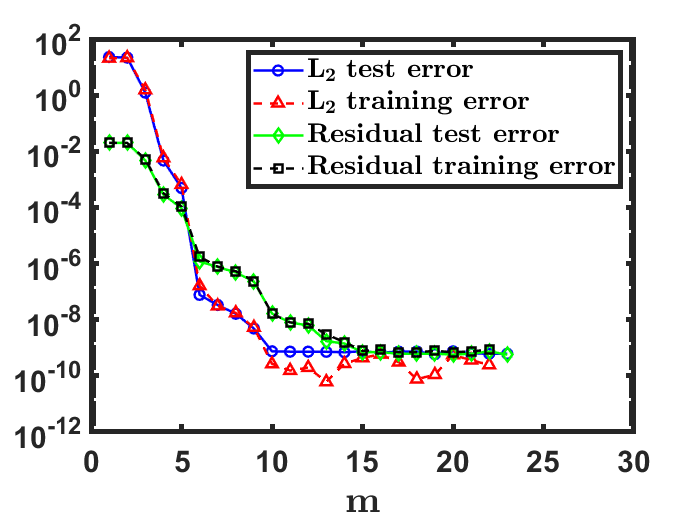}
\caption{}
\end{subfigure}
}
\caption{1D slab geometry: Spatially varying scattering. (a) Histories of the spectral ratio. (b) Maximum condition numbers of reduced system matrices over test parameter set $\mP_{\textrm{test}}$. (c)  Training and test errors using PG-Res and the variant implementation strategy Alg. \ref{alg:pLS:offline}$^\prime$ - Alg. \ref{alg:pLS:online} in Appendix \ref{app:variants:1} {and Eqn. \eqref{eq:Res-Alt1}} to evaluate the residual.
}
\label{fig:1d Spatially Varying Scattering Spectral Ratio and Max Cond(ROM) and PG-Res alternative strategy}
\end{figure}

Figure \ref{fig:1d Spatially Varying Scattering Points Picked in Greedy Search} shows the first 10 parameter points picked in the offline stage of the four ROMs. 
In both G-Res and PG-Res, extreme points are chosen early on. It is relatively less so for G-$L_1$ and PG-$L_1$. For PG-$L_1$, the first few selected parameter points all have their second component $\parameter_s^{(2)} \approx 95$, and the second selected point almost coincides with the initial $\parameter_1$. All of these may partially contribute to the dips in the training errors compared with the test errors of this ROM during early iterations (see Figure \ref{fig:1d Spatially Varying Scattering Training and Test Errors}-(c)).

\begin{figure}[ht]
\centering{
\includegraphics[width=0.24\textwidth, height=1.3in]
{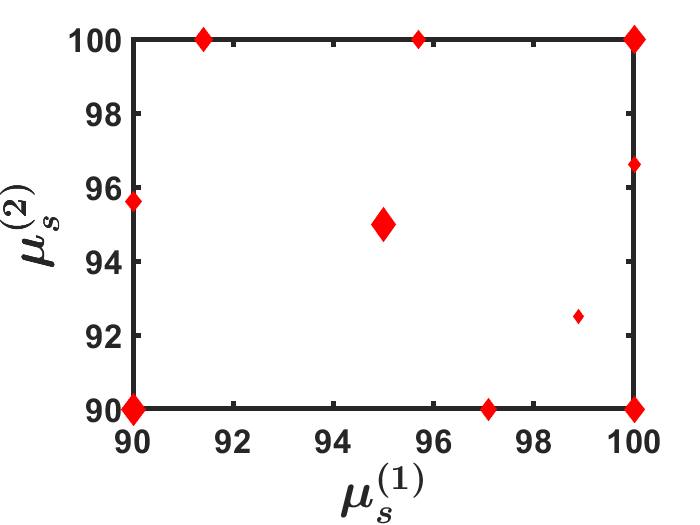}
\includegraphics[width=0.24\textwidth, height=1.3in]{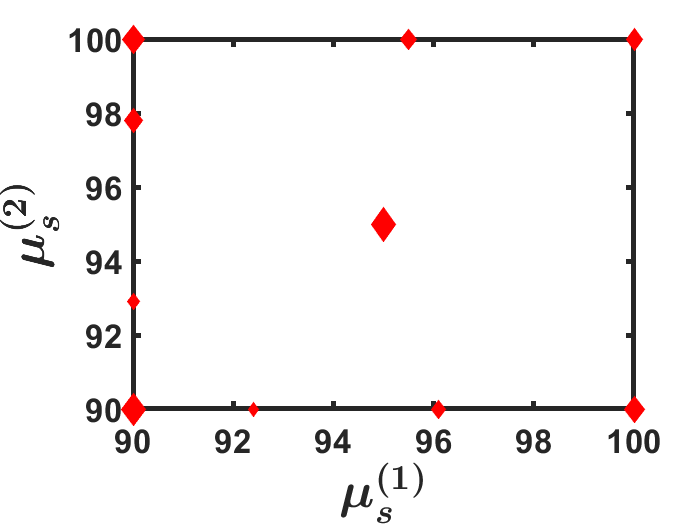}
\includegraphics[width=0.24\textwidth, height=1.3in]
{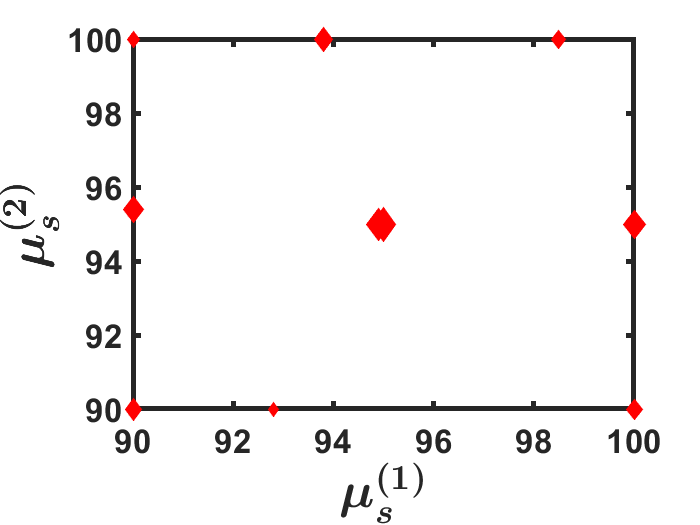}
\includegraphics[width=0.24\textwidth, height=1.3in]
{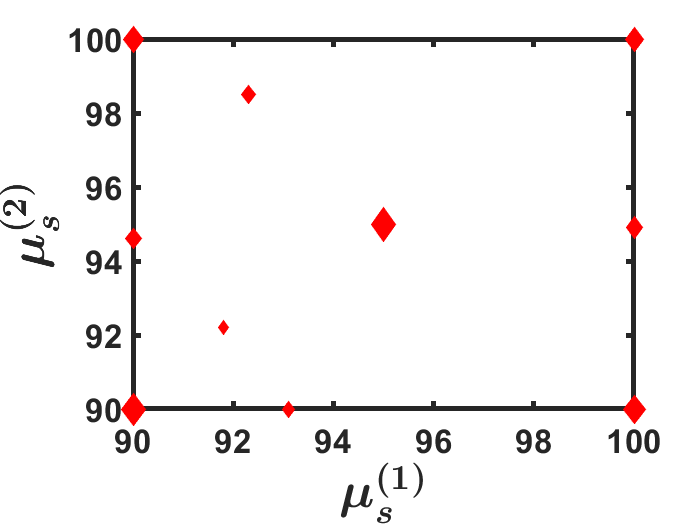}
}
\caption{1D slab geometry: Spatially varying scattering. From left to right: first 10 parameter points picked in greedy search using G-$L_1$, G-Res,  PG-$L_1$, PG-Res, respectively.}
\label{fig:1d Spatially Varying Scattering Points Picked in Greedy Search}

\end{figure}

For this example, we also demonstrate {the performance of} PG-Res implemented by a variant strategy, consisting of Alg. \ref{alg:pLS:offline}$^\prime$ - Alg. \ref{alg:pLS:online} from  Appendix  \ref{app:variants:1} and  the residual evaluation (in both the error indicator and the residual error computation) {based on Eqn. \eqref{eq:Res-Alt1}.}   The resulting training and test errors are reported in Figure   \ref{fig:1d Spatially Varying Scattering Spectral Ratio and Max Cond(ROM) and PG-Res alternative strategy}-(c). Very similar observations as in the previous section can be made, namely the residual errors plateau when $\rbdim \geq 15$, while those  in Figure \ref{fig:1d Spatially Varying Scattering Training and Test Errors}-(d) continue to decrease for such $\rbdim$. This shows that our chosen implementation strategy is numerically more robust. This residual error stagnation, occurring at magnitude $10^{-9}$, is not a practical concern.

Overall, for this example, all four ROMs perform well, with PG-Res most preferable due to its robustness and the relatively low reduced dimension needed to achieve high resolution/fidelity.

 \subsection{2D2v examples}
\label{sec:num:2D2v}
We next study the performance of the proposed ROMs for the 2D2v model (with $d_{\x} = 2$, $d_{\vel} = 3$) in which $\x = (x,y) \in \xset = [x_L,x_R] \times [y_L,y_R]$.
For all numerical examples, a uniform mesh of $N_x \times N_y$ elements is used for the spatial domain where $N_x$ and $N_y$ denote the number of elements in the $x$ and $y$ direction, respectively.
As discussed in Appendix \ref{app:impt:cost:FOM}, the FOM will be solved iteratively by the source iteration scheme accelerated by diffusion synthetic acceleration (SI-DSA). Unless otherwise specified, the zero initial guess is taken for $\macrovec$ along with $tol_{SI} = 10^{-12}$
as the stopping tolerance for the iterative solver. Table \ref{tab:average iterations SI-DSA:rho0} summarizes the average number\footnote{The average number of iterations is computed over the test parameter set $ \mP_{\textrm{test}}$ specified in each example.} of iterations required for SI-DSA to converge when solving the examples presented in this section. 

\begin{table}[!h] 
\centering 
\caption{Average number of iterations required for SI-DSA to converge, 
with the zero initial guess $\macrovec=\mathbf{0}$ and $tol_{SI} = 10^{-12}$. Lattice: $N_x = N_y = 70$ with $(40,6)$ CL-quadrature. Line Source and Pin-Cell: $N_x = N_y = 80$ with $(30,6)$ CL-quadrature.} 
\begin{tabular}{|c|c|c|}\hline 
19 (Lattice) & 17.15 (Line Source) & 40.99 (Pin-Cell)  \\
\hline
\end{tabular}
\label{tab:average iterations SI-DSA:rho0}
\end{table}

\subsubsection{Lattice example}
\label{sec:lattice}

In this subsection, we consider a lattice example on $\xset = [-3.5,3.5]\times [-3.5,3.5]$ with zero inflow boundary conditions.
Figure \ref{fig:2d Lattice Layout-2d Lattice Spectral Ratio}-(a) displays how $\scat(\x;\parameter)$, $\absorp(\x;\parameter)$, $\source(\x)$ are defined in this example. 
In particular, the scattering and absorption cross sections are parametrized by $\parameter=(\parameter_s,\parameter_a)$, namely, 
\[
(\scat(\x;\parameter),\absorp(\x;\parameter))
=
\begin{cases}
(0,\parameter_a) & \textrm{in black region}\\
(\parameter_s,0) & \textrm{elsewhere}
\end{cases},
\]
with $(\parameter_s,\parameter_a) \in \mP = [0.5,1.5] \times [8,12]$.
The training set $\mP_{\textrm{train}}$ consists of $21 \times 21$ equally spaced points in $\mP$ and the test set $\mP_{\textrm{test}}$ consists of $10 \times 10$ randomly chosen points uniformly sampled from $\mP$.
Unless otherwise specified, a uniform spatial mesh with $N_x = N_y = 70$ is used with $(40,6)$-CL quadrature for the angular discretization, and this corresponds to $\mN = (K+1)^2*N_x*N_y*N_{\theta}*N_{\xi} = 4.704\times 10^6$. We take $\text{tol}_{SRatio}=10^{-9}$ 
to terminate the offline stage of the RBM.

\begin{figure}[ht]
\centering{
\begin{subfigure}{0.4\textwidth}
\includegraphics[width=\linewidth]{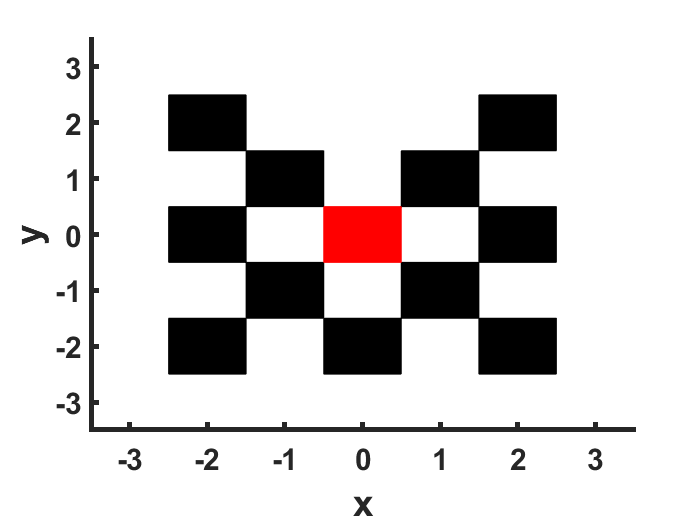}
\caption{}
\end{subfigure}
\begin{subfigure}{0.4\textwidth}
\includegraphics[width=\linewidth]{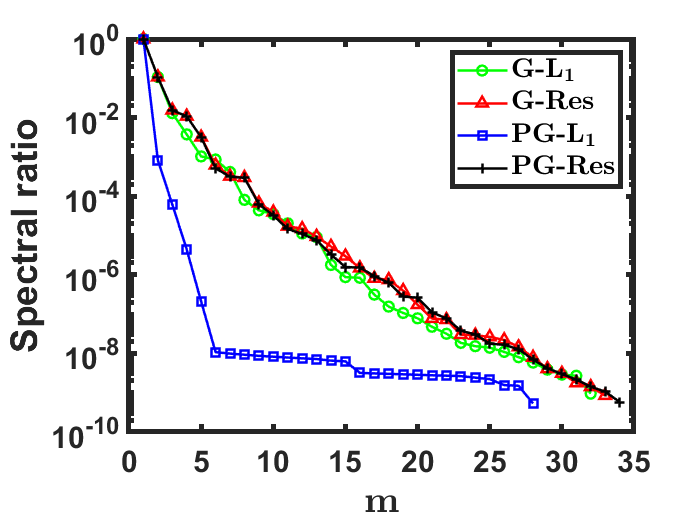}
\caption{}
\end{subfigure}
}
\caption{2D2v example: Lattice. (a) Material property.  Black: $(\scat,\absorp) = (0,\parameter_a)$, $\source(\x) = 0$. White: $(\scat,\absorp) = (\parameter_s,0)$, $\source(\x) = 0$. Red: $(\scat,\absorp) = (\parameter_s,0)$, $\source(\x) = 1$. (b) Histories of the spectral ratio.}
\label{fig:2d Lattice Layout-2d Lattice Spectral Ratio}
\end{figure}

 We implement all four ROMs, namely, G-$L_1$, G-Res, PG-$L_1$, PG-Res, with their $L_2$ and residual type training and test errors presented in Figure \ref{fig:2d Lattice Training and Test Errors}, and their spectral ratios $r^{(\rbdim)}$ in  Figure \ref{fig:2d Lattice Layout-2d Lattice Spectral Ratio}-(b),
all as a function of the reduced dimension $\rbdim$.
In Figure \ref{fig:2d Lattice Points Picked in Greedy Search}, we further present the first 10 parameter points that are picked in the offline stage of the four ROMs.

\begin{figure}[ht]
\centering{
\begin{subfigure}{0.45\textwidth}
\includegraphics[width=\linewidth]{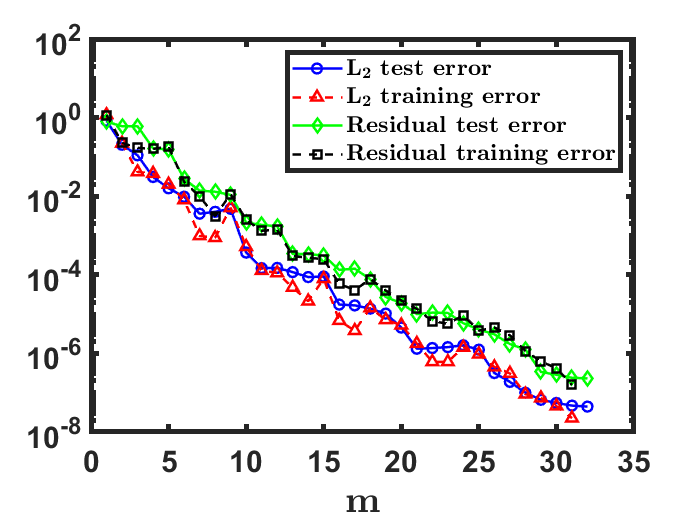}
\caption{}
\end{subfigure}
\begin{subfigure}{0.45\textwidth}
\includegraphics[width=\linewidth]{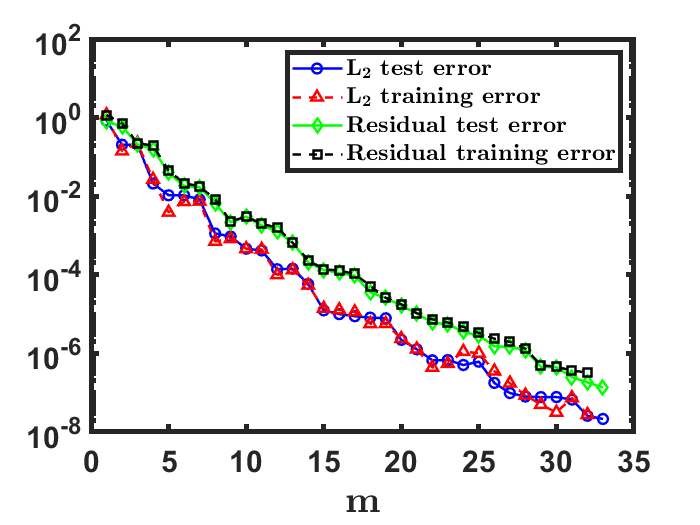}
\caption{}
\end{subfigure}
}

\begin{subfigure}{0.45\textwidth}
\includegraphics[width=\linewidth]{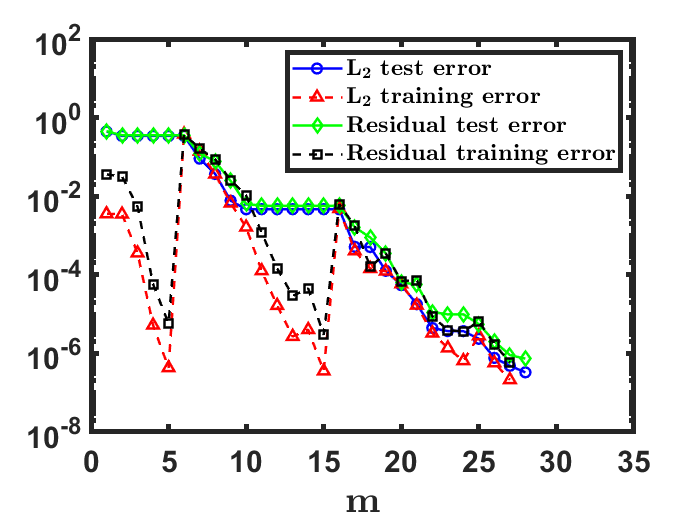}
\caption{}
\end{subfigure}
\begin{subfigure}{0.45\textwidth}
\includegraphics[width=\linewidth]{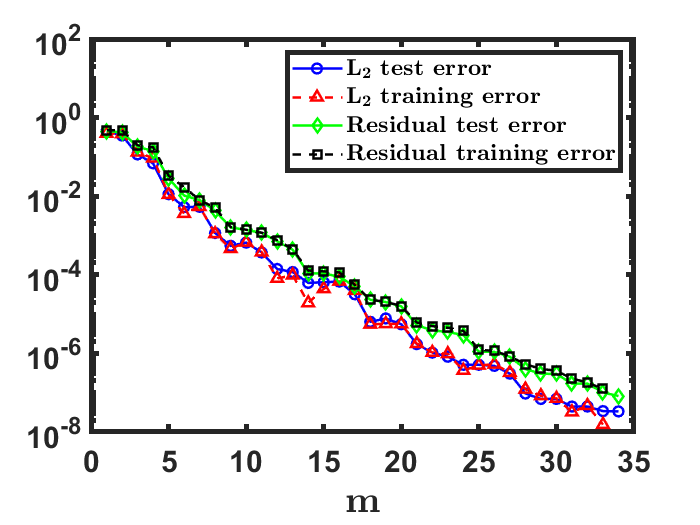}
\caption{}
\end{subfigure}
\caption{2D2v example: Lattice. Training and test errors using (a) G-$L_1$, (b) G-Res, (c) PG-$L_1$, (d) PG-Res. {See Figure \ref{fig:2d Lattice First Point Bottom Left Corner} for improved PG-$L_1$ results.}}
\label{fig:2d Lattice Training and Test Errors}

\end{figure}

\begin{figure}[ht]
\centering{
\includegraphics[width=0.24\textwidth, height=1.3in]
{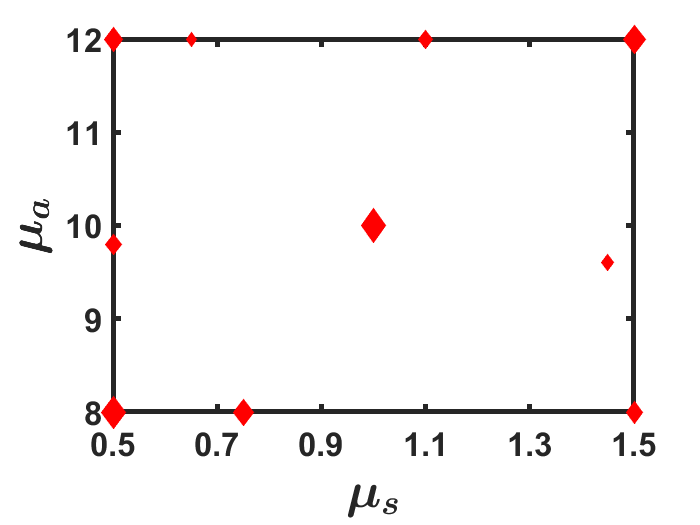}
\includegraphics[width=0.24\textwidth, height=1.3in]
{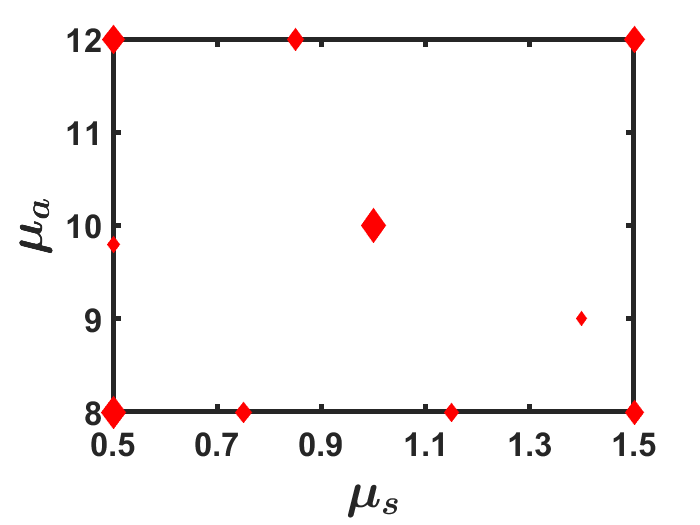}
\includegraphics[width=0.24\textwidth, height=1.3in]
{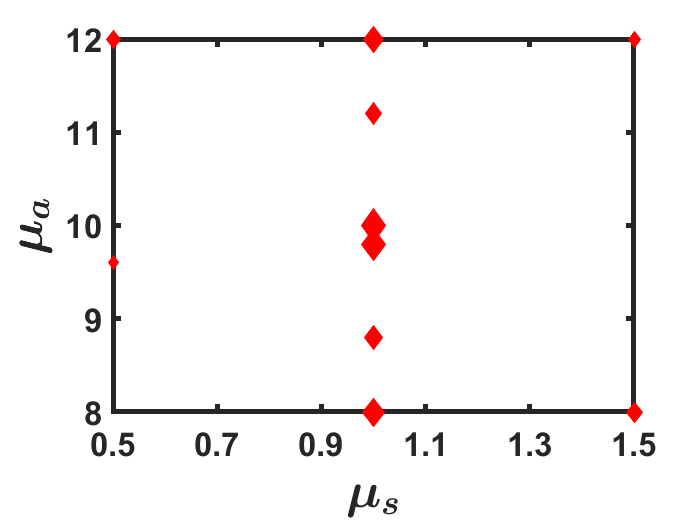}
\includegraphics[width=0.24\textwidth, height=1.3in]
{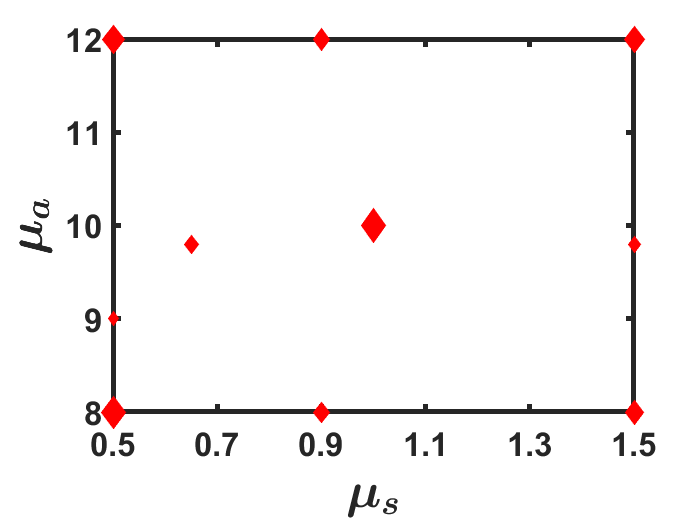}
}
\caption{2D2v example: Lattice. From left to right: first 10 parameter points picked in greedy search using G-$L_1$, G-Res, PG-$L_1$, PG-Res, respectively.}
\label{fig:2d Lattice Points Picked in Greedy Search}

\end{figure}

We want to start with PG-$L_1$. First of all, the severe dips in the training errors and the plateaus in the respective testing errors 
by PG-$L_1$ make it a relatively less efficient surrogate solver compared with the other three ROMs in order to achieve  relatively lower resolution/fidelity. For example, to achieve $L_2$ test errors around $10^{-1}$,  $10^{-2}$, $10^{-3}$, PG-$L_1$ needs to be of dimension 7, 9, 17, respectively, while the other three ROMs need to be of lower dimension around 3-4, 5-6, 8-10, respectively. The difference in these required reduced dimensions by all four ROMs becomes smaller for relatively higher resolution, e.g. for test errors to be around $10^{-4}$ and lower. The selected parameter values by PG-$L_1$ in Figure \ref{fig:2d Lattice Points Picked in Greedy Search} also show a much more imbalanced distribution during the early greedy selection. They are much less independent, and this likely explains the unusually fast decay of the spectral ratio. This test  reinforces that the spectral ratio cannot fully inform us the resolution of a ROM.

On the other hand, the remaining three ROMs, namely G-$L_1$, G-Res, PG-Res, all perform well, with some similar observations as {1D examples in slab geometry}  in Section \ref{sec:num:1D1v}, 
particularly regarding the overall decreasing trends of the training and testing errors at exponential decay rates as $\rbdim$ grows,
the training and test errors of the same type being at comparable magnitude for each ROM, the training and test errors of the residual
type by PG-Res decreasing monotonically, and the greedy selection favoring ``extreme" parameter values during the early iterations.  For each of these three ROMs, there are dips in the training errors from the respective test errors, with those by G-$L_1$ more noticeable, evidencing that the selected points are not optimal. Given that these dips are small, they will not be of real concern in terms of the overall efficiency and resolution of these ROMs.

\medskip
\noindent{\bf{Improved PG-$L_1$.}}
A natural question is whether PG-$L_1$ can be improved by starting from a different initial parameter value $\parameter_1$. Informed by Figure \ref{fig:2d Lattice Points Picked in Greedy Search}, particularly the selected $\parameter_2$ in the other three ROMs,   PG-$L_1$ is re-trained/re-built by using $\parameter_1=(\parameter_s, \parameter_a)\left|_{(0.5, 8)}\right.$, with the training and test errors as well as the first 10 selected parameter points
presented in Figure \ref{fig:2d Lattice First Point Bottom Left Corner} {(top row).}  Though fairly big dips are still present in training errors during the early greedy iterations, with the respective selected parameter values (during early greedy iterations) showing unbalanced distribution,  
PG-$L_1$ starting with this $\parameter_1=(0.5, 8)$ is greatly improved. In fact the required reduced dimension $\rbdim$ by all ROMs are more comparable when the test errors are below $10^{-2}$. This study shows that the efficiency of the PG-$L_1$ with low reduced dimension  
is more sensitive to the initial selection of the parameter value $\parameter_1$ compared with the other three ROMs. In other words, the impact of the {selection of the first parameter} value on the performance of G-$L_1$, G-Res, PG-Res seems to damp faster as $\rbdim$ grows.

The improvement above is built upon the knowledge that the parameter value $\parameter=(0.5, 8)$ is important. Such {\it a priori} knowledge in general is unavailable. We further re-build PG-$L_1$ via a more general strategy, namely by using {\it the enhanced 2-point $L_1$ error indicator} described in Appendix \ref{app:enhancedL1}, with the  training and test errors as well as the first 10 selected parameter points
presented in  
Figure \ref{fig:2d Lattice First Point Bottom Left Corner} {(bottom row).} The first parameter value $\parameter_1$ is chosen as the geometric center. Though there is a noticeable dip in the training errors in the first few greedy iterations, and there are noticeable oscillations in the $L_2$ training errors, both residual and $L_2$ test errors 
display consistent monotonicity as the reduced dimension grows.  Note that the terminal reduced dimension in Figure \ref{fig:2d Lattice First Point Bottom Left Corner}-(c) is lower than that in Figure \ref{fig:2d Lattice First Point Bottom Left Corner}-(a) by 5. The offline efficiency to build the ROMs along with their resolution for online prediction are greatly improved. In addition, the first 10 parameter values that are greedily selected are distributed in a much more balanced pattern. The improvement in this experiment comes at a price of an almost  double yet affordable offline training cost. (One can refer to Appendix \ref{app:enhancedL1}.)  As to be seen next, the offline training cost is dominated by that of FOM solves.

It is informative to observe the  overall balanced distributions of the first 20 greedily selected parameter values in Figure \ref{fig:2d Lattice 20 Points Greedy Search} by PG-$L_1$ with $\parameter_1=(\parameter_s,\parameter_a) = (0.5,8)$,  by PG-$L_1$ with the enhanced 2-point $L_1$ error indicator and $\parameter_1$ as geometric center of $\mP_{\textrm{train}}$, {compared to that of} 
G-Res with $\parameter_1$ as geometric center of $\mP_{\textrm{train}}$.

\begin{figure}[ht]
\centering{
\begin{subfigure}{0.36\textwidth}
\includegraphics[width=\linewidth]{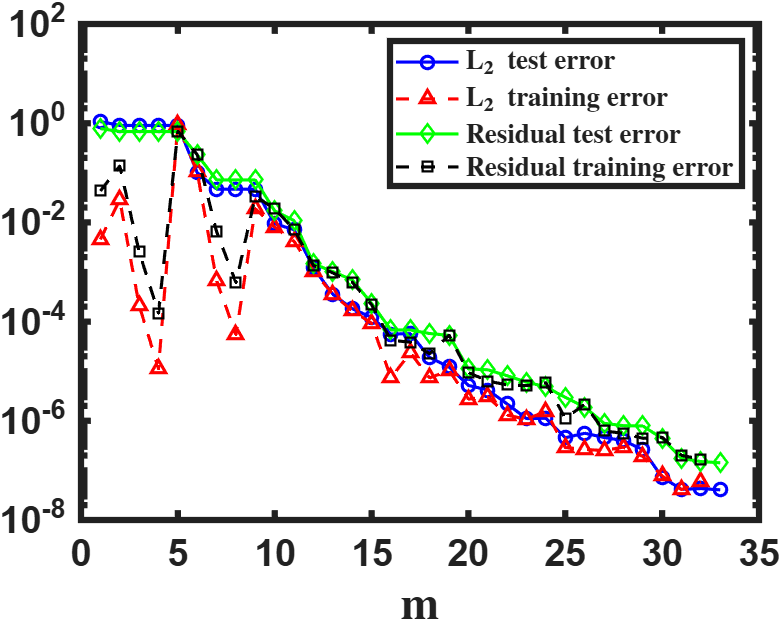}
\caption{}
\end{subfigure}
\begin{subfigure}{0.4\textwidth}
\includegraphics[width=\linewidth]{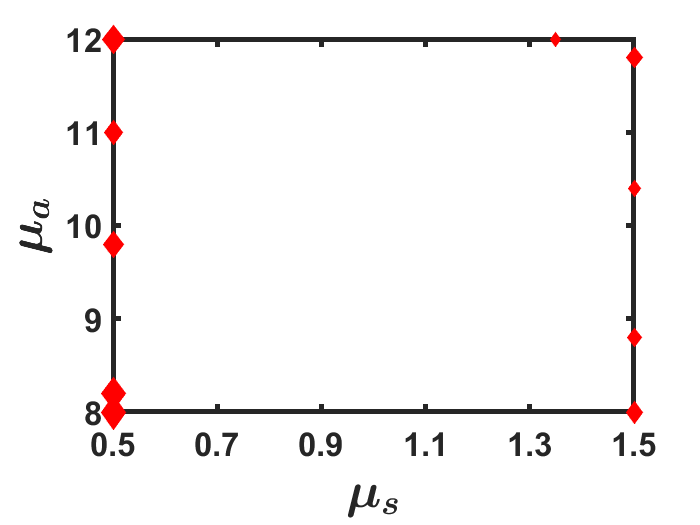}
\caption{}
\end{subfigure}
\begin{subfigure}{0.36\textwidth}
\includegraphics[width=\linewidth]{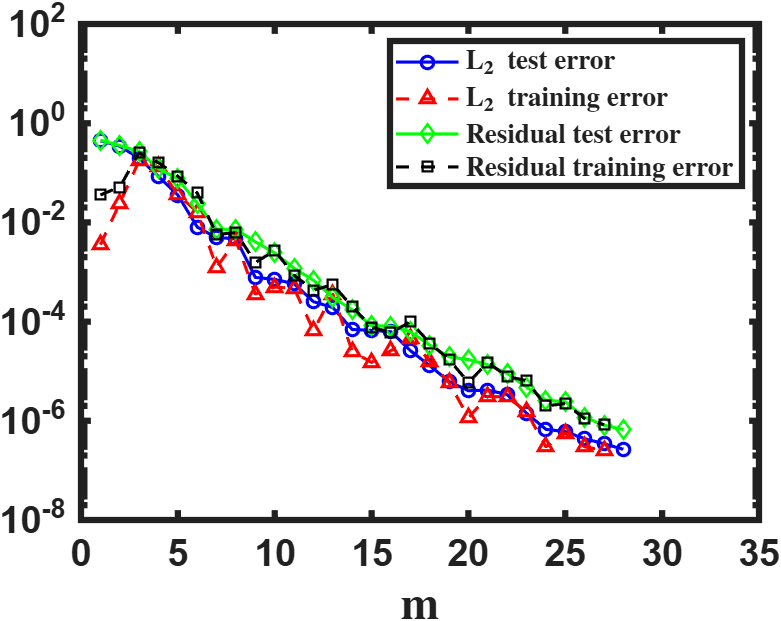}
\caption{}
\end{subfigure}
\begin{subfigure}{0.4\textwidth}
\includegraphics[width=\linewidth]{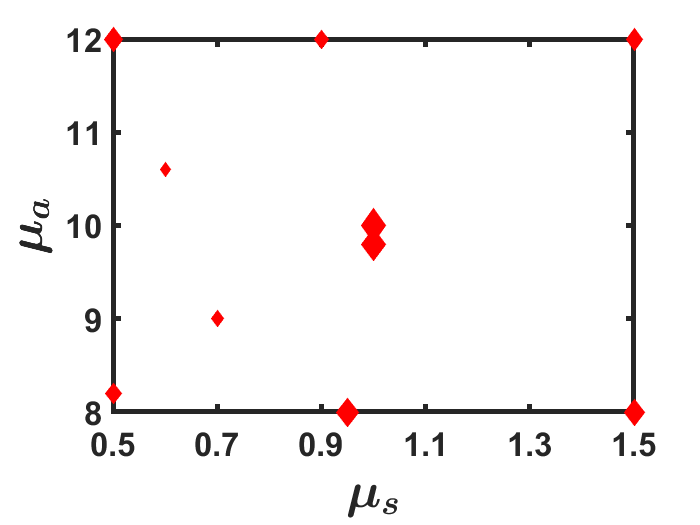}
\caption{}
\end{subfigure}
}
\caption{2D2v example: Lattice, with improved PG-$L_1$. {Top row: PG-$L_1$ with the first parameter value $\parameter_1=(\parameter_s,\parameter_a) = (0.5,8)$. Bottom row: PG-$L_1$ with enhanced 2-point $L_1$ error indicator. Left column:  Training and test errors.  
Right column: First 10 parameter points picked in greedy search. }}
\label{fig:2d Lattice First Point Bottom Left Corner}
\end{figure}

\begin{figure}[ht]
\centering{
\begin{subfigure}{0.32\textwidth}
\includegraphics[width=\linewidth]{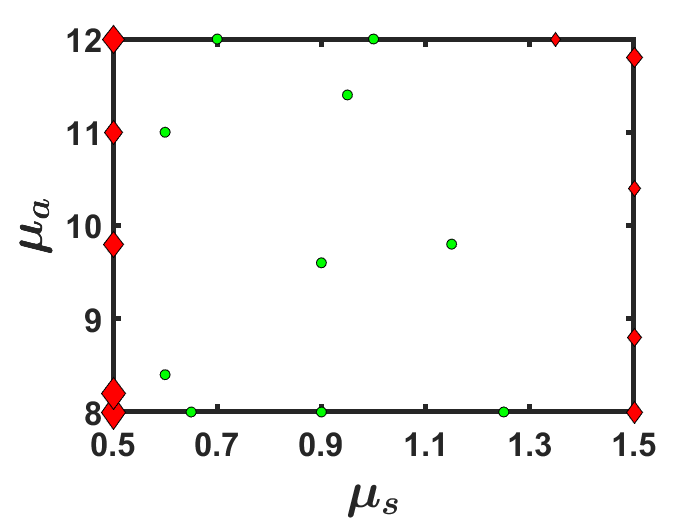}
\caption{}
\end{subfigure}
\begin{subfigure}{0.32\textwidth}
\includegraphics[width=\linewidth]{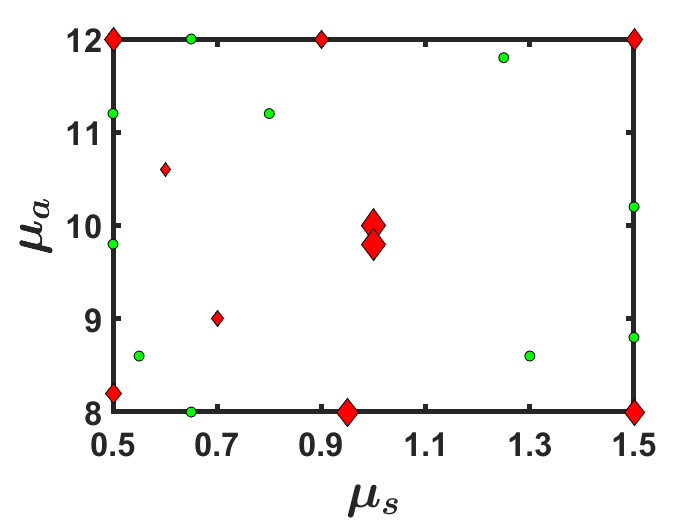}
\caption{}
\end{subfigure}
\begin{subfigure}{0.32\textwidth}
\includegraphics[width=\linewidth]{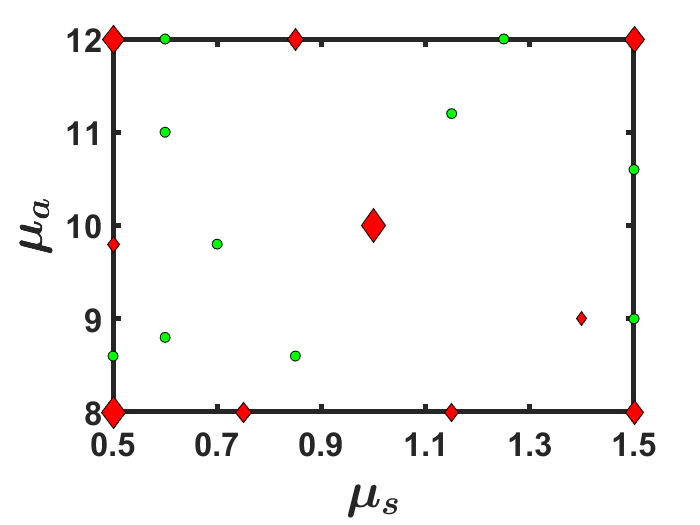}
\caption{}
\end{subfigure}
}
\caption{2D2v example: Lattice. First 20 parameter points picked in greedy search using (a) PG-$L_1$ with $\parameter_1=(\parameter_s,\parameter_a) = (0.5,8)$, (b) PG-$L_1$ with enhanced 2-point $L_1$ error indicator and $\parameter_1$ as geometric center of $\mP_{\textrm{train}}$, (c) G-Res with $\parameter_1$ as geometric center of $\mP_{\textrm{train}}$. Red: the first 10 selected points. Green: the next 10 selected points.}
\label{fig:2d Lattice 20 Points Greedy Search}
\end{figure} 

\medskip
\noindent{\bf Study of offline and online computational costs.}
This lattice  example is also used to study the computational costs of the proposed ROMs in comparison to the FOM.  
Two tests are conducted to examine the offline costs to build ROMs of different reduced dimension $\rbdim$ and to showcase the significant savings of ROMs to predict the solutions.
In both tests, the results for PG-$L_1$ are generated when using $\parameter_1=(\parameter_s, \parameter_a)\left|_{(0.5, 8)}\right.$ as the initial parameter value and the standard $L_1$ error indicator.
We want to point out that our algorithms are implemented in MATLAB on a personal laptop. 
As a high-level interpreted language, MATLAB does not allow users to have control over many implementation details, hence one cannot take all the costs presented here too absolutely. We will focus on those observations that we believe are robust and contribute to/complement  our understanding of the proposed ROMs.

\medskip
\noindent{\bf Test 1.} In this test, we focus on the offline training stage of the RBM, and compare in Figure \ref{fig:2d Lattice FOM solve time vs ROM offline time} the computational time (i.e. wall-clock time) to construct each ROM of dimension $\rbdim$ with that to run $N_{run}=\rbdim$ FOM solves at the selected parameter values during the ROM building process, with $\rbdim$ ranging from $1$ to $30$.  The solid line in each subfigure is for FOM solves, while the dashed lines are for building ROMs with a few representative $\rbdim=5, 15, 25, 30$. 

First of all, for each ROM, we want to find out when the crossover occurs between each dashed line and the solid  line. Recall that in order to construct each ROM($\rbTrial^{\rbdim};\cdot$),  FOM solvers will be called $\rbdim$ times at $\rbdim$ greedily selected parameter values. For G-$L_1$ of dimension $\rbdim=5, 15, 25, 30$, the crossover happens almost at $N_{run}=\rbdim$. This shows that the cost of the offline training stage is dominated by the FOM solvers, and the offline training stage is extremely efficient. 
For G-Res, PG-$L_1$, and PG-Res, the crossover happens around $N_{run}$ (e.g., $N_{run}=\rbdim, \rbdim+1, \rbdim+2$), again supporting that the offline training stage is very efficient. 

Related, we also present in Figure \ref{fig:2d Lattice-Line Source ROM offline times}-(a) the history of the computational times to build each ROM of growing dimension. The growth rates in the cost seem to be linear at least up to $\rbdim=30$, again indicating the computational cost is dominated by that of FOM solves when one refers to the theoretical computational complexity in  \eqref{eq:final:cost:off}. 

From  Figure \ref{fig:2d Lattice FOM solve time vs ROM offline time} and Figure \ref{fig:2d Lattice-Line Source ROM offline times}-(a), one can see that the times to build all four ROMs of the same reduced dimension do not differ significantly, with PG-$L_1$ {being} the fastest to build for this specific example. It is also observed that  it is relatively faster to build G-$L_1$ than to build G-Res, and it is faster to build PG-$L_1$ than to build PG-Res, as discussed in Remark \ref{rem:costComp:offline}, due to the fact that $L_1$ error indicators involve lower costs than residual-based error indicators. 

{Finally we present in Figure \ref{fig:2d Lattice-Line Source ROM offline times and errors}-(a) the computational time (sec) to build ROMs versus $L_2$ training errors
during the offline training stage. This is a different way to view the offline ($L_2$) training errors from Figure \ref{fig:2d Lattice Training and Test Errors}-(a)(b)(d),   Figure \ref{fig:2d Lattice First Point Bottom Left Corner}-(a) and the computational time from  Figure \ref{fig:2d Lattice-Line Source ROM offline times}-(a) as the reduced dimension $\rbdim$  grows. Qualitatively comparable rates of change are observed for all four ROMs, with no clear winner over the range of the reduced dimension  examined here.}

\begin{figure}[ht]
    \centering
    {
    \begin{subfigure}{0.24\textwidth}
    \includegraphics[width=\linewidth]{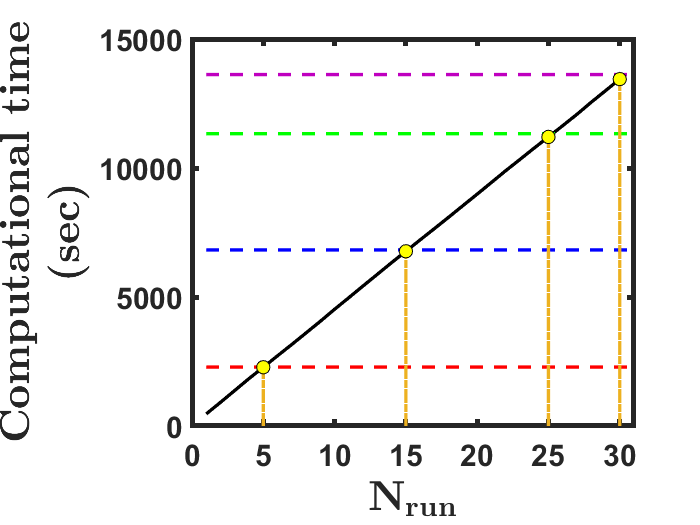}
    \caption{G-$L_1$}
    \end{subfigure}
    \begin{subfigure}{0.24\textwidth}
    \includegraphics[width=\linewidth]{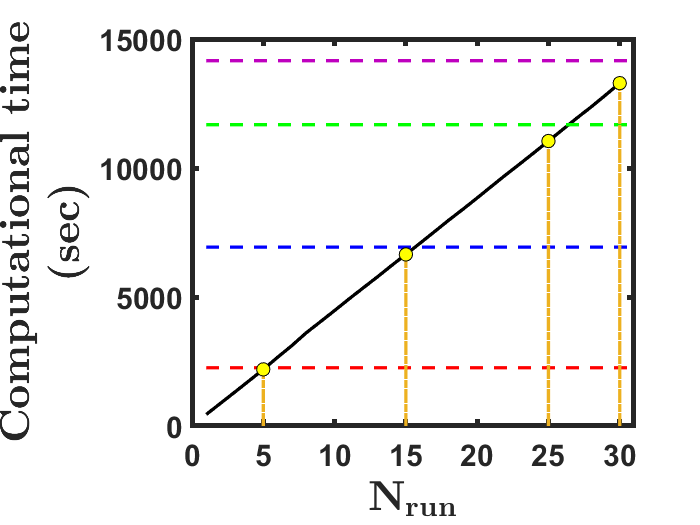}
    \caption{G-Res}
    \end{subfigure}
    \begin{subfigure}{0.24\textwidth}
    \includegraphics[width=\linewidth]{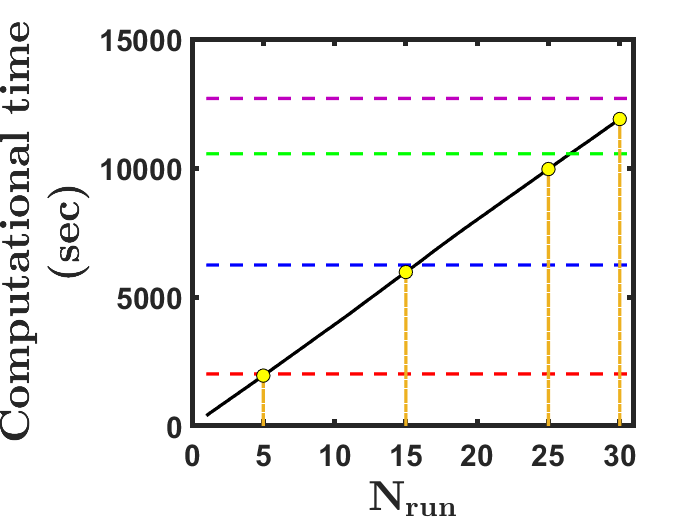}
    \caption{PG-$L_1$}
    \end{subfigure}
    \begin{subfigure}{0.24\textwidth}
    \includegraphics[width=\linewidth]{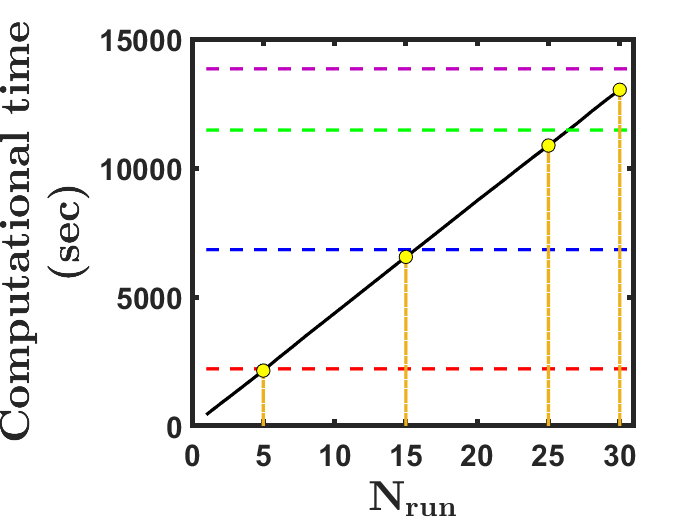}
    \caption{PG-Res}
    \end{subfigure}
    }
    \caption{2D2v example: Lattice. $\mN = 4.704 \times10^6$. Comparison between the computational time (sec)  to build each ROM($\rbTrial^{\rbdim};\cdot$) and the total computational time (sec)  of $N_{run}=\rbdim$ FOM solves at the selected parameters during the same process.   Solid line: FOM. Dashed line: ROM  of dimension $\rbdim=5, 15, 25, 30$ (from bottom to top).
 }
    \label{fig:2d Lattice FOM solve time vs ROM offline time}
\end{figure}

\begin{figure}[ht]
    \centering{
    \begin{subfigure}{0.45\textwidth}
\includegraphics[width=\linewidth]{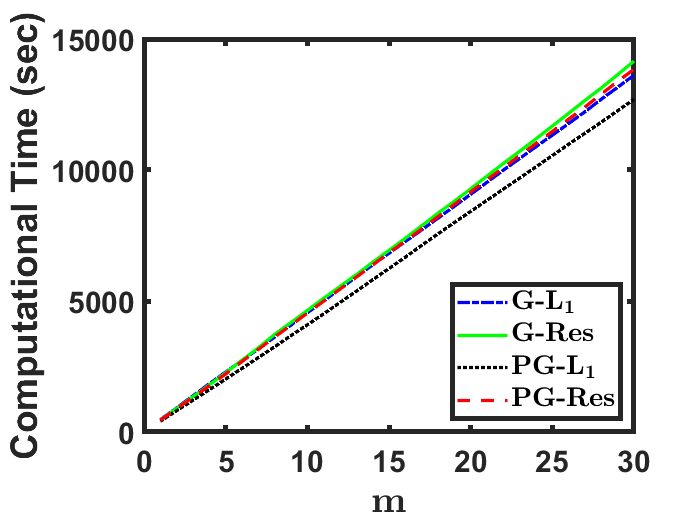}
        \caption{} 
    \end{subfigure}
     \begin{subfigure}{0.45\textwidth}  \includegraphics[width=\linewidth]{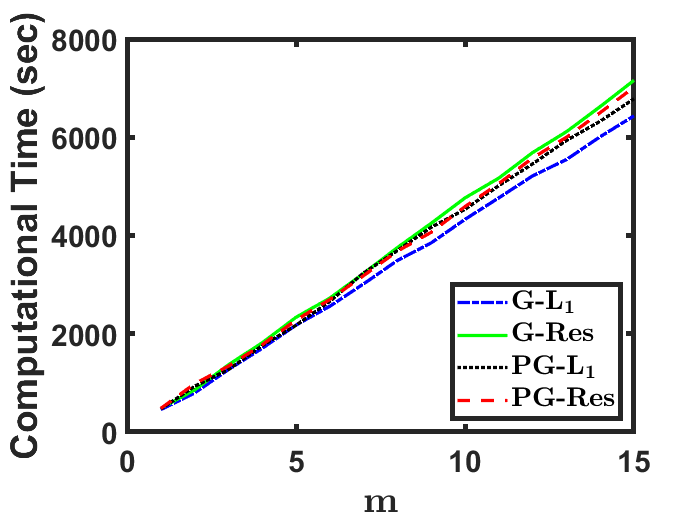}
        \caption{} 
    \end{subfigure}   
    }
    \caption{2D2v example: The history of the computational time (sec) to build ROMs of $\rbdim$ dimensions  using G-$L_1$, G-Res, PG-$L_1$, and PG-Res.  Left:  Lattice with $\mN = 4.704 \times10^6$. Right:  Line source with $\mN = 4.608 \times10^6$. }
    \label{fig:2d Lattice-Line Source ROM offline times}
\end{figure}

\begin{figure}[ht]
    \centering{
  \begin{subfigure}{0.45\textwidth}    \includegraphics[width=\linewidth] 
{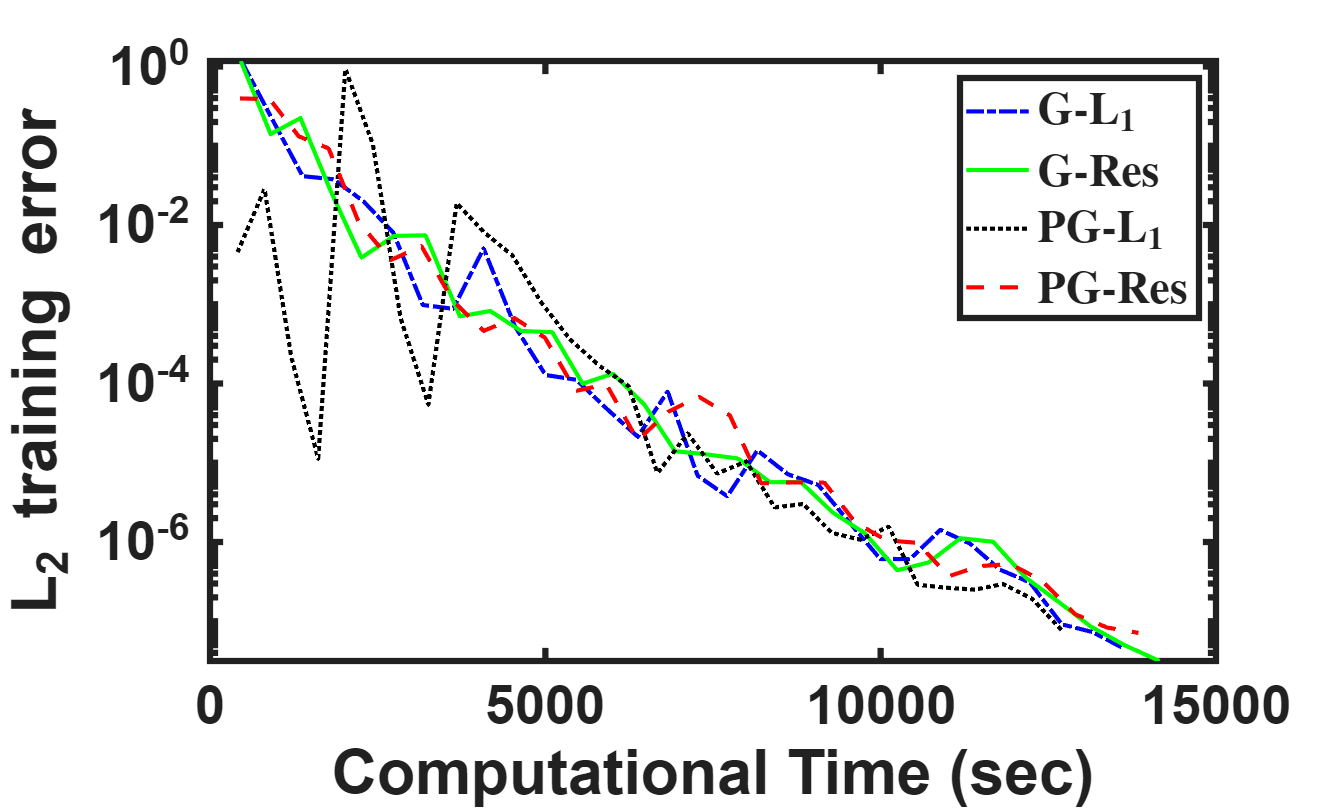}
    \caption{} 
    \end{subfigure}
     \begin{subfigure}{0.45\textwidth}   \includegraphics[width=\linewidth] 
{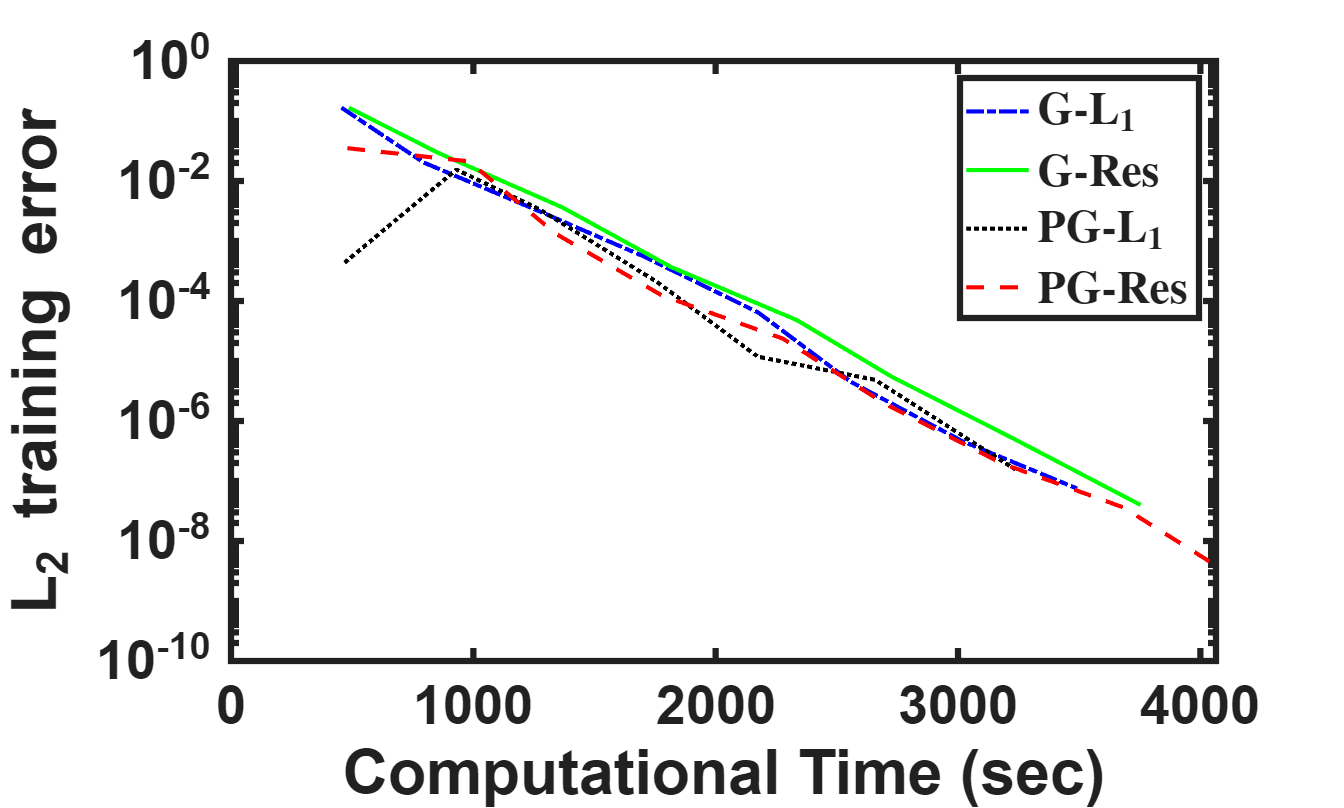}
    \caption{}
    \end{subfigure}
    }
    \caption{{2D2v example: The computational time (sec) to build ROMs versus $L_2$ training errors during the offline training stage as the reduced dimension $\rbdim$ grows.  Left:  Lattice with $\mN = 4.704 \times10^6$. Right:  Line source with $\mN = 4.608 \times10^6$. }}
    \label{fig:2d Lattice-Line Source ROM offline times and errors}
\end{figure}

\medskip
\noindent{\bf Test 2.} We next compare the computational costs of the proposed ROMs (with the reduced dimension $\rbdim=15$) and FOM during the online prediction stage. Note that for the online prediction, the reduced system matrices and data vectors have already been pre-computed and are available as indicated in Alg. \ref{alg:entireAlg}. Table \ref{tab:2d Lattice FOM solve time vs ROM online time} reports the averaged wall-clock times for each ROM and the wall-clock time of FOM to solve the RTE at $(\parameter_s,\parameter_a) = (0.6,11.5)$ and $(\parameter_s,\parameter_a) = (1.4,8.5)$ when $\mN = 20*6*4*N_x*N_y$ with $N_x = N_y = 21k$, $k = 1, 2, 3$.
Note that $20*6$ corresponds to $(20,6)$-CL quadrature  in angle used in this test. 
At each $\parameter$, the {presented wall-clock time of a ROM is the average over} 10 runs of the method.

As expected, the FOM solve time increases as the spatial mesh is refined. In particular, on the finest mesh with $\mN\sim 1.9 \times 10^6$, it takes about 3 minutes to run one FOM solve. On the other hand, the ROM solve times are around $10^{-5}$ seconds, with a $3\times 10^6$ speedup rate when using G-$L_1$ or G-Res, and they are about $10^{-4}$ seconds with a $6\times 10^5$ speedup rate when using PG-$L_1$ or PG-Res.  Unlike FOM, the computational times of each ROM are overall independent of $\mN$. Though our implementation in MATLAB on a personal laptop does not estimate the speedup rates more absolutely, one can be confident to conclude that these ROMs will result in significant cost savings over the FOM to simulate the parametric RTE, especially in more realistic settings with even larger full dimension $\mN$.

\begin{table}[h]
\caption{2D2v example: Lattice. Computational time (sec) to solve FOM and ROM with G-$L_1$, G-Res, PG-$L_1$, PG-Res at $(\parameter_s,\parameter_a)$ with different $\mN$ when $\rbdim = 15$.}
\vspace{0.1in} 
\centering 
\begin{tabular}{|c|c|c|c|c|c|c|}\hline 
$(\mu_s,\mu_a)$ & $\mathcal{N}$ & FOM & G-$L_1$ & G-Res & PG-$L_1$ & PG-Res \\ \hline 
\multirow{3}{4em}{(0.6,11.5)} & 2.12E+05 & 9.20E+00 & 6.31E-05 & 5.93E-05 & 2.04E-04 & 2.75E-04 \\ 
& 8.47E+05 & 6.00E+01 & 6.79E-05 & 6.33E-05 & 2.14E-04 & 2.70E-04 \\ 
& 1.91E+06 & 1.80E+02 & 6.18E-05 & 6.25E-05 & 2.18E-04 & 2.89E-04 \\ 
\hline 
\multirow{3}{4em}{(1.4,8.5)} & 2.12E+05 & 9.78E+00 & 5.48E-05 & 5.51E-05 & 2.10E-04 & 2.53E-04 \\ 
& 8.47E+05 & 5.60E+01 & 5.97E-05 & 5.82E-05 & 2.11E-04 & 2.68E-04 \\ 
& 1.91E+06 & 1.66E+02 & 5.83E-05 & 5.58E-05 & 2.04E-04 & 2.69E-04 \\ 
\hline 
\end{tabular} 
\label{tab:2d Lattice FOM solve time vs ROM online time}
\end{table}

\subsubsection{Line source}
\label{sec:line source}

In this subsection, we consider a line source example on $\xset = [0,1]\times [0,1]$ with zero absorption and zero inflow boundary conditions and $\source(\x) = \source(x,y) = \exp(-100((x-0.5)^2+(y-0.5)^2))$.
We treat the scattering cross section as the parameter such that $\scat(\x; \parameter) = \parameter_s$ with $\parameter_s \in \mP = [0.5,5]$.
The training set $\mP_{\textrm{train}}$ consists of $101$ equally spaced points in $\mP$ and the test set $\mP_{\textrm{test}}$ consists of $20$ randomly chosen points uniformly sampled from $\mP$.
A uniform spatial mesh with $N_x = N_y = 80$ is used with $(30,6)$-CL quadrature for the {angular discretization}, {and this} corresponds to
$\mN = 4.608*10^6$. We take $\text{tol}_{SRatio}=10^{-7}$ 
to terminate the offline stage of the RBM.

We implement all four ROMs, namely, G-$L_1$, G-Res, PG-$L_1$, PG-Res, with their $L_2$ and residual type training and test errors presented in Figure \ref{fig:2d Line Source Training and Test Errors}, and their spectral ratios $r^{(\rbdim)}$ in  Figure \ref{fig:2d Line Source Spectral Ratio Points Picked in Greedy Search-Pin Cell Spectral Ratio}-(a), all as a function of the reduced dimension $\rbdim$.
In Figure \ref{fig:2d Line Source Spectral Ratio Points Picked in Greedy Search-Pin Cell Spectral Ratio}-(b), we further present the first 10 parameter points that are picked in the offline stage to build each ROM.

\begin{figure}[h]
\centering{
\begin{subfigure}{0.45\textwidth}
\includegraphics[width=\linewidth]{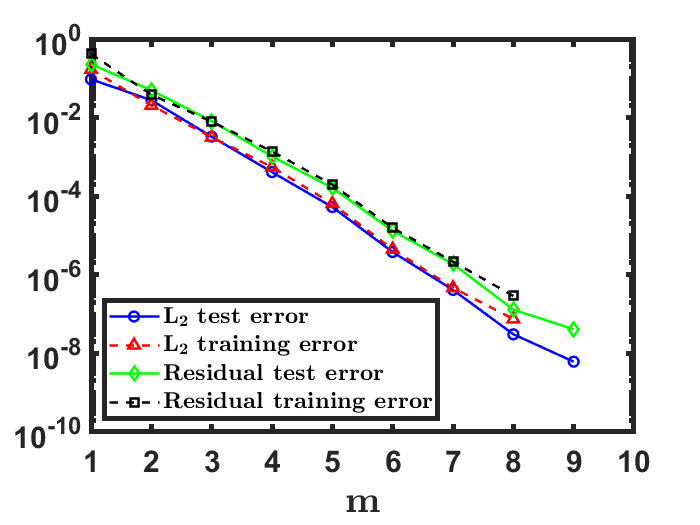}
\caption{}
\end{subfigure}
\begin{subfigure}{0.45\textwidth}
\includegraphics[width=\linewidth]{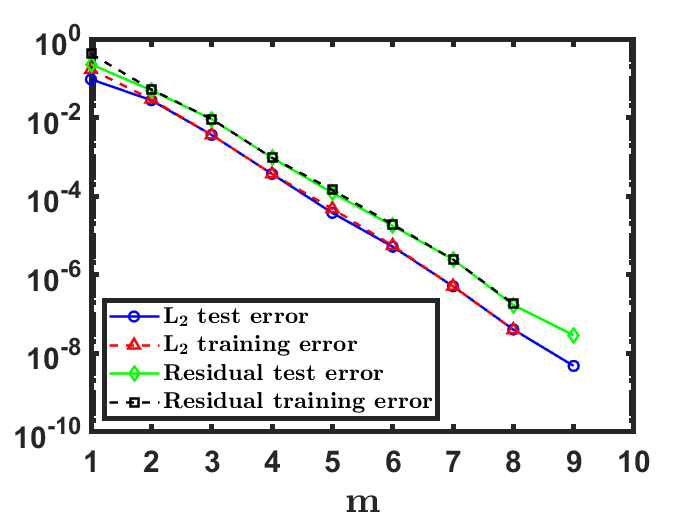}
\caption{}
\end{subfigure}
}

\begin{subfigure}{0.45\textwidth}
\includegraphics[width=\linewidth]{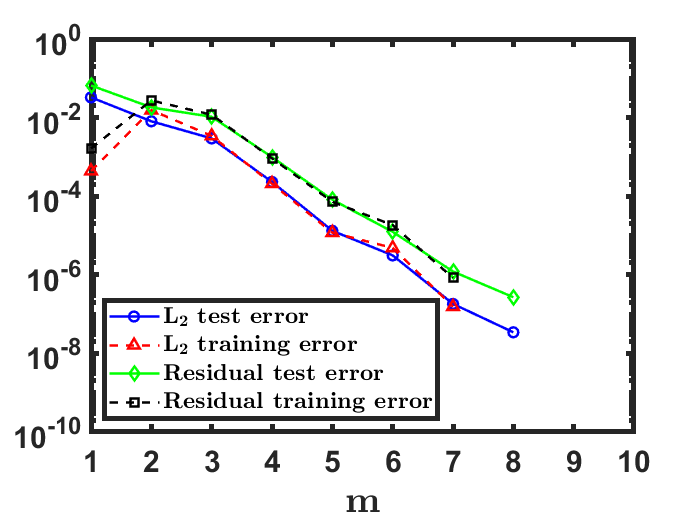}
\caption{}
\end{subfigure}
\begin{subfigure}{0.45\textwidth}
\includegraphics[width=\linewidth]{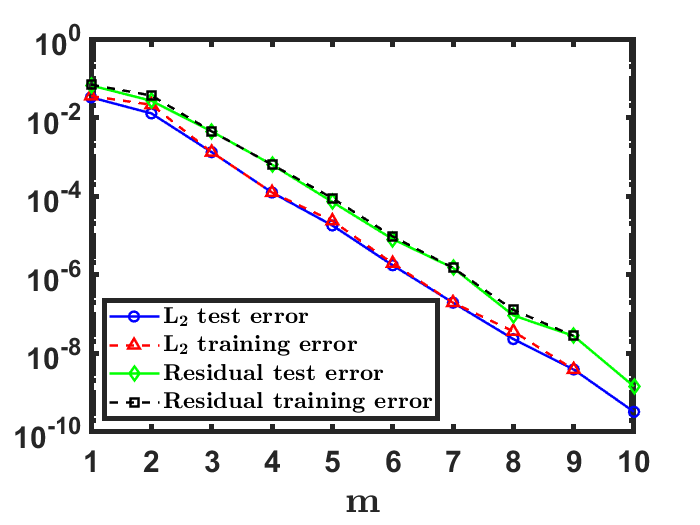}
\caption{}
\end{subfigure}

\caption{2D2v example: Line source. Training and test errors using (a) G-$L_1$, (b) G-Res, (c) PG-$L_1$, (d) PG-Res.}
\label{fig:2d Line Source Training and Test Errors}

\end{figure}

\begin{figure}[ht]
    \centering
    {
    \begin{subfigure}{0.3\textwidth}
    \includegraphics[width=\linewidth]{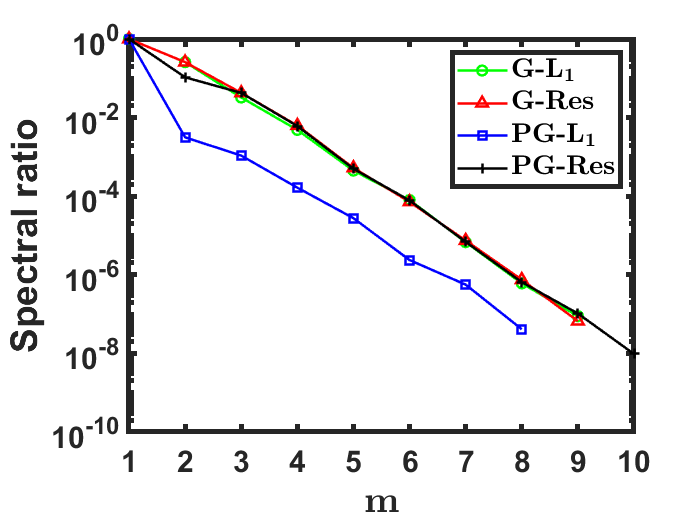}
    \caption{}
    \end{subfigure}
    \begin{subfigure}{0.3\textwidth}
    \includegraphics[width=\linewidth]{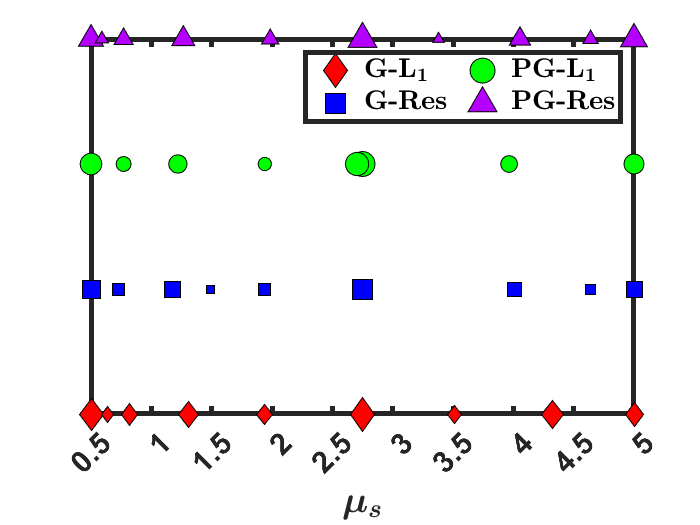}
        \caption{}
    \end{subfigure}
        \begin{subfigure}{0.3\textwidth}
    \includegraphics[width=\linewidth]{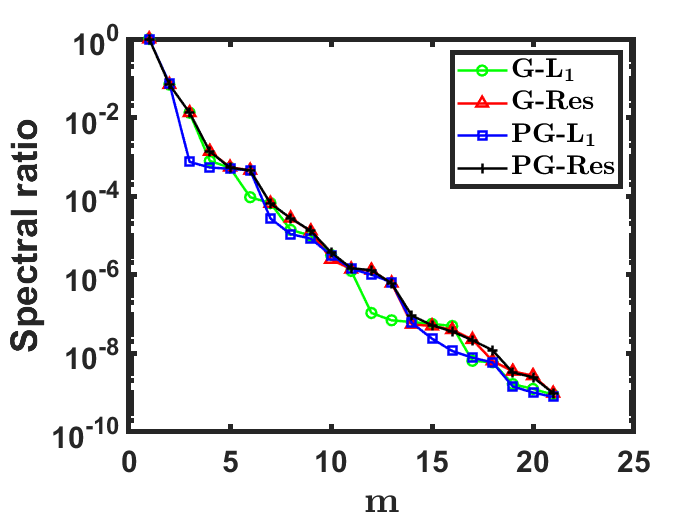}
        \caption{}
    \end{subfigure}
    }
    \caption{2D2v example: (a) Line source,   histories of  the spectral ratio. (b) Line source, parameter points picked in greedy search. (c) Pin-cell,  histories of  the spectral ratio.}
    \label{fig:2d Line Source Spectral Ratio Points Picked in Greedy Search-Pin Cell Spectral Ratio}
\end{figure}

Overall, all four ROMs perform well, with a lot of similar properties as seen in previous examples. We will not repeat these, but highlight a few observations.

\begin{itemize}
    \item 
    There is a noticeable dip in the training errors of PG-$L_1$ at the first iteration. This is related to the fact that  the first two parameter points picked in the offline stage are close to each other.
    This also explains the more rapid decay of the spectral ratio of PG-$L_1$ at the second iteration than the other methods.

    \item 
    Relatively more parameter points are picked in the left half of the parameter domain, with a noticeable cluster in $[0.5,1]$, indicating that the solution manifold is more sensitive to smaller values of $\parameter_s$.
    
\end{itemize}

This example is also used to demonstrate 
the offline costs to build ROMs of different reduced dimension $\rbdim$ and to showcase the significant saving of ROMs to predict the solutions. Two tests similar to those in the previous section are conducted.

\medskip
\noindent{\bf Test 1.} In this test, we focus on the offline training stage of the RBM, and compare in Figure \ref{fig:2d Line Source FOM solve time vs ROM offline time} the computational time (i.e. wall-clock time) to construct each ROM of dimension $\rbdim$ with that to run  $N_{run}=\rbdim$ FOM solves at the selected parameter values during the ROM building process, with $\rbdim$ ranging from $1$ to $15$.  The solid line in each subfigure is for FOM solves, while the dashed lines are for building ROMs with a few representative $\rbdim=5, 10, 15$.  For each ROM, we find that the crossover between each dashed line and the solid line occurs nearly at $N_{run}=\rbdim$. This confirms what we have learned from the lattice example in Section \ref{sec:lattice} that the cost of the offline training stage is dominated by the FOM solvers, and this offline training stage is extremely efficient.

Related, we also present in Figure \ref{fig:2d Lattice-Line Source ROM offline times}-(b) the history of the computational times to build each ROM of growing dimension; again as in Section \ref{sec:lattice}, we observe near linear growth of such cost in $\rbdim$. This is consistent with the observation that the offline computational cost is dominated by that of FOM solves. 
The costs to build all four ROMs are comparable, with G-$L_1$ the fastest for this specific example. Similar to what is observed in Section \ref{sec:lattice} and predicted by Remark \ref{rem:costComp:offline}, it is relatively faster to build G-$L_1$ than G-Res, and it is faster to build PG-$L_1$ than PG-Res.

{Finally we present in Figure \ref{fig:2d Lattice-Line Source ROM offline times and errors}-(b) the computational time (sec) to build ROMs versus $L_2$ training errors
during the offline training stage, as a different way to view the offline ($L_2$) training errors from Figure \ref{fig:2d Line Source Training and Test Errors} and the computational time from  Figure \ref{fig:2d Lattice-Line Source ROM offline times}-(b). To achieve the same resolution in the $L_2$ training errors, it is observed that PG-Res (resp. G-Res) is the most (resp. least) efficient in training  throughout the range of the reduced dimension examined here.}

\begin{figure}[ht]
    \centering
    {
    \begin{subfigure}{0.24\textwidth}
    \includegraphics[width=\linewidth]{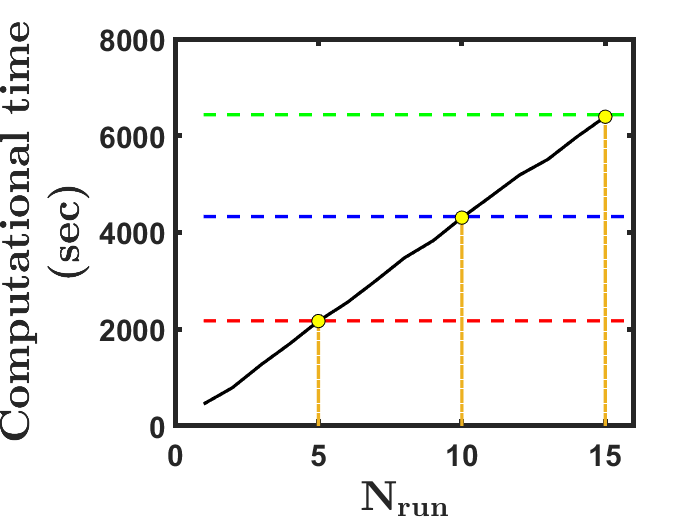}
    \caption{G-$L_1$}
    \end{subfigure}
    \begin{subfigure}{0.24\textwidth}
    \includegraphics[width=\linewidth]{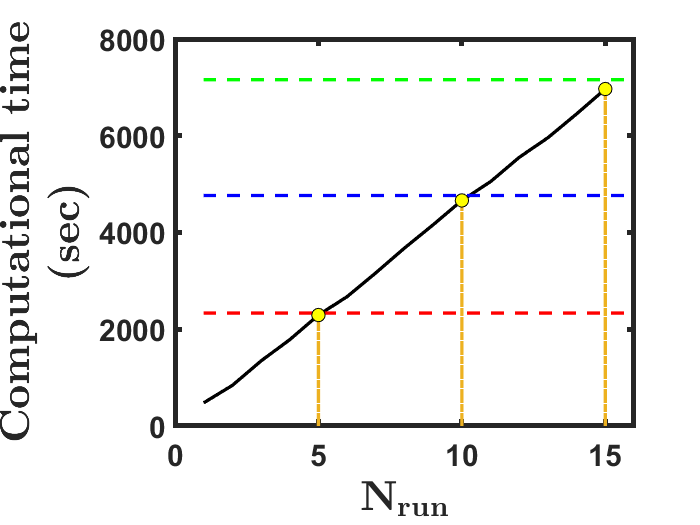}
    \caption{G-Res}
    \end{subfigure}
    \begin{subfigure}{0.24\textwidth}
    \includegraphics[width=\linewidth]{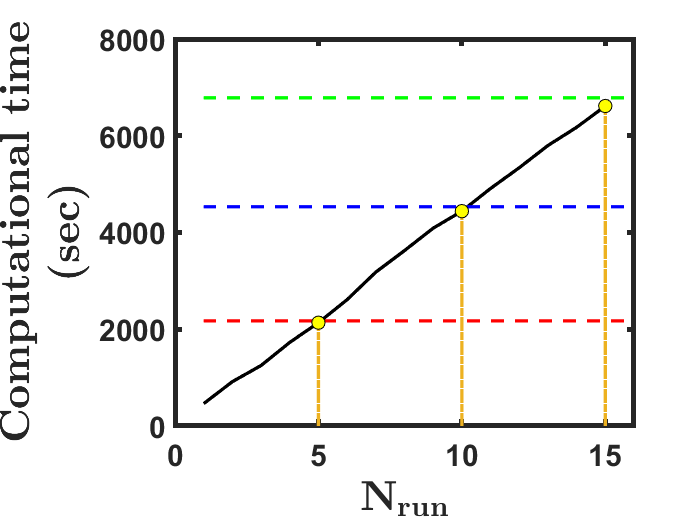}
    \caption{PG-$L_1$}
    \end{subfigure}
    \begin{subfigure}{0.24\textwidth}
    \includegraphics[width=\linewidth]{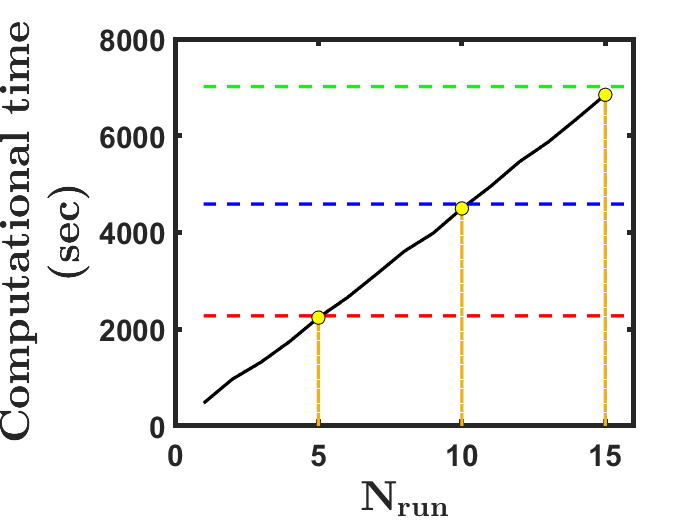}
    \caption{PG-Res}
    \end{subfigure}
    }
    \caption{2D2v example: Line source. $\mN = 4.608 \times10^6$. Comparison between the computational time (sec)  to build each ROM($\rbTrial^{\rbdim};\cdot$) and the total computational time (sec)  of $N_{run}=\rbdim$ FOM solves at the selected parameters during the same process.   Solid line: FOM. Dashed line: ROM  of dimension $\rbdim=5, 10, 15$ (from bottom to top). }
    \label{fig:2d Line Source FOM solve time vs ROM offline time}
\end{figure}

\medskip
\noindent{\bf Test 2.} We next compare the computational costs of the proposed ROMs (with the reduced dimension $\rbdim=8$) and the FOM during the online prediction stage. 
Table \ref{tab:2d Line Source FOM solve time vs ROM online time} reports the averaged wall-clock times for each ROM and the wall-clock time of the FOM to solve the RTE at $\parameter_s = 1, 5$ when $\mN = 30*6*4*N_x*N_y$ with $N_x = N_y = 20*2^k$, $k = 0, 1, 2$.
Note that $30*6$ corresponds to $(30,6)$-CL quadrature in angle used in this test. 
At each $\parameter$, the {presented wall-clock time of a ROM is the average over} 10 runs of the method.

As expected, the FOM solve time increases as the spatial mesh is refined. In particular, on the finest mesh with $\mN\sim 4.6\times 10^6$, it takes about 7-8 minutes to run one FOM solve. On the other hand, the ROM solve times are around $10^{-4}, 10^{-5}$ seconds, with a speedup rate of magnitude $10^6$ for all four ROMs. The computational times of each ROM stay at the same magnitude for increasing $\mN$. Overall,  each of our four ROMs will result in significant cost savings over the FOM to simulate the parametric RTE.

\begin{table}[h] 
\caption{2D2v example: Line source. Computational time (sec) to solve FOM and ROM with G-$L_1$, G-Res, PG-$L_1$, PG-Res at $\parameter_s$ with different $\mN$ when $\rbdim = 8$.}
\vspace{0.1in} 
\centering 
\begin{tabular}{|c|c|c|c|c|c|c|}\hline 
$\mu_s$ & $\mathcal{N}$ & FOM & G-$L_1$ & G-Res & PG-$L_1$ & PG-Res \\ \hline 
\multirow{3}{*}{1} & 2.88E+05 & 1.40E+01 & 5.66E-05 & 5.41E-05 & 1.52E-04 & 1.51E-04 \\ 
& 1.15E+06 & 6.75E+01 & 5.29E-05 & 5.67E-05 & 1.77E-04 & 1.70E-04 \\ 
& 4.61E+06 & 4.16E+02 & 5.57E-05 & 5.42E-05 & 1.67E-04 & 1.45E-04 \\ 
\hline 
\multirow{3}{*}{5} & 2.88E+05 & 1.14E+01 & 5.34E-05 & 5.07E-05 & 1.67E-04 & 1.52E-04 \\ 
& 1.15E+06 & 7.00E+01 & 5.89E-05 & 5.75E-05 & 1.79E-04 & 1.45E-04 \\ 
& 4.61E+06 & 4.69E+02 & 5.58E-05 & 5.48E-05 & 1.61E-04 & 1.41E-04 \\ 
\hline 
\end{tabular} 
\label{tab:2d Line Source FOM solve time vs ROM online time}
\end{table}

\subsubsection{Pin-cell}
\label{sec:pin-cell}

In this subsection, we consider a pin-cell example on $\xset = [-1,1] \times [-1,1]$ with zero inflow boundary conditions and $\source(\x) = \exp(-100(x^2+y^2))$.
The parameter $\parameter=(\parameter_s, \parameter_a)$ arises from modeling the scattering and absorption cross sections, namely,
\[
\scat(\x; \parameter)
=
\begin{cases}
\parameter_s, & \left| x \right| \leq 0.5 \ \text{and} \ \left| y \right| \leq 0.5 \\
100, & \text{otherwise}
\end{cases}
,
\quad
\absorp(\x; \parameter)
=
\begin{cases}
\parameter_a, & \left| x \right| \leq 0.5 \ \text{and} \ \left| y \right| \leq 0.5 \\
0, & \text{otherwise}
\end{cases}
,
\]
and $(\parameter_s,\parameter_a) \in \mP = [0.05,0.5] \times [0.05,0.5]$.
The training set $\mP_{\textrm{train}}$ consists of $19 \times 19$ equally spaced points in $\mP$ and the test set $\mP_{\textrm{test}}$ consists of $10 \times 10$ randomly chosen points uniformly sampled from $\mP$.
A uniform spatial mesh with $N_x = N_y = 80$ is used with $(30,6)$-CL quadrature for the angular discretization.
This gives $\mN = 4.608*10^6$, and we take $\text{tol}_{SRatio}=10^{-9}$ to terminate the offline stage of the RBM.

We implement all four ROMs, 
with their $L_2$ and residual type training and test errors presented in Figure \ref{fig:2d Pin Cell Training and Test Errors}, and their spectral ratios $r^{(\rbdim)}$ in Figure \ref{fig:2d Line Source Spectral Ratio Points Picked in Greedy Search-Pin Cell Spectral Ratio}-(c), all as a function of the reduced dimension $\rbdim$.
We further present the first 10 parameter points that are picked in the offline stage of each ROM in Figure \ref{fig:2d Pin Cell Points Picked in Greedy Search}. Overall,  all four ROMs perform well for this example. Many of the similar observations and comments can be made as the examples in previous sections, e.g. regarding the overall decreasing trends of the training and testing errors as $\rbdim$ grows, the training and test errors of the same type being at the same magnitude for each ROM, 
the number of iterations for the spectral ratio to reach the specified tolerance being comparable among the four ROMs,
the training and test errors of the residual type by PG-Res decreasing monotonically, and the greedy selection favoring ``extreme" parameter values during the early iterations.

\begin{figure}[h]
\centering{
\begin{subfigure}{0.45\textwidth}
\includegraphics[width=\linewidth]{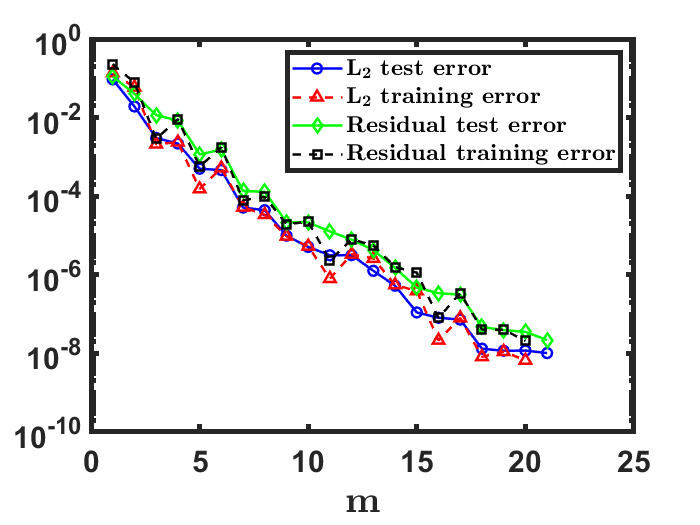}
\caption{}
\end{subfigure}
\begin{subfigure}{0.45\textwidth}
\includegraphics[width=\linewidth]{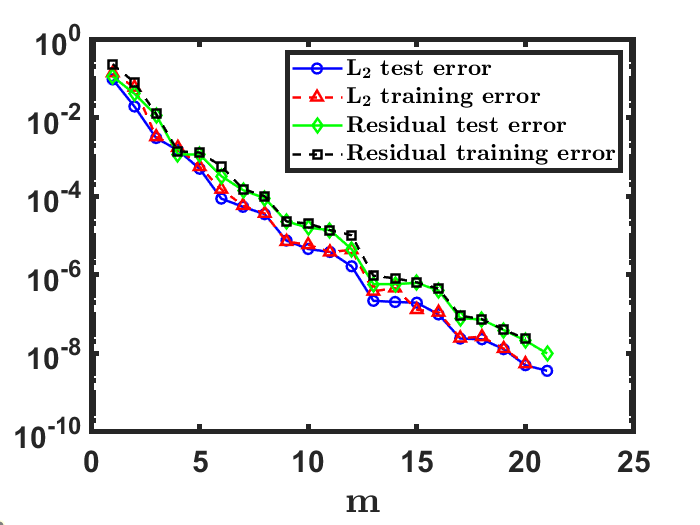}
\caption{}
\end{subfigure}
}

\begin{subfigure}{0.45\textwidth}
\includegraphics[width=\linewidth]{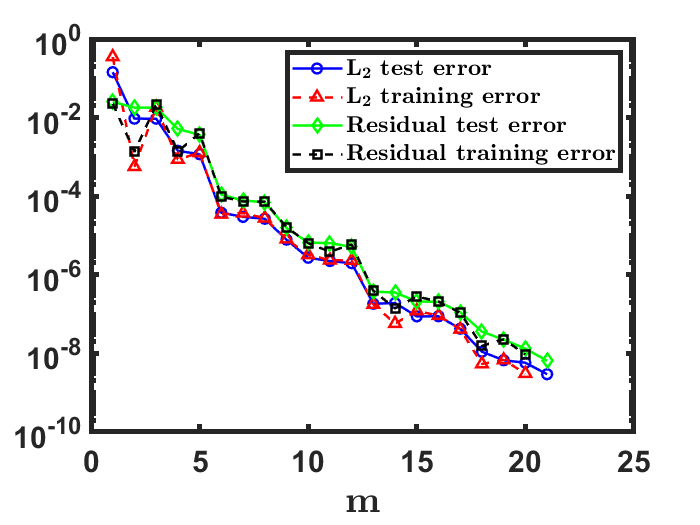}
\caption{}
\end{subfigure}
\begin{subfigure}{0.45\textwidth}
\includegraphics[width=\linewidth]{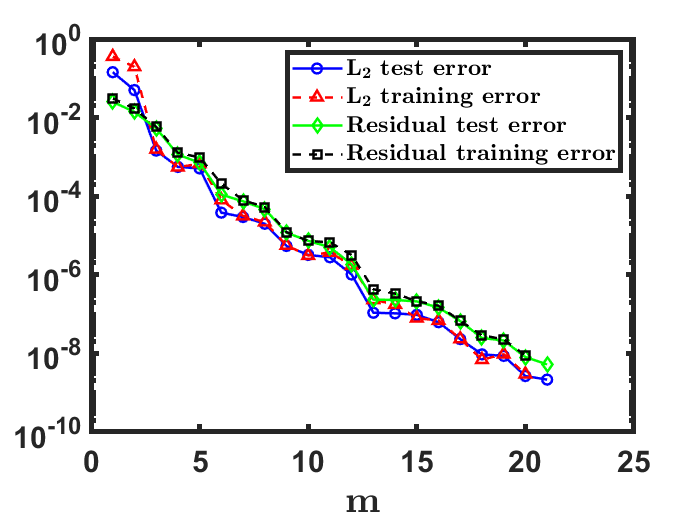}
\caption{}
\end{subfigure}

\caption{2D2v example: Pin-cell. Training and test errors using (a) G-$L_1$, (b) G-Res, (c) PG-$L_1$, (d) PG-Res.}
\label{fig:2d Pin Cell Training and Test Errors}

\end{figure}

\begin{figure}[h!]
\centering{
\includegraphics[width=0.24\textwidth, height=1.3in]
{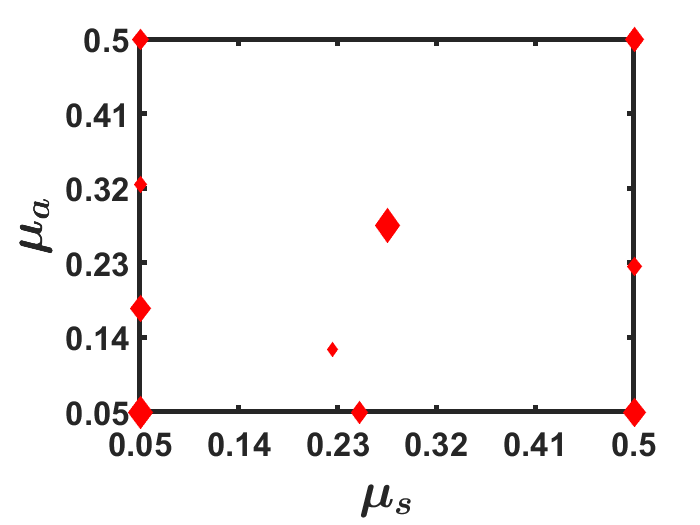}
\includegraphics[width=0.24\textwidth, height=1.3in]
{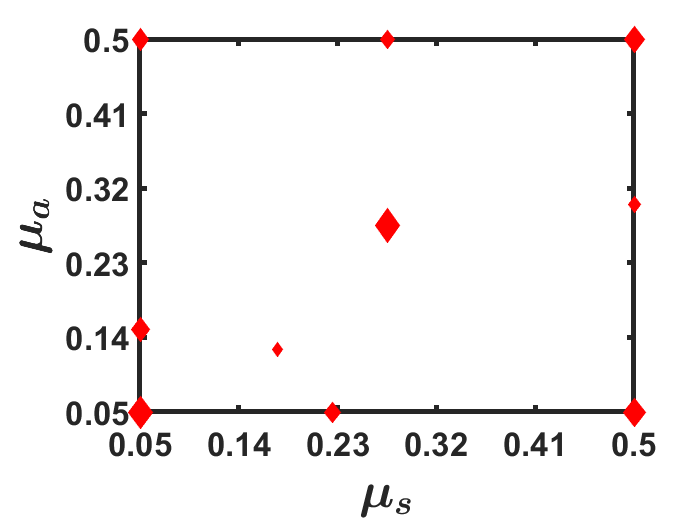}
\includegraphics[width=0.24\textwidth, height=1.3in]
{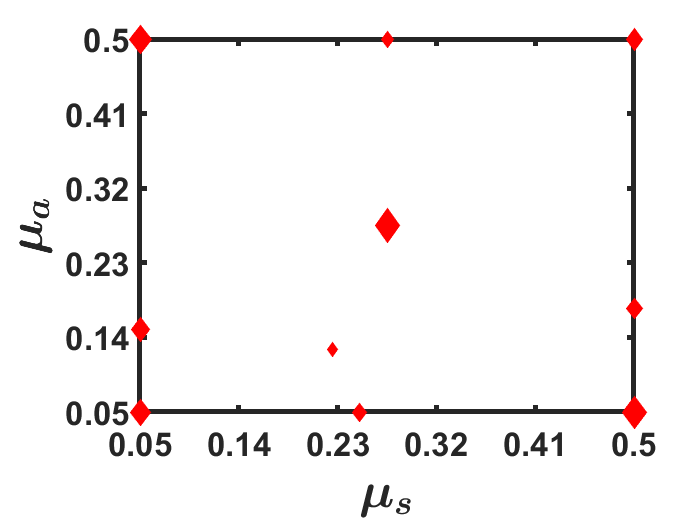}
\includegraphics[width=0.24\textwidth, height=1.3in]
{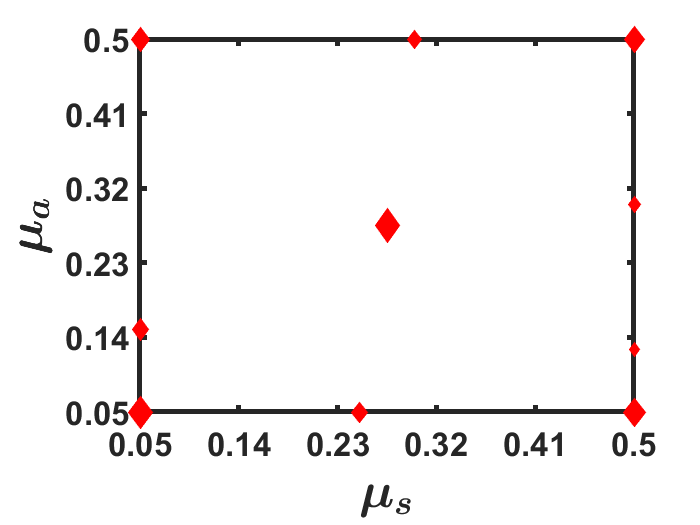}
}
\caption{2D2v example: Pin-cell. First 10 parameter points picked in greedy search using G-$L_1$, G-Res, PG-$L_1$, PG-Res, respectively.}
\label{fig:2d Pin Cell Points Picked in Greedy Search}

\end{figure}

For this example, the average number of iterations for SI-DSA to converge to the specified error tolerance $tol_{SI} = 10^{-12}$ is about 41, as reported in  Table \ref{tab:average iterations SI-DSA:rho0}, when the initial guess is taken as $\macrovec^{(0)}={\bf 0}$. Intuitively, the reduced density by our ROMs can provide a better initial guess and potentially speed up the convergence, and this idea was investigated in  \cite{peng2024reduced} based  on POD-based ROM, and will be demonstrated here.  Particularly, the reduced densities computed by our G-$L_1$ with different reduced dimensions are used as the initial guess. The average number of iterations, computed over $\mP_{\textrm{test}}$,  for SI-DSA to converge with such ROM initial guesses are  reported in Table \ref{tab:average iterations SI-DSA-ROM}. The results confirm that if one is interested in using FOMs to compute high-fidelity solutions for RTE, the reduced densities by our ROMs can be better initial guesses than zero initial guess and hence reduce the number of iterations required by the iterative SI-DSA solvers. The reduction is more prominent as the reduced dimension grows, due to that $\|\angflux_{RB}^{\rbdim}(\parameter)-\angflux_{h}(\parameter)\|_h\rightarrow 0$ as $\rbdim$ grows.
In Table \ref{tab:average iterations SI-DSA-ROM}, we also include the results when the initial densities in SI-DSA are computed by G-$L_1$ built/trained on a coarser $40\times 40$ spatial mesh (while the angular mesh and training set stay the same). This is motivated by the quest to possibly further reduce the cost to prepare/train ROM initial guesses, and turns out to be limited in its efficiency. Here is a simple explanation:  
$\|\angflux_{RB, 40\times 40}^{\rbdim}(\parameter)-\angflux_{h,80\times 80}(\parameter)\|_h\rightarrow 
\|\angflux_{h, 40\times 40}(\parameter)-\angflux_{h, 80\times 80}(\parameter)\|_h$  as $\rbdim$ grows, while for the pin-cell example we have $\|\angflux_{h, 40\times 40}(\parameter)-\angflux_{h, 80\times 80}(\parameter)\|_h\sim 10^{-2}$. The subscript $40\times 40$ or $80\times 80$ stands for the spatial mesh size to compute the FOM and to train the reduced order solvers.

\begin{table}[h!] 
\centering 
\caption{2D2v example: Pin-cell. Average number of iterations required for SI-DSA to converge with $tol_{SI} = 10^{-12}$ and $N_x = N_y = 80$ with (30,6)-CL quadrature. SI-DSA-0: SI-DSA with $\macrovec^{(0)} = \mathbf{0}$. SI-DSA-ROM($n_x, n_y$): SI-DSA with $\macrovec^{(0)}$ taken as the reduced density by G-$L_1$ trained  on a uniform $n_x \times n_y$ spatial mesh  with (30,6)-CL quadrature.} 
\begin{tabular}{|c|c|c|c|c|}\hline 
$\rbdim$ & $r^{(\rbdim)}$ & SI-DSA-0 & SI-DSA-ROM(80,80) & SI-DSA-ROM(40,40) \\
\hline
5 & $\sim 10^{-4}$ & \multirow{3}{*}{40.99} & 31.13 & 37.24 \\
\cline{1-2} \cline{4-5}
10 & $\sim 10^{-6}$ & & 22.48 & 37.25 \\
\cline{1-2} \cline{4-5}
15 & $\sim 10^{-8}$ & & 17.28 & 37.25 \\
\hline
\end{tabular}
\label{tab:average iterations SI-DSA-ROM}
\end{table}

\section{Conclusion}
\label{sec:conclusion}

In this work, we follow the reduced basis method framework, and systematically design and numerically test four families of projection-based ROMs to efficiently solve the parametric RTE.  They are shown to be efficient and reliable as accurate surrogate solvers for this intrinsically high dimensional transport model at many instances of parameter values. 
Though the RTE is not symmetric, ROMs based on Galerkin and Petrov-Galerkin projections work well overall. In addition, the simple $L_1$ error/importance indicator is shown to be effective {in guiding} the greedy parameter selection when combined with Galerkin projection. When it is combined with the Petrov-Galerkin projection defined via residual minimization, the initial selection of the parameter value can have a more lingering effect on the resolution of the ROMs with relatively smaller dimension. This can be improved by using an enhanced $L_1$ error indicator instead at a price of additional yet affordable offline training cost.  The residual of a function in the full order approximation space, when being used to define the reduced order solver in Petrov-Galerkin projection, or as an error indicator to guide  parameter selection, is always measured as a functional. With this,  our constructed ROMs are independent of a specific basis used to represent the approximation space.

By utilizing the affine assumption of the parameter dependence in the model, implementation strategies are carefully designed to best harvest the potential efficiency of the reduced algorithms, and this saving holds for both the offline training stage to construct the reduced order solvers, and the online prediction stage to compute the solutions by these surrogate solvers. 
Our reduced order solvers can predict the solutions for any parameter value at a significantly reduced cost. In some 2D2v examples we tested, a $10^4-10^6$  speedup factor is observed, and this will be even higher for larger full dimension $\mN$. Some algorithmic highlights for implementation include Alg. \ref{alg:pLS:offline}-Alg.\ref{alg:pLS:online} and Alg. \ref{alg:ResC:1} when residual minimization over and residual evaluation at many instances of parameter are involved. Unlike some previous developments, our proposed strategies for efficient implementation also take into account the conditioning of the reduced system matrices as well as the numerical robustness without the residual stagnation especially when the resolution/fidelity of the ROMs is high. It is worth noting that the implementation strategies proposed in this work are broadly applicable to ROMs for other parametric PDEs.  

The computational complexity to construct each ROM is estimated. Though it is nontrival to thoroughly verify the sharpness of these estimates, our numerical study of the computational costs with the 2D2v lattice problem and line source problem indicate that the cost during the offline training is dominated by computing the FOM solutions at greedily selected parameters. Recall that the number of degrees of freedom of the FOM for the high dimensional model like RTE generally is quite high (e.g. $10^5-10^6$ in our 2D2v examples in some moderate discretization setting). Unlike in POD-type ROMs where a large collection of FOM solutions are needed for the offline training stage, the number of FOM inquiries is minimal in the RBM setting, and this particularly supports that RBM-based ROMs are ideal for building efficient surrogate solvers for high dimensional parametric transport models like the RTE. 

For all the proposed ROMs, the spectral ratio stopping criterion is applied to terminate the offline stage. The spectral ratio is a quantity that measures how much extra information is added if the reduced basis space is further enriched. This simple and intuitive stopping criterion however is not certifiable, namely, it does not inform us how accurate the reduced solutions are. A scaled residual-based error indicator on the other hand can certify the $L_2$ errors in the reduced solutions when the absorption cross section is positively bounded below uniformly; therefore, G-Res and PG-Res are certified ROMs in such physical settings.  One important future task is to investigate and design certified RBM-based ROMs for the parametric steady-state (and even time-dependent) RTE with more general absorption cross sections.
Another interesting direction is to further improve the offline efficiency, by using the surrogate solvers in \cite{peng2022reduced} to compute the FOM solution during each greedy iteration in Alg. \ref{alg:RBM} and in Alg. \ref{alg:entireAlg} at least in the regimes when the solvers in \cite{peng2022reduced} are efficient. 
{In the presence of boundary layers, additional algorithmic developments are needed at both the FOM \cite{guermond2010asymptotic} and ROM levels, in order to design accurate and efficient surrogate solvers.}
The model order reduction techniques developed for the RTE with isotropic scattering and one energy group will contribute to efficient simulations of the parametric RTE in more realistic settings, e.g., with anisotropic scattering and multi-group energy. 

\section{Acknowledgments}
{In this work, Matsuda was supported  in part by National Science Foundation  grant DMS-1913072,  Chen was supported in part by National Science Foundation grant DMS-2208277 and by Air Force Office of Scientific Research grant FA9550-25-1-0181, Cheng was supported in part by DOE grant DE-SC0023164, Air Force Office of Scientific Research grant FA9550-25-1-0154 and Virginia Tech, and Li was supported in part by  National Science Foundation  grant DMS-1913072 and Air Force Office of Scientific Research grant FA9550-26-1-0003. 
The authors would also like to thank Dr. Zhichao Peng at  Hong Kong University of Science and Technology for generously sharing  his computer codes.}
\begin{appendices}%

\counterwithin{equation}{section}

\section{FOM by SI-DSA
and computational complexity}
\label{app:impt:cost:FOM}

The FOM is solved iteratively by the source iteration scheme accelerated by diffusion synthetic acceleration (SI-DSA) \cite{adams2002fast,peng2022reduced}. 
This is a well-established iterative method for the 
RTE model \eqref{eq:kinetic transport equation} when it is   discretized in space by a pure upwind strategy along with the $S_N$ angular treatment. The algorithm starts from an initial guess of $\macrovec^{(0)}$, then iterates  by a source  iteration (SI) combined with a transport sweep, and accelerated by a (consistent) diffusion synthetic acceleration (DSA).   By a transport sweep,  we refer to that the angular flux in the next iterate, namely, $\{\angfluxvec_j^{(k)}\}_{j=1}^{N_{\vel}}$,  is solved algebraically from
\begin{equation}
  \label{eq:trans-sweep}
  (\Upwind_j + \hat{\mathbf{\Sigma}}_a+\hat{\mathbf{\Sigma}}_s) \angfluxvec_j^{(k)} = \hat{\mathbf{\Sigma}}_s \macrovec^{(k-1)} + \dataVec_j,
  \end{equation}
  which is a block lower-triangular system once  the mesh elements (hence the degrees of freedom) are suitably ordered. The DSA correction/pre-conditioning is especially important for accelerating the convergence in the scattering dominated regime. In this work, we follow Algorithm 1 in  \cite{peng2022reduced} with the correction step given in Appendix A.1 therein.

  Let us examine the computational complexity of SI-DSA.
  If SI is performed without acceleration, at the $k$th iteration, one solves \eqref{eq:trans-sweep} for $\{\angfluxvec_j^{(k)}\}_{j=1}^{N_{\vel}}$ and then computes $\macrovec^{(k)} = \sum_{j=1}^{N_{\vel}} \vweight_j \angfluxvec_j^{(k)}$.
  SI terminates when $\| \macrovec^{(k)} - \macrovec^{(k-1)} \|_{\ell_{\infty}} < \vareps_{SI}$ where $\vareps_{SI}$ is a  specified tolerance.
  With acceleration, the algorithm includes a correction step after the transport sweep in which one solves an equation for $\delta \macrovec^{(k)}$, a correction to $\macrovec^{(k)}$, and then updates $\macrovec^{(k)}$ as $\macrovec^{(k)} = \sum_{j=1}^{N_{\vel}} \vweight_j \angfluxvec_j^{(k)} + \delta \macrovec^{(k)}$.
  This correction equation is solved using algebraic multigrid methods at a cost of $O(N_{\x})$, which is much lower than $O(\mN)$, the cost of one iteration  of  \eqref{eq:trans-sweep}.
  Therefore, assuming that SI-DSA converges in $N_{iter}$ iterations, the leading order cost of the FOM by SI-DSA is:   
  $\boxed{O\big(\mN N_{iter})}$.

\section{{Details of several proofs}}
\subsection{Proof of Lemma \ref{lem:norm:alg}}
\label{ap:lem1}

Consider any $g_h, \hat{g}_h\in U_h$,  let   $\mathbf{g}$, $\hat{\mathbf{g}}$, $\mathbf{r}$ be the coordinate vectors of $g_h, \hat{g}_h, r_h(g_h)$, respectively. 
 By direct calculation, we have 
\begin{align}
    \|g_h\|^2_h=\sum_{j=1}^{N_{\vel}}\vweight_j \mathbf{g}_j^T\massMat \mathbf{g}_j=
\mathbf{g}^T\textrm{diag}(\vweight_1\massMat,\vweight_2\massMat,\dots,\vweight_{N_{\vel}}\massMat)\mathbf{g},
\end{align}
and this gives \eqref{lem:norm:alg.a}. Similarly, 
\begin{subequations}
\begin{align}
\|r_h(g_h)\|_h^2&=\mathbf{r}^T \textrm{diag}(\vweight_1\massMat,\vweight_2\massMat,\dots,\vweight_{N_{\vel}}\massMat)\mathbf{r},\label{lem:norm:alg.c}\\
(r_h(g_h), \hat{g}_h)_h&=\hat{\mathbf{g}}^T \textrm{diag}(\vweight_1\massMat,\vweight_2\massMat,\dots,\vweight_{N_{\vel}}\massMat) \mathbf{r}.\label{lem:norm:alg.d}
\end{align}
\end{subequations}
On the other hand, with $\bI_{N_\x}$ as the $N_{\x}\times N_{\x}$ identity matrix, one has
\begin{equation}
a_h(g_h, \hat{g}_h)-l_h(\hat{g}_h)=\hat{\mathbf{g}}^T\textrm{diag}(\vweight_1 \bI_{N_\x},\dots,\vweight_{N_{\vel}} \bI_{N_\x})(\bA \mathbf{g}-\dataVec),
\label{lem:norm:alg.e}
\end{equation}
    and \eqref{eq:res:gh} therefore leads to 
\begin{equation}
\mathbf{r}=\textrm{diag}(\massMat^{-1},\dots,\massMat^{-1})(\bA \mathbf{g}-\dataVec).
\label{lem:norm:alg.f}
\end{equation}  
By combining \eqref{lem:norm:alg.c} and \eqref{lem:norm:alg.f}, one reaches \eqref{lem:norm:alg.b}.

\subsection{Proof of Proposition \ref{prop:aPostErr}}
\label{ap:prop1}

    First of all, by performing the now standard $L_2$-type stability analysis for the upwind DG method applied to linear transport equations (also see \cite{sheng2021uniform}) and utilizing $(\langle g_h\rangle_h, g_h-\langle g_h\rangle_h)_h=0,\;\forall g_h\in U_h,$ with $\langle g_h\rangle_h:=\sum_{j=1}^{N_v}\vweight_j g_{h,j}$, one can derive the coercivity of the bilinear form $a_{h,\parameter}$, namely \begin{equation}
        \absorp^\star \|g_h\|_h^2\leq a_{h,\parameter}(g_h, g_h),\quad \; \forall g_h\in U_h.
         \end{equation}
  We then proceed as in \cite{haasdonk2017reduced} (i.e., Propositions 2.20, 2.24), with the shorthand notation $e_h(\parameter)=\angflux_{RB}^m(\cdot;\parameter)-\angflux_h(\cdot;\parameter)\in U_h$ and the definition of the residual in \eqref{eq:res:gh}: 
    \begin{align*}
      \absorp^\star \|e_h(\parameter)\|_h^2&\leq a_{h,\parameter}(e_h(\parameter), e_h(\parameter))\\
        &= 
    a_{h,\parameter}(\angflux_{RB}^{\rbdim}(\cdot;\parameter), e_h(\parameter))-l_{h,\parameter}(e_h(\parameter))\\
    &=(r_{h,\parameter}(\angflux_{RB}^\rbdim(\cdot;\parameter)), e_h(\parameter))_h \leq 
    \|r_{h,\parameter}(\angflux_{RB}^\rbdim(\cdot;\parameter))\|_h \|e_h(\parameter)\|_h.
    \end{align*}
    This leads to \eqref{eq:aPostErr}.

\subsection{Proof of Theorem \ref{thm:why:QRp}}
\label{ap:thm1}

The key to the proof lies in that the range of $\bAr_{\parameter}$ is in the range of $\bQ$, due to  \eqref{eq:A:PG:p} as well as  \eqref{eq:QRp:0}-\eqref{eq:QRp:1} in Step 2 of Alg. \ref{alg:pLS:offline}, namely
\begin{equation}
\textrm{Range}(\bAr_{\parameter})\subset
\textrm{Range}(\bB)=
\textrm{Range}(\bQ),
\end{equation}
where $\bB$ is defined in Step 2 of Alg. \ref{alg:pLS:offline}.
This implies $\bAr_{\parameter}=\bQ\bQ^T \bAr_\parameter$. This also implies $\rbdim=\textrm{Rank}(\bAr_{\parameter})\leq \textrm{Rank}(\bQ)=s$. From here, one can proceed as follows,
\begin{align*}
    \bAr_{\parameter}&=\bQ\bQ^T\left(\sum_{q=1}^{Q_A} \theta_q^A(\parameter) {\bAr}^q\right)  \quad \textrm{(by \eqref{eq:A:PG:p})}\\
    &=\bQ\left(\sum_{q=1}^{Q_A} \theta_q^A(\parameter) \bQ^T{\bAr}^q\right)=\bQ\Big(\underbrace{\sum_{q=1}^{Q_A} \theta_q^A(\parameter) \bY^q}_{\bY_\parameter}\Big)  \quad \textrm{(by Step 3 of Alg. \ref{alg:pLS:offline})}\\ 
        &=\bQ \tilde{\bQ}_\mu \bR_\mu \quad \textrm{(by \eqref{eq:QRp:2}).}
\end{align*}
Set $\widehat{\bQ}_\parameter=\bQ \tilde{\bQ}_\mu$, and one can verify 
$$\widehat{\bQ}_\parameter^T\widehat{\bQ}_\parameter=
(\bQ \tilde{\bQ}_\mu)^T\bQ \tilde{\bQ}_\mu=\tilde{\bQ}_\mu^T(\bQ^T\bQ) \tilde{\bQ}_\mu= \tilde{\bQ}_\mu^T \tilde{\bQ}_\mu=\bI_{\rbdim\times\rbdim}.$$
The full-rankness of $\bY_\parameter$ can be deduced from $\textrm{Null}(\bY_\parameter)\subset \textrm{Null}(\bAr_\parameter)$, with $\bAr_{\parameter}$ being rank-$\rbdim$. Finally, \eqref{eq:QRp:3} follows naturally.

\subsection{Proof of Theorem \ref{thm:ResC:1}}
\label{ap:thm2}

Using the definitions in \eqref{eq:A:PG:p}, \eqref{eq:QRp:0}, \eqref{eq:ResC:1}
and the Kronecker product $\otimes$, the residual can be rewritten into
\begin{equation} 
\|\bA_{\parameter} \orthorbsmat^{\rbdim}\mathbf{c} - \dataVec_\parameter\|_{\mM_h}=\| {\bAr}_{\parameter} \mathbf{c}-{\bbr}_{\parameter}\|_{\ell_2}=
\left\|\bB\begin{bmatrix}\boldsymbol{\theta}^A(\parameter)\otimes \mathbf{c}\\
-\boldsymbol{\theta}^b(\parameter)
\end{bmatrix}\right\|_{\ell_2}.
\end{equation}
Recall the column-pivoted QR factorization and the notation in \eqref{eq:QRp:1}. Note that $\bP$ is unitary, and left-multiplication by $\bQ$ does not change the 2-norm of a vector, one further gets
    \begin{equation*} 
\| {\bAr}_{\parameter} \mathbf{c}-{\bbr}_{\parameter}\|_{\ell_2}
=\left\|\bB\bP \bP^T\begin{bmatrix}\boldsymbol{\theta}^A(\parameter)\otimes \mathbf{c}\\
-\boldsymbol{\theta}^b(\parameter)
\end{bmatrix}\right\|_{\ell_2}
=\left\|\bQ\bR \bP^T\begin{bmatrix}\boldsymbol{\theta}^A(\parameter)\otimes \mathbf{c}\\
-\boldsymbol{\theta}^b(\parameter)
\end{bmatrix}\right\|_{\ell_2}=\left\|\bR \bP^T\begin{bmatrix}\boldsymbol{\theta}^A(\parameter)\otimes \mathbf{c}\\
-\boldsymbol{\theta}^b(\parameter)
\end{bmatrix}\right\|_{\ell_2}.
    \end{equation*}

\section{A generalized $L_1$ error indicator and variants of implementation strategies}
\label{app:variants}

 \subsection{Enhanced $k$-point $L_1$ error indicator}
\label{app:enhancedL1}

It is observed that the simple $L_1$ error indicator does not always find the best parameter value, especially during the early greedy iterations. Below, a generalized $L_1$ error indicator will be described, referred to as the {\it enhanced $k$-point  $L_1$ error indicator}, 
with $k$ as a hyper-parameter of a small natural number. 
Suppose the set $\mP_{\rbdim}$  of the first $\rbdim$ parameter values are greedily selected. 
Let $\mP_{\rbdim}^{(k)}\subset \mP_{\textrm{train}}\setminus\mP_{\rbdim}$ be the set of  $k$ parameter samples where the top $k$ largest values of $\{\Delta_{\rbdim}^{(L)}(\parameter): \parameter\in \mP_{\textrm{train}}\setminus\mP_{\rbdim}\}$ are attained. The next parameter $\parameter_{\rbdim+1}$ is then selected as 
\begin{equation}
\parameter_{\rbdim+1} = \argmax_{\parameter \in \mP_{\rbdim}^{(k)}}\| \angflux_h(\cdot;\parameter) - \angflux_{RB}^\rbdim(\cdot; \parameter) \|_h.
\end{equation}
Note that $k = 1$ corresponds to using the standard $L_1$ error indicator for parameter greedy selection. With the enhanced $k$-point $L_1$ error indicator ($k>1$), $k$ FOM solves are needed during each greedy iteration and the offline computational cost increases respectively. It is demonstrated in Section \ref{sec:lattice} (with $k=2$, and the offline computational cost doubled) that this enhanced $k$-point $L_1$ error indicator with $k>1$ can greatly improve the robustness of the offline training stage of PG-$L_1$ and hence the resolution of the resulting ROM.

\subsection{Variant of Alg. \ref{alg:pLS:offline} - Alg. \ref{alg:pLS:online}} 
\label{app:variants:1}

First of all, to implement ROM($\rbTrial^{\rbdim};\parameter$) via LS-Petrov-Galerkin projection for many parameter values, one natural way to form the concatenated matrix $\bB$ in \eqref{eq:QRp:0} of Alg. \ref{alg:pLS:offline} is 
\begin{equation}
\bB=[\bAr^1, \bAr^2, \dots, \bAr^{Q_A}],\quad \textrm{with}\;\; Q_B=\rbdim Q_A.
\label{eq:QRp:0a}
 \end{equation}
 That is, $\bB$ only gathers the parameter-independent sub-matrices defining $\bAr_{\parameter}$ in \eqref{eq:A:PG:p}, as in \cite{kressner2024randomized}.  Let us denote the variant of Alg. \ref{alg:pLS:offline} as Alg. \ref{alg:pLS:offline}$^\prime$,  where the only difference comes from using $\bB$ in \eqref{eq:QRp:0a}.
 This variant of the offline stage algorithm together with the online algorithm Alg. \ref{alg:pLS:online} provides a good procedure to efficiently and robustly solve the parametric least-squares problem \eqref{eq:rom:PG:noW} at many parameter values. Theorem \ref{thm:why:QRp} also holds for Alg. \ref{alg:pLS:offline}$^\prime$ - Alg. \ref{alg:pLS:online}, with no modification needed.

\begin{itemize}
\item  This variant, Alg. \ref{alg:pLS:offline}$^\prime$ - Alg. \ref{alg:pLS:online}, works well to implement {PG-$L_1$} when no residual errors are computed (e.g., the training residual errors or the test residual errors). Note that these errors are reported in our numerical experiments mainly to evaluate the proposed methods, and they are not needed for either building ROMs or using ROMs for prediction.   
\item This variant, Alg. \ref{alg:pLS:offline}$^\prime$ - Alg. \ref{alg:pLS:online}, combined with the residual computation in \eqref{eq:Res-Alt1}, can also  work well to implement {PG-Res} when the reduced dimension $\rbdim$ is relatively small or moderate, and the respective reduced order solvers have relatively low resolution/fidelity.

\item On the other hand,  for {PG-Res} with relatively larger reduced dimension $\rbdim$ and higher resolution/fidelity, Alg. \ref{alg:ResC:1} is preferred over the formula \eqref{eq:Res-Alt1} for residual computation to avoid the residual stagnation hence to improve the numerical robustness. Alg. \ref{alg:pLS:offline} - Alg. \ref{alg:pLS:online} team up better with Alg. \ref{alg:ResC:1} for the overall efficiency, without the need to compute the column-pivoted QR factorization for two different $\textbf{B}$ in \eqref{eq:QRp:0} and in \eqref{eq:QRp:0a}. 
\end{itemize}

The last two points are closely related to the variants to be discussed next. They are illustrated numerically in Figure \ref{fig:1d Two-Material Problem Training and Test Errors}-(d) and Figure \ref{fig:1d TWo-Material Problem Spectral Ratio and Max Cond(ROM) and PG-Res alternative strategy}-(c) for the two-material problem in Section \ref{sec:two-material}, and in  Figure \ref{fig:1d Spatially Varying Scattering Training and Test Errors}-(d) and Figure \ref{fig:1d Spatially Varying Scattering Spectral Ratio and Max Cond(ROM) and PG-Res alternative strategy}-(c) for the spatially varying scattering example in Section \ref{sec:spatially varying scattering}.

 \subsection{Variants for residual computation in {\bf PG-Res}}
 \label{app:variants:2}

 Suppose  Alg. \ref{alg:pLS:offline}$^\prime$ - Alg. \ref{alg:pLS:online} are applied to implement ROM($\rbTrial^{\rbdim};\parameter$) via LS-Petrov-Galerkin projection.  Below are two variants to evaluate the residual 
 $\|r_{h,\parameter}(\angflux_{RB}^{\rbdim}(\cdot;\parameter))\|_h=\| {\bAr}_{\parameter} \mathbf{c}-{\bbr}_{\parameter}\|_{\ell_2}$ at $\mathbf{c}=\reducedcoeff^{\rbdim}(\parameter)$. This is the residual-based error indicator $\Delta^{(R)}_{\rbdim}(\parameter)$ in PG-Res, and it is also used to compute the training and test residual errors at $\parameter$.

\begin{itemize}
\item [$\diamond$] {\bf Variant 1:}
 One alternative way to compute $\|r_{h,\parameter}(\angflux_{RB}^{\rbdim}(\cdot;\parameter))\|_h=\|\bAr_{\parameter} \reducedcoeff^{\rbdim}(\parameter)-{\bbr}_{\parameter}\|_{\ell_2}$ is by making good use of some intermediate quantities available from Alg. \ref{alg:pLS:online}, as shown next.

\begin{lemma} The following holds
 \begin{align}
\label{eq:Res-Alt1}
     \|\bAr_{\parameter} \reducedcoeff^{\rbdim}(\parameter)-{\bbr}_{\parameter}\|_{\ell_2}
     &=(\|\bbr_{\parameter}\|_{\ell_2}^2-\|\bd_\parameter\|_{\ell_2}^2)^{1/2}.
      \end{align}
\end{lemma}
\begin{proof} The proof is mostly by direct calculation.
    \begin{align*}
\|\bAr_{\parameter} \reducedcoeff^{\rbdim}(\parameter)-{\bbr}_{\parameter}\|_{\ell_2}^2
   &=\|{\bbr}_{\parameter}\|_{\ell_2}^2-\|\bAr_{\parameter} \reducedcoeff^{\rbdim}(\parameter)\|_{\ell_2}^2\\
     &=\|{\bbr}_{\parameter}\|_{\ell_2}^2-\|\widehat{\bQ}_\parameter\bR_\parameter \reducedcoeff^{\rbdim}(\parameter)\|_{\ell_2}^2\\
     &=\|{\bbr}_{\parameter}\|_{\ell_2}^2-\|\widehat{\bQ}_\parameter\bd_\parameter\|_{\ell_2}^2=\|{\bbr}_{\parameter}\|_{\ell_2}^2-\|\bd_\parameter\|_{\ell_2}^2.
\end{align*}
The first equality is due to the optimality condition of the least-squares solution in \eqref{eq:rom:PG:noW}, namely, 
${\bbr}_{\parameter}-\bAr_{\parameter} \reducedcoeff^{\rbdim}(\parameter)\perp\textrm{Range}(\bAr_{\parameter})$.  The last is due to that $\widehat{\bQ}_\parameter$ does not change the 2-norm of a vector. 
\end{proof}
In \eqref{eq:Res-Alt1}, $\bd_\parameter$ is available from Step 3 of Alg. \ref{alg:pLS:online}, while $\|\bbr_{\parameter}\|_{\ell_2}$ can be computed with an offline-online strategy by utilizing the affine structure of $\bbr_{\parameter}$, namely 
\begin{align}
\|{\bbr}_{\parameter}\|_{\ell_2}^2 = \|\sum_{q=1}^{Q_b} \theta_q^b(\parameter){\bbr}^q\|_{\ell_2}^2=\sum_{q=1}^{Q_b} \sum_{p=1}^{Q_b} \theta_q^b(\parameter)\theta_p^b(\parameter)\underbrace{({\bbr}^q)^T{\bbr}^p}_{\textrm{pre-computed}}.
\end{align}
This alternative strategy is quite efficient, and works well numerically in most cases. However,  it can suffer, in some numerical examples, from the loss of significance as the reduced dimension $\rbdim$ grows. This degradation in robustness can be observed from  Figure \ref{fig:1d TWo-Material Problem Spectral Ratio and Max Cond(ROM) and PG-Res alternative strategy}-(c) and Figure  \ref{fig:1d Spatially Varying Scattering Spectral Ratio and Max Cond(ROM) and PG-Res alternative strategy}-(c), in the form of the  stagnation of the residual errors when $\rbdim\geq 15$. It does not help if one uses $\Delta^{(R)}_{\rbdim}(\parameter)= \big((\|\bbr_{\parameter}\|_{\ell_2}-\|\bd_\parameter\|_{\ell_2})(\|\bbr_{\parameter}\|_{\ell_2}+\|\bd_\parameter\|_{\ell_2})\big)^{1/2}$. It is understandable that the observed stagnation is mild and is unlikely to be a concern in practice. Nevertheless, we find Alg. \ref{alg:pLS:offline} - Alg. \ref{alg:pLS:online} with Alg. \ref{alg:ResC:1} are more robust regardless of the reduced dimension $\rbdim$ being large or small (before the termination of the offline training stage).

\item[$\diamond$]{\bf Variant 2:} Another alternative way to compute $\| {\bAr}_{\parameter} \mathbf{c}-{\bbr}_{\parameter}\|_{\ell_2}$ comes from a reformulation of the reduced residual, derived by following a similar argument as for Theorem \ref{thm:ResC:1}. As a reminder, we are considering Alg. \ref{alg:pLS:offline}$^\prime$-Alg.  \ref{alg:pLS:online}, and $\textbf{B}$ is defined in \eqref{eq:QRp:0a} in Alg. \ref{alg:pLS:offline}$^\prime$ with its column-pivoted QR factorization in \eqref{eq:QRp:1}. We also introduce $\hat{\bB}=[\bbr^1,\bbr^2,\dots,\bbr^{Q_b}]$:
\begin{align} 
\| {\bAr}_{\parameter} \mathbf{c}-{\bbr}_{\parameter}\|_{\ell_2}^2
&=\left\|\bB \left(\boldsymbol{\theta}^A(\parameter)\otimes \mathbf{c}\right)-\hat{\bB}\boldsymbol{\theta}^b(\parameter)\right\|_{\ell_2}^2=\left\|\bB\bP\bP^{T} \left(\boldsymbol{\theta}^A(\parameter)\otimes \mathbf{c}\right)-\hat{\bB}\boldsymbol{\theta}^b(\parameter)\right\|_{\ell_2}\notag\\
&=\left\|\bQ\bR \bP^T \left(\boldsymbol{\theta}^A(\parameter)\otimes \mathbf{c}\right)-\bQ\bQ^T\hat{\bB}\boldsymbol{\theta}^b(\parameter)\right\|^2_{\ell_2}+\left\|(\mathbf{I}-\bQ\bQ^T)\hat{\bB}\boldsymbol{\theta}^b(\parameter)\right\|_{\ell_2}^2\notag\\
&=\left\|\underbrace{\bR \bP^T}_{s\times (\rbdim Q_A)} \left(\boldsymbol{\theta}^A(\parameter)\otimes \mathbf{c}\right)-\underbrace{\bQ^T\hat{\bB}}_{s\times Q_b}\boldsymbol{\theta}^b(\parameter)\right\|^2_{\ell_2}+\left\|\underbrace{(\mathbf{I}-\bQ\bQ^T)\hat{\bB}}_{\mN\times Q_b}\boldsymbol{\theta}^b(\parameter)\right\|_{\ell_2}^2.
\label{eq:Res-Alt2}
\end{align}
This  reformulation of the residual into their orthogonal components in \eqref{eq:Res-Alt2} will address the potential stagnation of the computed residual  and improve the robustness of the reduced model (also see \cite{chen2019robust}), and an offline-online strategy is readily available with the parameter-independent matrices $\bR \bP^T, \bQ^T\hat{\bB}, (\mathbf{I}-\bQ\bQ^T)\hat{\bB}$ pre-computed. \eqref{eq:Res-Alt2} is however not adopted for residual computation in this work, due to that the size of $(\mathbf{I}-\bQ\bQ^T)\hat{\bB}$ depends on the full dimension $\mN$, hence the overall efficiency will be compromised. 
\end{itemize}
 \subsection{Variant for residual computation in {\bf G-Res}}
 \label{app:variant:3}
 
For G-Res, one variant is related to computing the residual $\|\bA_{\parameter} \orthorbsmat^{\rbdim}\mathbf{c} - \dataVec_\parameter\|_{\mM_h}=\| {\bAr}_{\parameter} \mathbf{c}-{\bbr}_{\parameter}\|_{\ell_2}$ at $\mathbf{c}=\reducedcoeff^{\rbdim}(\parameter)$. This variant is based on 
    $\|\bA_{\parameter} \orthorbsmat^{\rbdim}\mathbf{c} - \dataVec_\parameter\|^2_{\mM_h}=(\bA_{\parameter} \orthorbsmat^{\rbdim}\mathbf{c} - \dataVec_\parameter)^T\mM_h(\bA_{\parameter} \orthorbsmat^{\rbdim}\mathbf{c} - \dataVec_\parameter)$
    and an offline-online procedure ensured by the parameter-separability assumption on $\bA_{\parameter}, \dataVec_\parameter$. This  works well in many cases, yet it is not recommended as a general strategy, as the conditioning of the task can be worsened due to $\textrm{cond}_2(B^TB)=(\textrm{cond}_2(B))^2$.

\section{Some details on computational complexity}
\label{app:CC}

\subsection{Computational complexity of Alg. \ref{alg:pG}} 
\label{app:CC:alg2}
Due to the sparsity of $\Upwind$ and $\{\absorpmat^q\}_{q=1}^{Q_a}$ as well as {the special structure of $\{\scatmat^q\}_{q=1}^{Q_s}$ as  shown in \eqref{eq:sys:scat},}
the cost to pre-compute $\{{\bArr}^q\}_{q=1}^{Q_A}$, $\{{\bbrr}^q\}_{q=1}^{Q_b}$ in Alg. \ref{alg:pG} is $O(\mN \rbdim^2 Q_A) + O(\mN \rbdim Q_b)$.
At a given $\parameter$, the cost to compute ${\bArr}_\parameter$ and ${\bbrr}_{\parameter}$ is $O(\rbdim^2 Q_A) + O(\rbdim Q_b)$, and solving \eqref{eq:rom:G:sol}, e.g., by a direct method, will be at a cost of $O(\rbdim^3)$. Therefore, the computational complexity of Alg. \ref{alg:pG} is \eqref{eq:CC:ROM-G} to solve the Galerkin-based ROM \eqref{eq:rom:alg:p:G} (or \eqref{eq:rom:G:sol}) for $N_\parameter$ parameter values in $\mP_{\textrm{interest}}$.

\subsection{Computational complexity of  Alg. \ref{alg:pLS:offline}-Alg. \ref{alg:pLS:online}}
\label{app:CC:alg34}

\noindent $\diamond$
{\bf Alg. \ref{alg:pLS:offline}:}  Due to the sparsity of $\mG$, $\Upwind$ and $\{\absorpmat^q\}_{q=1}^{Q_a}$ as well as the special structure of $\{\scatmat^q\}_{q=1}^{Q_s}$ as in \eqref{eq:sys:scat}, the cost to pre-compute  $\{\bAr^q\}_{q=1}^{Q_A}$, $\{\bbr^q\}_{q=1}^{Q_b}$ in Step 1 of  Alg. \ref{alg:pLS:offline} is $O(\mN \rbdim Q_A) + O(\mN Q_b)$.
Recall that the leading order costs of computing the reduced QR of a matrix with or without the column pivoting are the same, hence the cost of Step 2 is  $O(\mN Q_B^2)$, with $Q_B=\rbdim Q_A+Q_b$. It is easy to see the cost of Step 3 is $O(\mN s \rbdim Q_A) + O(\mN s Q_b)$. With these, the total cost of  Alg. \ref{alg:pLS:offline} is $O\big(\mN(Q_B^2+sQ_B)\big)$, 
which, with $s\leq Q_B$,
is further bounded above by 
$O\big(\mN Q_B^2\big)=O\big(\mN (m^2Q_A^2+Q_b^2)\big). $  

\smallskip
\noindent $\diamond$
{\bf Alg. \ref{alg:pLS:online}:} 
At a given $\parameter$, the costs of Steps 1-4 of Alg. \ref{alg:pLS:online} are  $O(s \rbdim Q_A)$,  $O(s \rbdim^2)$, $ O(s(\rbdim + Q_b))$, $ O(\rbdim^2)$, respectively. Therefore the total cost of Alg. \ref{alg:pLS:online} applied to $N_\parameter$ parameter values is 
$O\big(N_{\parameter}(s(\rbdim Q_A+ \rbdim^2+\rbdim + Q_b) +\rbdim^2)\big)$,
which, with  $s\leq Q_B$, is bounded above by 
$O\big(N_{\parameter} (\rbdim^3 Q_A+\rbdim^2 (Q_A^2+Q_b)+Q_b^2\big).$

By combining the costs above, one comes up with the  computational complexity as in \eqref{eq:CC:ROM-PG} to solve the LS-Petrov-Galerkin-based ROM  \eqref{eq:rom:alg:p:PG} (or \eqref{eq:rom:PG:noW}) for $N_\parameter$ parameter values in $\mP_{\textrm{interest}}$.

\end{appendices}

\bibliographystyle{plain}
\bibliography{references.bib, references1.bib}

\end{document}